%% file: File_unico_ArXiv.tex
\documentclass[11pt,a4paper]{amsart}
\usepackage[utf8]{inputenc}
\usepackage{comment}
\usepackage{amsmath,amsthm,amssymb,amscd}
\usepackage{mathtools}
\usepackage{graphicx}
\usepackage{cancel} 
\usepackage[dvipsnames]{xcolor}
\usepackage{enumitem}
\usepackage{mathrsfs}
\usepackage{tikz}
\usetikzlibrary {arrows.meta}
\usepackage{adjustbox}
\usepackage[percent]{overpic}
\usepackage[margin=2.5cm,
  marginparwidth=2.4cm, 
  marginparsep=0pt 
]{geometry}
\usepackage[style=numeric-comp,sorting=nyvt,backend=biber,doi=true,url=false]{biblatex} 
\usepackage{needspace}
\usepackage{booktabs}
\usepackage{longtable}
\usepackage{array}
\usepackage{hyperref}
\usepackage{caption}
\usepackage{subcaption}
\usepackage{array} 
\newcolumntype{C}[1]{>{\centering\arraybackslash}m{#1}}

\newtheorem{thm}{Theorem}[section]
\newtheorem{prop}[thm]{Proposition}
\newtheorem{lem}[thm]{Lemma}
\newtheorem{cor}[thm]{Corollary}

\newtheorem{defn}[thm]{Definition}

\theoremstyle{definition}
\newtheorem{rem}[thm]{Remark}
\newtheorem{rems}[thm]{Remarks}
 
\newtheorem*{rmk}{Remark}

\newcommand{\bbR}{\mathbb R}
\newcommand{\bbZ}{\mathbb Z}
\newcommand{\calS}{\mathcal{S}}
\newcommand{\bbQ}{\mathbb Q}

\newcommand{\de}{\delta}
\newcommand{\si}{\sigma}

\newcommand{\ep}{\varepsilon}
\newcommand{\Ga}{\Gamma}
\newcommand{\De}{\Delta}

\DeclareMathOperator{\rot}{rot}
\DeclareMathOperator{\sgn}{sign}
\DeclareMathOperator{\wt}{wt}

\newcommand{\sid}{\mathcal{D}}
\newcommand{\fsid}{\mathcal{D}^\mathrm{free}}
\newcommand{\nfsid}{\mathcal{D}^\mathrm{inv}}
\newcommand{\gsid}{\mathcal{D}^\mathrm{good}}

\newcommand{\matsym}[1]{\mathord{\vcenter{\hbox{\adjustbox{valign=c}{\includegraphics[scale=0.25]{#1}}}}}}

\newcommand{\acirc}[1]{\raisebox{-4pt}{\begin{tikzpicture}[scale = .75]
\draw[red, very thin] (-.5,0) -- (.5,0);
\draw[thick] (0,0) circle (.3);
\node[right] at (0,.45) {#1};
\end{tikzpicture}}}
 
\newcommand{\acircs}[1]{\raisebox{-14pt}{\begin{tikzpicture}[scale = .75]
\draw[red, very thin] (-.5,0) -- (.5,0);
\draw[thick] (0,.4) circle (.2);
\node[right] at (.1,.6) {#1};
\draw[thick] (0,-.4) circle (.2);
\node[right] at (.1,-.6) {#1};
\end{tikzpicture}}}

\newcommand{\unknot}{\matsym{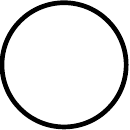}}

\newcommand{\doublecrossingoffa}{\matsym{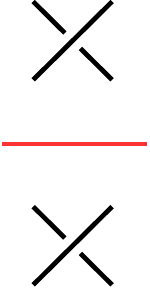}}
\newcommand{\doublecrossingoffb}{\matsym{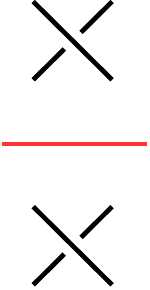}}
\newcommand{\doubleoneres}{\matsym{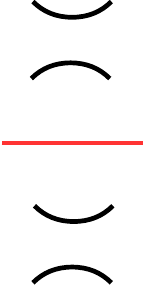}}
\newcommand{\doublezerores}{\matsym{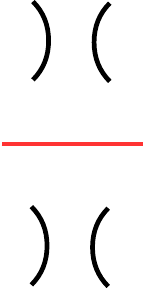}}
\newcommand{\crossingonaxisa}{\matsym{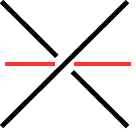}}
\newcommand{\crossingonaxisb}{\matsym{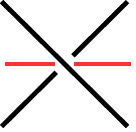}}
\newcommand{\zeroresonaxis}{\matsym{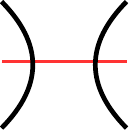}}
\newcommand{\oneresonaxis}{\matsym{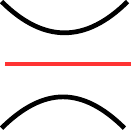}}
\newcommand{\poscrossing}{\matsym{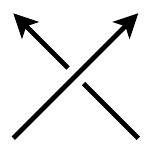}}
\newcommand{\negcrossing}{\matsym{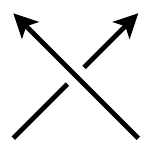}}
\newcommand{\orientedresol}{\matsym{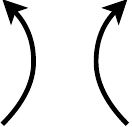}}
\newcommand{\poscrossingonaxis}{\matsym{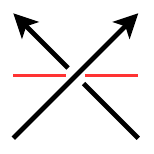}}
\newcommand{\negcrossingonaxis}{\matsym{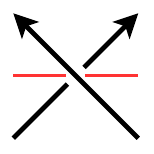}}
\newcommand{\orientedresolutiononaxis}{\matsym{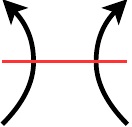}}

\newcommand{\edgeresolutiononaxis}{\matsym{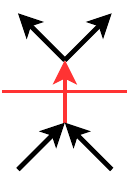}}
\newcommand{\poscrossingoffaxis}{\matsym{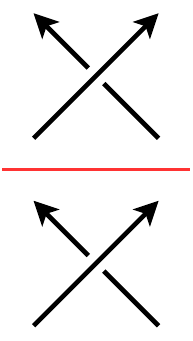}}
\newcommand{\negcrossingoffaxis}{\matsym{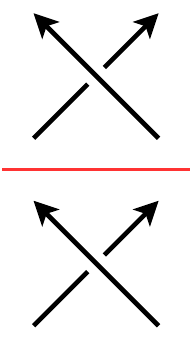}}
\newcommand{\claspRtwo}{\matsym{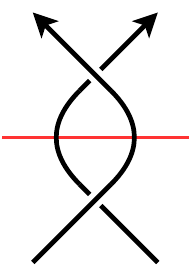}}
\newcommand{\claspresol}{\matsym{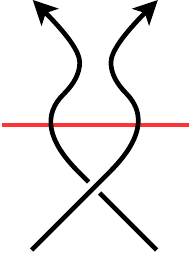}}
\newcommand{\edgebubbleoffaxis}{\matsym{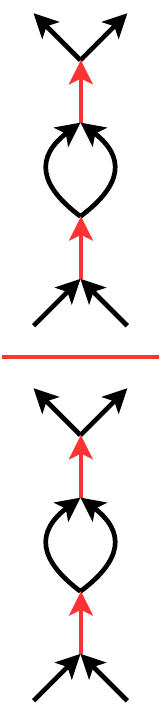}}
\newcommand{\orientedresolutionoffaxis}{\matsym{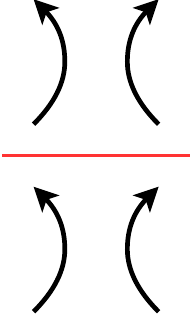}}
\newcommand{\Claspa}{\matsym{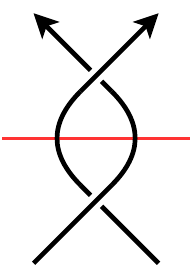}}

\newcommand{\edgeresolutionoffaxis}{\matsym{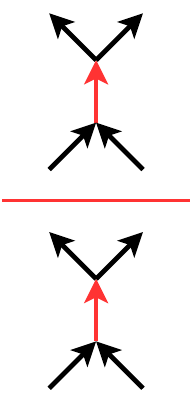}}
\newcommand{\IRoneposloop}{\matsym{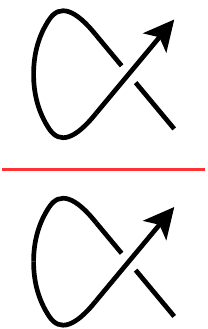}}
\newcommand{\IRonenegloop}{\matsym{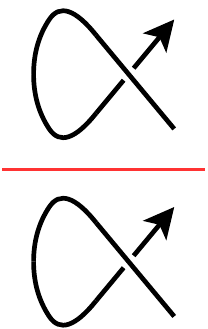}}
\newcommand{\IRtwotopleft}{\matsym{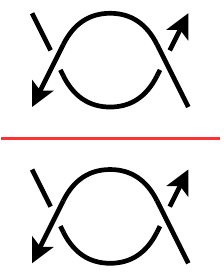}}
\newcommand{\IRtwotopright}{\matsym{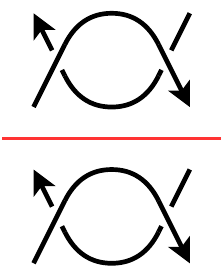}}
\newcommand{\IRtworesoo}{\matsym{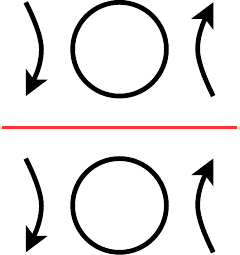}}
\newcommand{\IRtworesoob}{\matsym{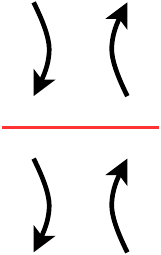}}
\newcommand{\IRtworesoi}{\matsym{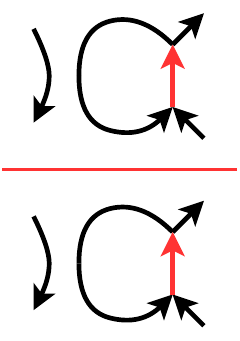}}
\newcommand{\IRtworesio}{\matsym{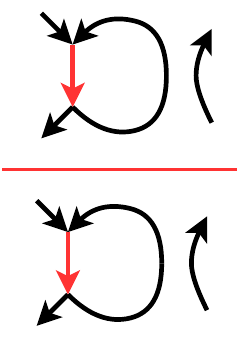}}
\newcommand{\IRtworesii}{\matsym{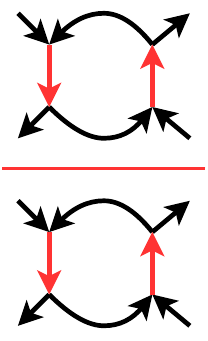}}
\newcommand{\IRtwobottomleft}{\matsym{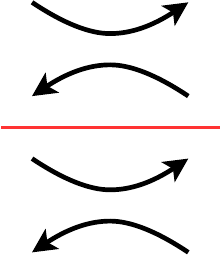}}
\newcommand{\IRtwobottomright}{\matsym{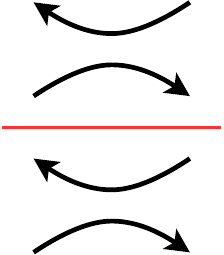}}
\newcommand{\IRthreeonethicka}{\matsym{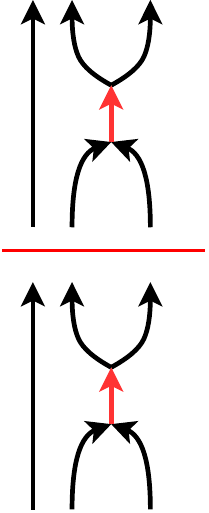}}
\newcommand{\IRthreethreethicka}{\matsym{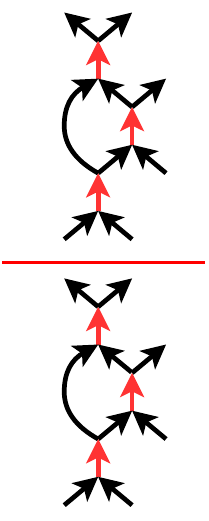}}
\newcommand{\IRthreeonethickb}{\matsym{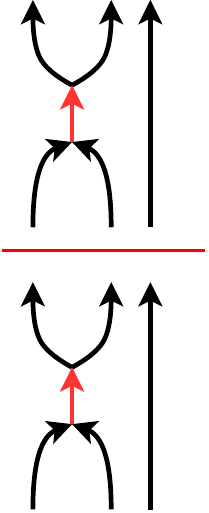}}
\newcommand{\IRthreethreethickb}{\matsym{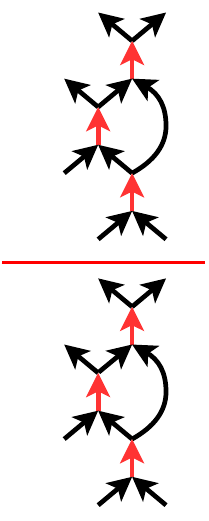}}
\newcommand{\IRthreeresooo}{\matsym{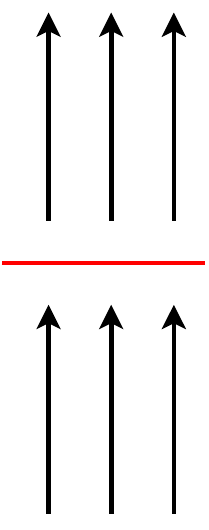}}
\newcommand{\IRthreeresioo}{\matsym{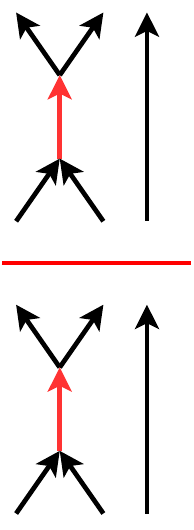}}
\newcommand{\IRthreeresoio}{\matsym{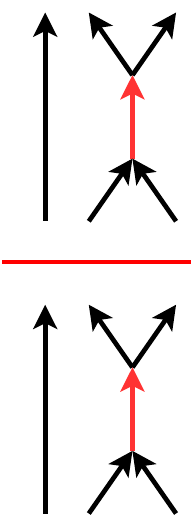}}
\newcommand{\IRthreeresiio}{\matsym{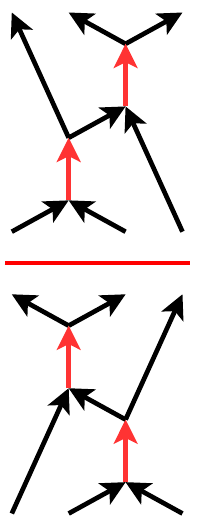}}
\newcommand{\IRthreeresoii}{\matsym{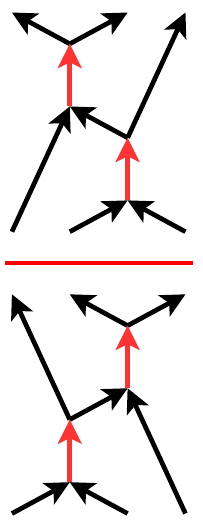}}
\newcommand{\IRthreeresiiia}{\matsym{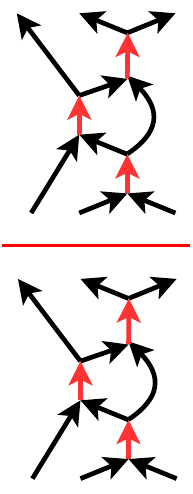}}
\newcommand{\IRthreeresiiib}{\matsym{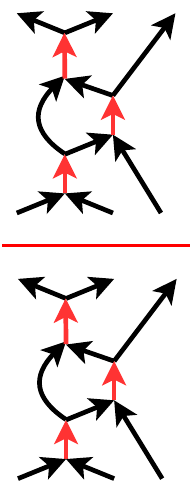}}
\newcommand{\Roneposloop}{\matsym{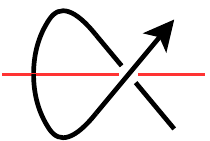}}
\newcommand{\Ronenegloop}{\matsym{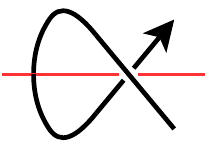}}
\newcommand{\Roneposloopb}{\matsym{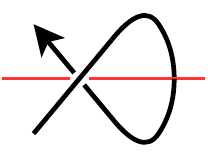}}
\newcommand{\Ronenegloopb}{\matsym{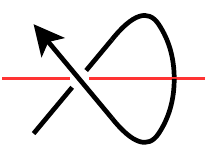}}
\newcommand{\Roneorientedres}{\matsym{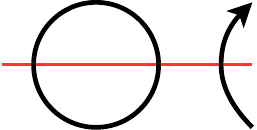}}
\newcommand{\arcuponaxis}{\matsym{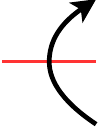}}
\newcommand{\arcuponaxisb}{\matsym{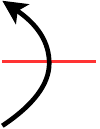}}
\newcommand{\Roneedgeresol}{\matsym{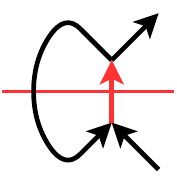}}
\newcommand{\Roneedgeresolb}{\matsym{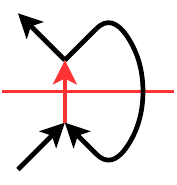}}
\newcommand{\Soneedgeresol}{\matsym{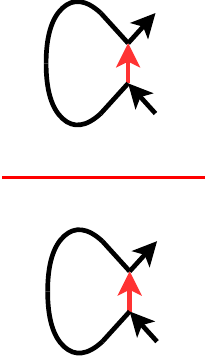}}
\newcommand{\Mtwocrossingsneg}{\matsym{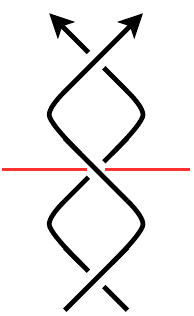}}
\newcommand{\Mtwocrossingspos}{\matsym{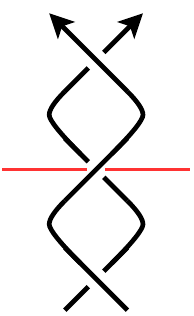}}
\newcommand{\Mtwoedgeresolutions}{\matsym{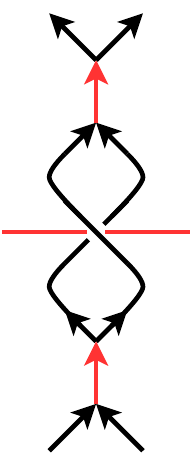}}
\newcommand{\twoedgesbubbleaxis}{\matsym{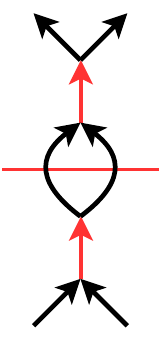}}
\newcommand{\threeedgesaxisbubbles}{\matsym{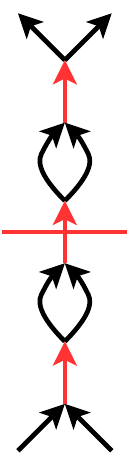}}
\newcommand{\Rtwocrossings}{\matsym{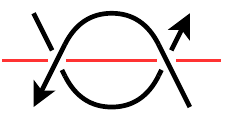}}
\newcommand{\Rtwonocrossings}{\matsym{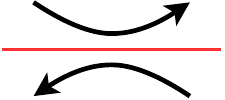}}
\newcommand{\Rtwocrossingsmirror}{\matsym{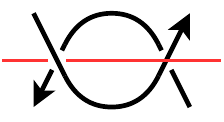}}
\newcommand{\Rtwoorientedresolution}{\matsym{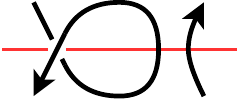}}
\newcommand{\Rtwoedgeresolution}{\matsym{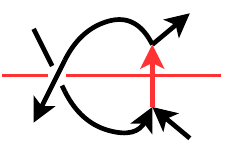}}
\newcommand{\Rtwodoubleresolution}{\matsym{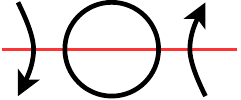}}
\newcommand{\Rtwodoubleedgeresolution}{\matsym{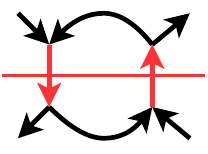}}
\newcommand{\Rtwoedgeorientedresolution}{\matsym{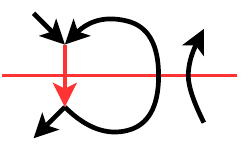}}
\newcommand{\Rtwoorientededgeresolution}{\matsym{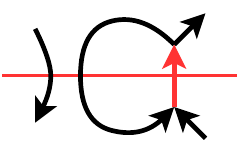}}

\newcommand{\arcsdownuponaxis}{\matsym{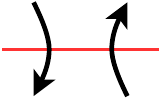}}

\newcommand{\arcsoffaxis}{\matsym{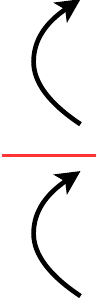}}
\newcommand{\Moneedgeondoubleedgeoff}{\matsym{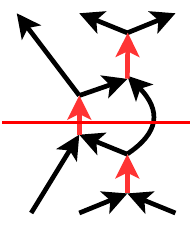}}
\newcommand{\Monedoubleedgeoffedgeon}{\matsym{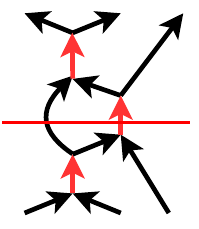}}
\newcommand{\arcedgeon}{\matsym{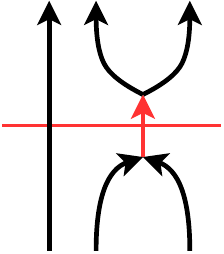}}
\newcommand{\edgeonarc}{\matsym{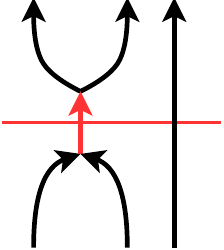}}
\newcommand{\LemmastatoneA}{\matsym{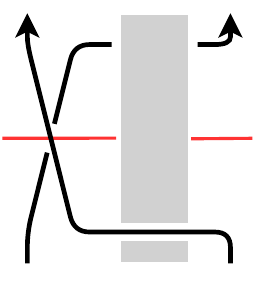}}
\newcommand{\LemmastatoneB}{\matsym{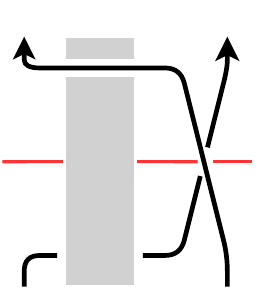}}
\newcommand{\LemmastattwoA}{\matsym{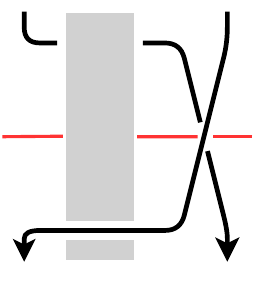}}
\newcommand{\LemmastattwoB}{\matsym{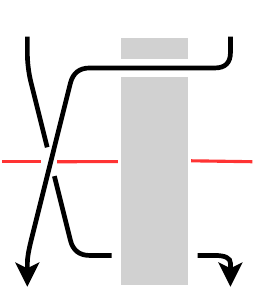}}
\newcommand{\LemmaproofA}{\matsym{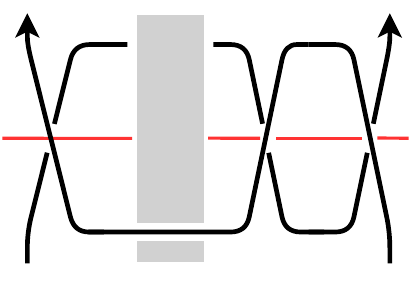}}
\newcommand{\LemmaproofB}{\matsym{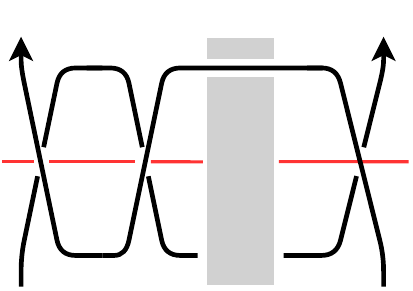}}
\newcommand{\Monecrossonfirst}{\matsym{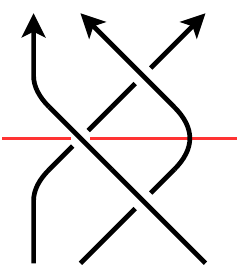}}
\newcommand{\Monecrossonsecond}{\matsym{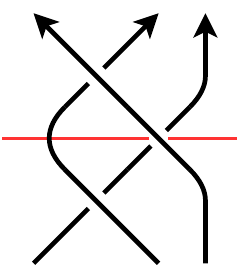}}
\newcommand{\Moneposcrossonfirst}{\matsym{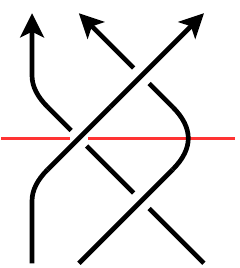}}
\newcommand{\Moneposcrossonsecond}{\matsym{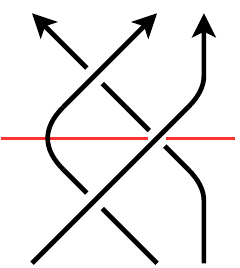}}
\newcommand{\Monecrossonfirstb}{\matsym{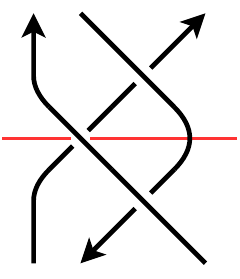}}
\newcommand{\Monecrossonsecondb}{\matsym{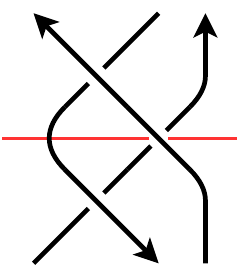}}
\newcommand{\Moneposcrossonfirstb}{\matsym{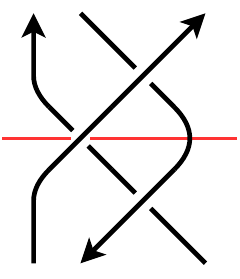}}
\newcommand{\Moneposcrossonsecondb}{\matsym{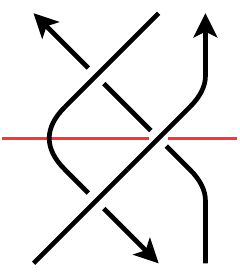}}

\newcommand{\Moneallresolvedfirst}{\matsym{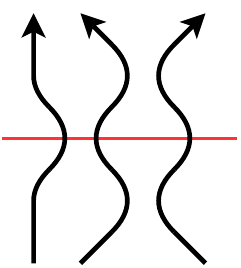}}
\newcommand{\Moneallresolvedsecond}{\matsym{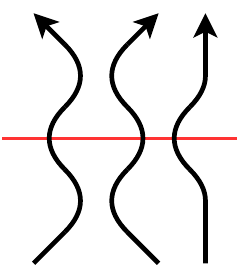}}

\newcommand{\Mthreeedgeonfirst}{\matsym{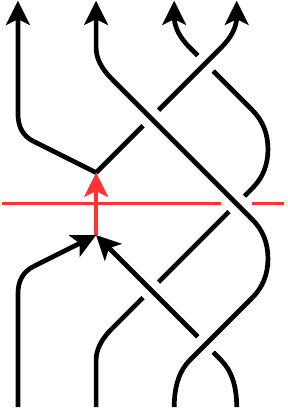}}
\newcommand{\Mthreeedgeonsecond}{\matsym{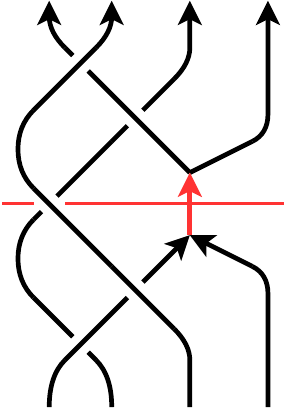}}
\newcommand{\Mthreeedgeonfirsta}{\matsym{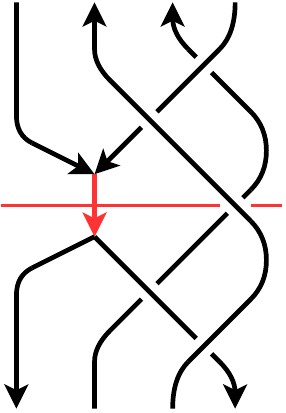}}
\newcommand{\Mthreeedgeonseconda}{\matsym{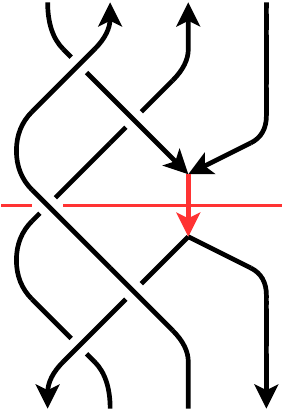}}
\newcommand{\Mthreebubblededgefirst}{\matsym{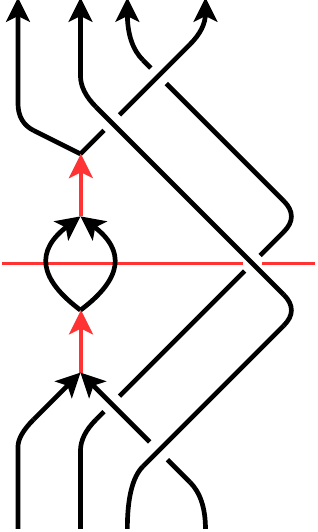}}
\newcommand{\MthreebubblededgeRtwo}{\matsym{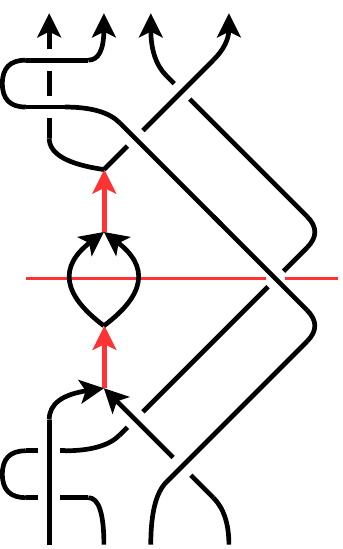}}
\newcommand{\Mthreeapplylemmafirst}{\matsym{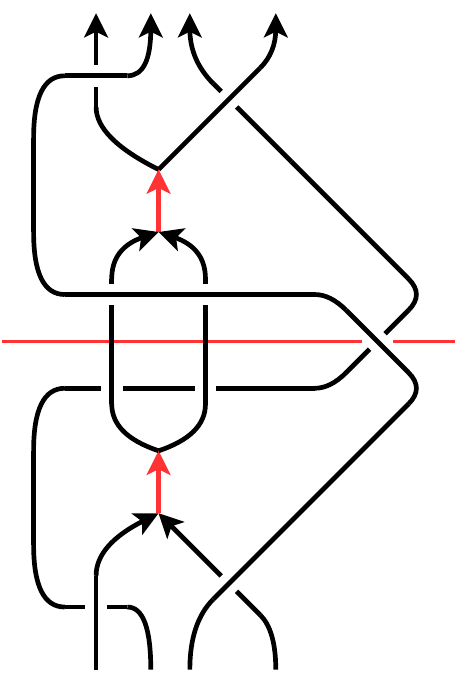}}
\newcommand{\Mthreeedgesymmetric}{\matsym{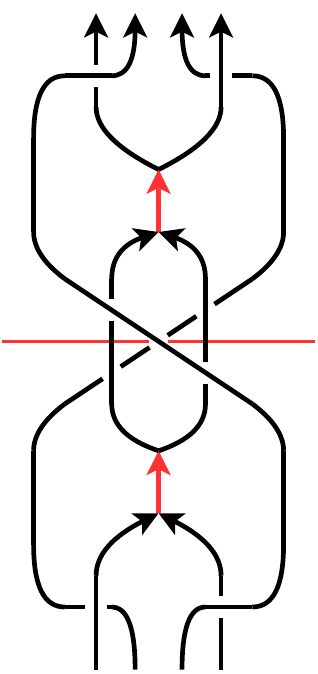}}
\newcommand{\edgeIRthreefirst}{\matsym{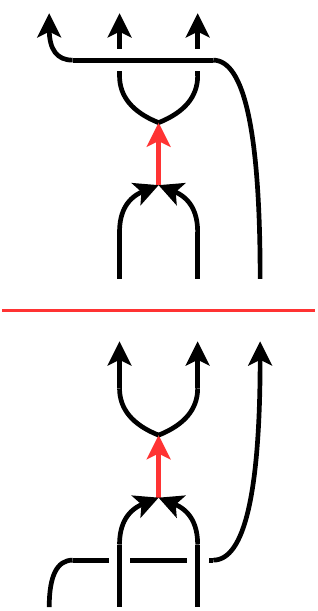}}
\newcommand{\edgeIRthreesecond}{\matsym{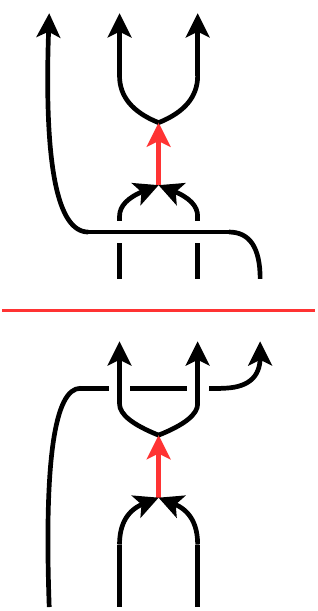}}
\newcommand{\edgeIRthreefirsta}{\matsym{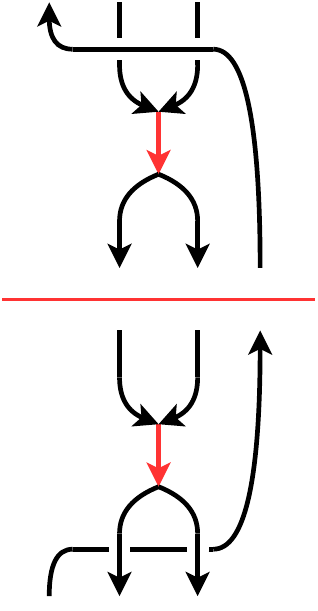}}
\newcommand{\edgeIRthreeseconda}{\matsym{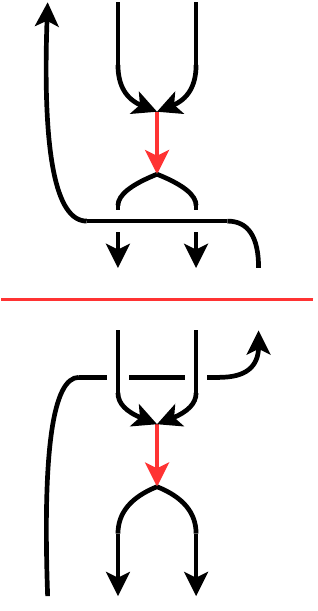}}
\newcommand{\IRthreeone}{\matsym{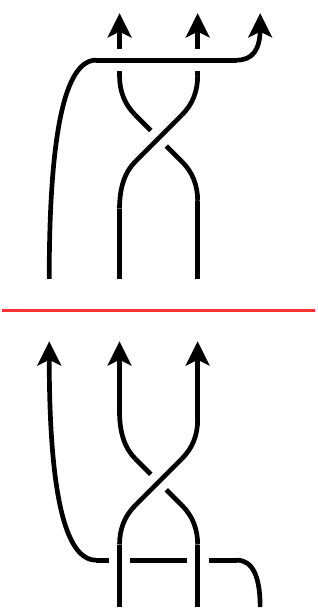}}
\newcommand{\IRthreetwo}{\matsym{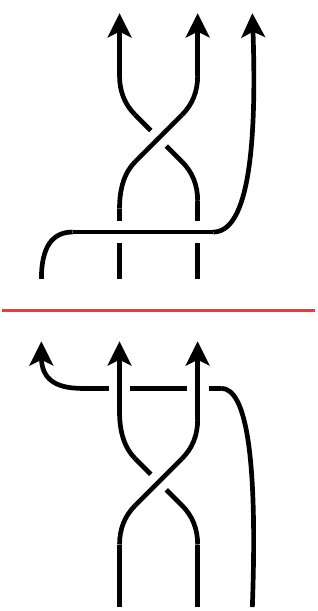}}
\newcommand{\IRthreeoneProof}{\matsym{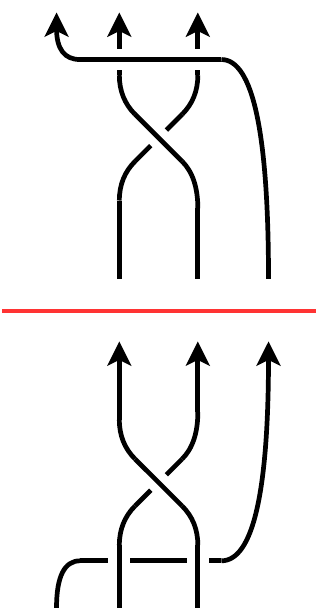}}
\newcommand{\IRthreetwoProof}{\matsym{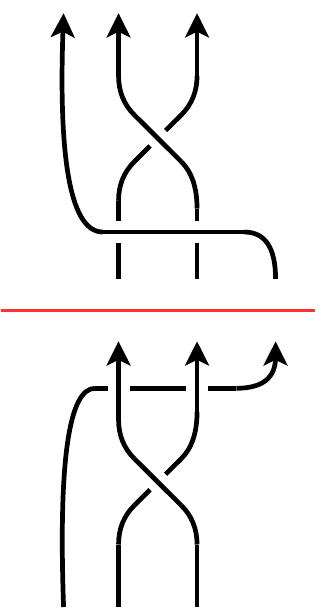}}
\newcommand{\IRthreeres}{\matsym{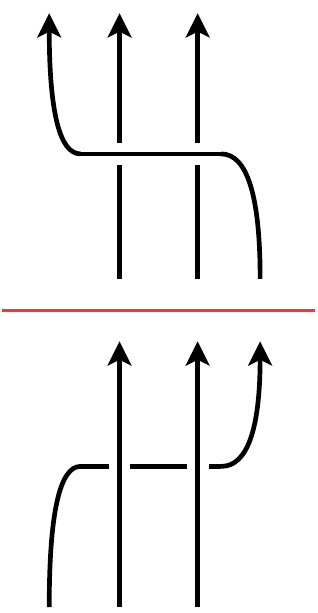}}
\newcommand{\Mthreefirst}{\matsym{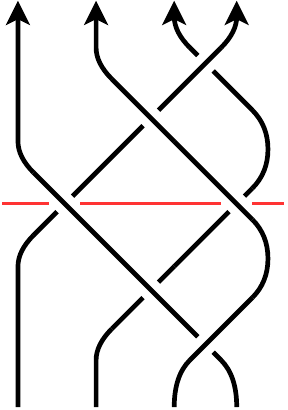}}
\newcommand{\Mthreesecond}{\matsym{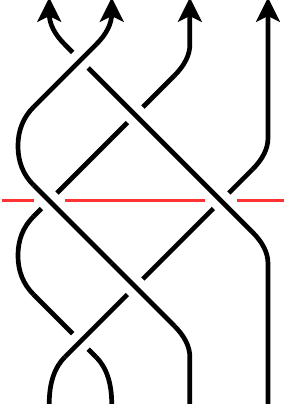}}
\newcommand{\Mthreefirsta}{\matsym{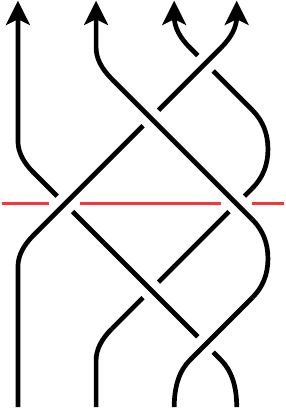}}
\newcommand{\Mthreeseconda}{\matsym{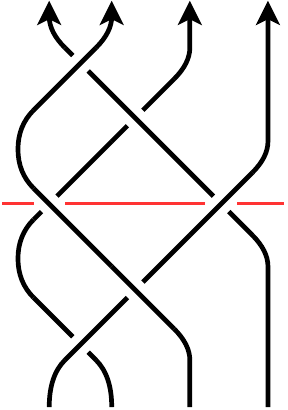}}
\newcommand{\Mthreefirstb}{\matsym{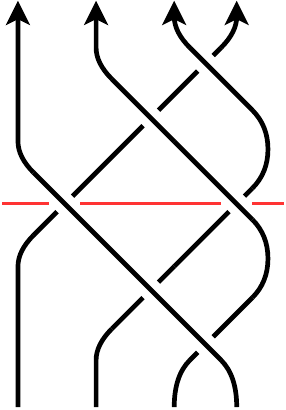}}
\newcommand{\Mthreebproof}{\matsym{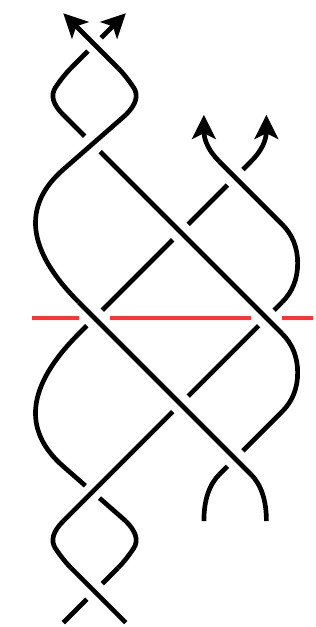}}
\newcommand{\Mthreesecondb}{\matsym{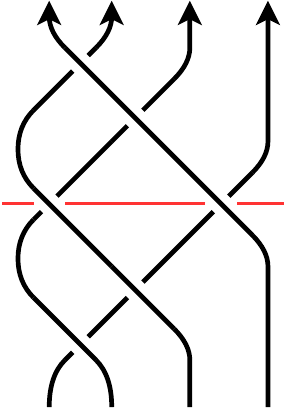}}

\newcommand{\Mthreefirstd}{\matsym{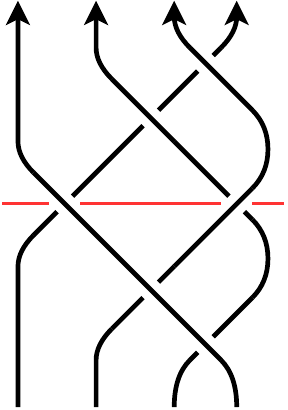}}
\newcommand{\Mthreesecondd}{\matsym{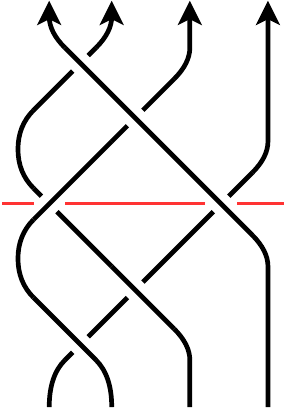}}
\newcommand{\Mthreefirstrev}{\matsym{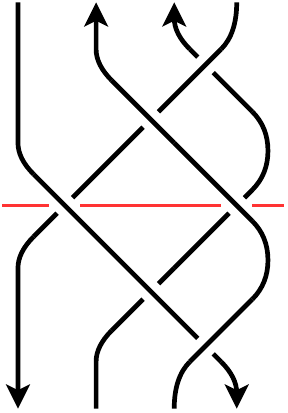}}
\newcommand{\Mthreesecondrev}{\matsym{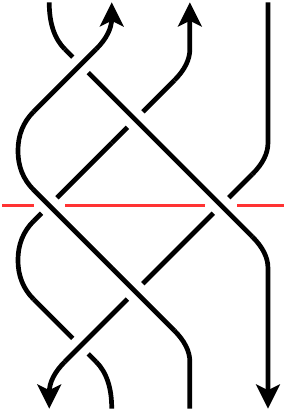}}
\newcommand{\Mthreefirstarev}{\matsym{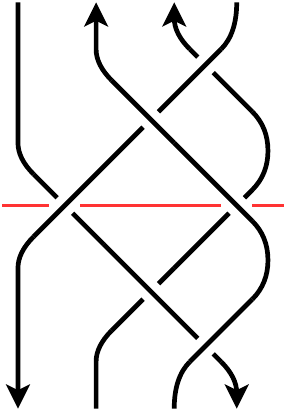}}
\newcommand{\Mthreesecondarev}{\matsym{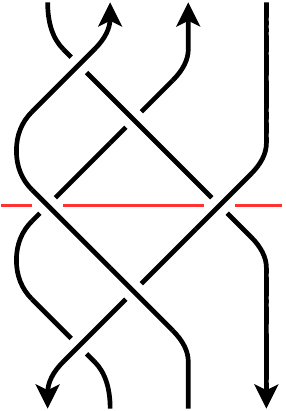}}
\newcommand{\Mthreefirstbrev}{\matsym{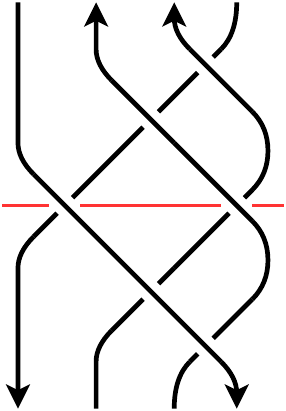}}
\newcommand{\Mthreesecondbrev}{\matsym{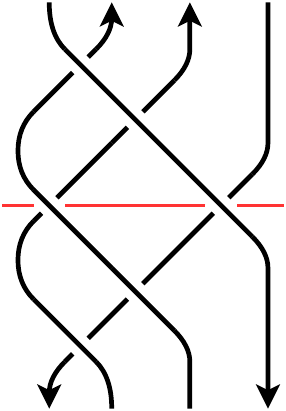}}

\newcommand{\Mthreefirstdrev}{\matsym{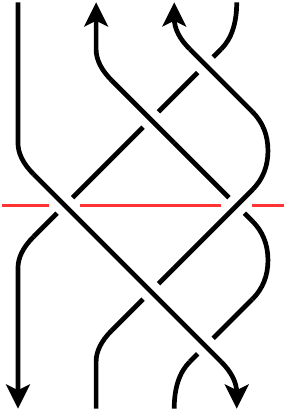}}
\newcommand{\Mthreeseconddrev}{\matsym{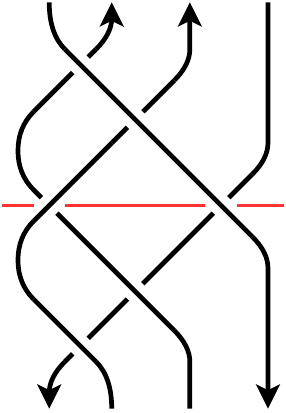}}
\newcommand{\Mthreeresonfirst}{\matsym{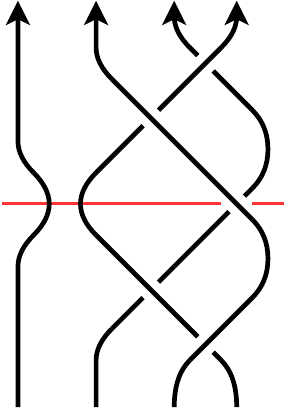}}
\newcommand{\Mthreeresonsecond}{\matsym{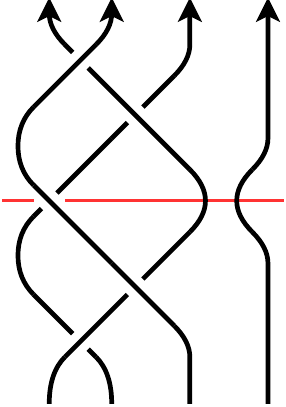}}
\newcommand{\Mthreereson}{\matsym{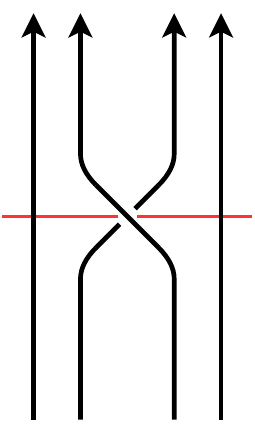}}
\newcommand{\Squaredoubleedgesoffaxis}{\matsym{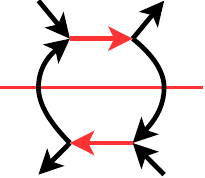}}

\newcommand{\Addresses}{ 
\bigskip
\footnotesize
\noindent\textsc{Dipartimento di Matematica e Informatica ``U. Dini'', Viale Morgagni 67/A, 50134 Firenze, Italy}\par\nopagebreak
\textit{E-mail address:}\
\texttt{carlo.collari@unifi.it}\par\nopagebreak~\\
\textsc{Dipartimento di Matematica, Largo Bruno Pontecorvo 5, 56127 Pisa, Italy}\par\nopagebreak
\textit{E-mail address:}\
\texttt{paolo.lisca@unipi.it}
}

\title{A Polynomial Invariant of Strongly Involutive Links}
\author{Carlo Collari and Paolo Lisca}
\date{\today}

\begin{document}

\begin{abstract}
We introduce a new two-variable polynomial invariant \(P^e\) of strongly involutive links,
uniquely characterized by equivariant skein relations and naturally viewed as an
equivariant analogue of the HOMFLY--PT polynomial. We prove that a specialization of
\(P^e\) recovers the graded Euler characteristic of the third page of the Lobb--Watson
\(\mathcal{G}\)-filtration spectral sequence, generalizing Couture's polynomial
invariant. We further show that, after a change of variables, \(P^e\) reduces modulo
\(2\) to the HOMFLY--PT polynomial, up to an explicit power of the skein variable,
thereby answering a generalized form of a question due to Couture. We use the resulting
skein relations to distinguish infinitely many pairs of alternating mutant knots, and
show that \(P^e\) is strictly stronger than the refined Lobb--Watson invariants on
infinitely many strongly invertible knots.
\end{abstract}

\maketitle
\section{Introduction}

\subsection*{Background and motivation} 

Polynomial invariants of knots and links play a central role in low-dimensional topology, providing computable algebraic tools for distinguishing links, detecting subtle geometric properties, and organizing large classes of examples.

Since the discovery of the Jones polynomial \cite{Jo85, Ka87}, a wide family of related invariants
has been developed, including the HOMFLY--PT polynomial \cite{FYHLMO85,PT87}
and the $\mathfrak{sl}_N$ polynomials \cite{MOY98}. 
These invariants admit recursive characterizations via skein relations and have deep
connections with representation theory, quantum topology, and categorification.

More recently, attention has shifted toward invariants that are sensitive not only to the
underlying link type, but also to \emph{additional symmetry data}.
In such situations, classical polynomial invariants fail to detect the presence or nature
of the symmetry, motivating the development of \emph{equivariant} refinements.

Strongly involutive links provide a natural framework for studying knot invariants in the presence
of symmetry. A \emph{strongly involutive link} consists of a link $L \subset S^3$ together with an
orientation-preserving involution
\[
\tau \colon S^3 \to S^3
\]
whose fixed-point set is an unknotted circle meeting $L$ transversely, and such that $\tau$
restricts to an orientation--reversing involution of $L$.
This framework encompasses strongly invertible knots, links with free involutions, and mixed
situations where some components are fixed while others are exchanged.
The topology of such links encodes additional structure that is invisible to ordinary link
invariants.

While much of the classical literature focuses on the special case of strongly
invertible knots and links \cite{Sa86, Co08, Co09, La22, HHS23}, more recent work has developed
equivariant homological and polynomial invariants in the broader involutive setting
\cite{LW21,Wa17,DMS23,BI22,MP23,Di23,Di24,DF25}.

Early approaches to invariants of strongly involutive links include the theory of signed divides developed by Couture, which produced a polynomial invariant $W$ for strongly invertible knots via a combinatorial state sum~\cite{Co08}, as well as its categorification  \cite{Co09}.
More recently, Lobb and Watson introduced a refinement of Khovanov homology for links with
involution \cite{LW21}.
In the strongly involutive setting, their construction yields a triply graded homology theory, as well as various spectral sequences whose pages are themselves invariants.
Of particular interest is the third page $E^{3}_{\mathcal{G}}$ of the
$\mathcal{G}$--filtration spectral sequence which, by \cite[Remark~4.6]{LW21}, recovers Couture's homology theory from~\cite{Co09}. 
Consequently its graded Euler characteristic~$J^e(L)$ recovers, for strongly invertible links, the polynomial invariant $W$ previously defined by Couture.

The polynomial $J^e(L)$ is, in this sense, the equivariant analogue of the Jones polynomial, and one therefore 
expects it to fit into a broader skein-theoretic picture analogous to the role played by the HOMFLY--PT 
polynomial in the classical theory.  Couture's work already contains skein relations in the language of signed 
divides, and Lobb--Watson's construction suggests that such relations should arise naturally from equivariant 
homological invariants.  What has been missing, however, is a unified framework which both produces equivariant 
$\mathfrak{sl}_N$-type specializations and packages them into a two-variable polynomial invariant.  The aim of 
this paper is to provide such a framework.

\subsection*{Main results}

We define and study a new two-variable polynomial invariant
\[
P^e(L)\in \mathbb Z[a^{\pm1},z^{\pm1}]
\]
of strongly involutive links.  The invariant \(P^e\) plays, in the equivariant setting, the role played by the 
HOMFLY--PT polynomial in the ordinary theory: it is characterized by skein relations, admits 
\(\mathfrak{sl}_N\)-type specializations, and its Jones-type specialization recovers the graded Euler 
characteristic of the Lobb--Watson invariant.

Beyond its formal analogy with the HOMFLY--PT polynomial, the invariant \(P^e\) contains genuinely new 
equivariant information.  In particular, we show that it distinguishes infinitely many pairs of strongly 
invertible knots which have identical refined Lobb--Watson invariants.  
Thus \(P^e\) is not merely a repackaging of previously known equivariant data.

Throughout the paper, we work with transvergent strongly involutive diagrams, with horizontal fixed-point axis; 
this entails no loss of generality, as any strongly involutive link admits such a representative up 
to equivariant isotopy (cf.~\cite{LW21}). 

Given a strongly involutive link $L$, we denote by $\ell(L)$ the number of connected 
components with fixed points. We construct an invariant $P^e(L)\in\bbZ[a^{\pm 1},z^{\pm 1}]$
and show that it is uniquely determined by the following properties: 

\begin{equation}\label{eq:Pskeinonaxis}
a\, P^e\left(\negcrossingonaxis\right) - a^{-1} P^e\left(\poscrossingonaxis\right)
  = P^e\left(\orientedresolutiononaxis\right),
\end{equation}
\begin{equation} \label{eq:Pskeinoffaxis}
a^2 P^e\left(\negcrossingoffaxis\right) - a^{-2} P^e\left(\poscrossingoffaxis\right)
  = z\, P^e\left(\orientedresolutionoffaxis\right),
\end{equation}
\begin{equation}\label{eq:Pskeinincoerentclasp}
P^e\left(\Claspa\right)
  = a\, z\, P^e\left(\poscrossingonaxis\right) - a^2 P^e\left(\orientedresolutiononaxis\right), 
\end{equation}
\begin{gather}
P^e\left(\acirc{}\right) = 1\quad\text{and}\quad P^e\left(\acircs{}\right) = \frac{a + a^{-1}}{z}, \label{eq:Pskeinnormalisation}\\
P^e(L \sqcup_e L') = (a - a^{-1})\,  P^e(L)\, P^e(L'), \label{eq:Pskeinunion}\\
P^e(L^*)(a,z) = (-1)^{\ell(L)+1} P^e(L)(a^{-1},-z). \label{eq:Pskeinmirror}
\end{gather}
Here by the \emph{equivariant disjoint union} $L \sqcup_e L'$ of two strongly involutive links $L$ and $L'$ we mean their disjoint union equipped with a strong involution such that, upon forgetting either of the two links, the remaining link retains its structure as a strongly involutive link.
Note that there might be inequivalent strongly involutive links satisfying the definition given, nonetheless they share the same $P^e$.
Similarly, for a strongly involutive link $L$, its \emph{equivariant mirror} $L^*$ is defined as the image of $L$ under an orientation-reversing diffeomorphism of $S^3$ that commutes with the involution $\tau$. 

Our main result establishes the existence and uniqueness of the two-variable polynomial
invariant $P^e$, as well as its relationship to $J^e(L)$.  

\begin{thm}\label{t:main}
There exists a unique invariant $P^e(L)\in \bbZ[a^{\pm 1}, z^{\pm 1}]$ of strongly involutive links satisfying 
Properties~\eqref{eq:Pskeinonaxis}--\eqref{eq:Pskeinmirror}. Moreover, 
\[
(q^2 - q^{-2}) P^e(L)|_{a = q^2,\, z = q^2-q^{-2}} = (q - q^{-1})^{\ell(L)} J^e(L)
\]
where $J^e(L)$ is the graded Euler characteristic of $E^{3}_{\mathcal{G}}$, that is the third page 
of the Lobb--Watson $\mathcal{G}$--filtration spectral sequence. 
\end{thm}

The proof of this result proceeds in several stages. We introduce equivariant MOY graphs 
and define an equivariant $N$--bracket, adapting the approach of Murakami, Ohtsuki, and Yamada~\cite{MOY98} 
to the symmetric setting. We then construct equivariant $\mathfrak{sl}_N$ polynomial invariants 
$P^e_N\in\bbZ[q^{\pm 1}]$, establish existence and properties of an auxiliary invariant 
$Q(L)\in\bbZ[a^{\pm 1},z]$ related to the polynomials $P^e_N$ via evaluations, and finally 
we prove Theorem~\ref{t:main}.

As a key step in the proof of Theorem~\ref{t:main}, we show that the graded Euler characteristic of the Lobb--Watson invariant $J^e$ satisfies equivariant skein relations analogous to those defining $P^e(L)$. As an illustration of the applicability of the skein relations, we show that $J^e$ can be used to distinguish infinitely 
many pairs of alternating, mutant knots. 
Recall that mutant knots have the same Jones polynomial and the signature, and in the alternating 
case these invariants uniquely determine the Khovanov homology~\cite{Le05}. Similarly, the Alexander 
polynomial is a knot mutation invariant and together with the signature it determines the 
Heegaard Floer homology of alternating knots~\cite{OS04, Ra03}. Thus, alternating, mutant knots 
cannot be distinguished by Khovanov homology, nor by Heegaard Floer homology. 
Let \(K = K(a_1,\ldots,a_s)\) denote the \(s\)-strand pretzel link, and recall that \(K\) 
is a knot precisely when exactly one of the integers \(a_i\) is even.

\begin{thm}\label{t:pretzel}
For every sufficiently large $n$, the alternating mutant knot pairs $K(5,3,3,2 n,5)$ and $K(5,3,2 n,3,5)$
can be distinguished, up to (non-equivariant) isotopy, using $J^e$. 
\end{thm}

We also establish the following relationship between the invariant $P^e(L)$ of a strongly 
involutive link $L$ and the HOMFLY--PT polynomial of $L$, which shows that, modulo $2$, the 
equivariant information carried by $P^e(L)$ collapses to the ordinary HOMFLY-PT polynomial, 
up to the explicit factor $z^{\ell(L)-1}$. Denote by $P(L)\in\bbZ[a^{\pm 1},z^{\pm 1}]$ 
the HOMFLY--PT polynomial of $L$ in the standard normalization: 
\[
 a^{-1} P(\poscrossing) -  aP(\negcrossing) = z P(\orientedresol),\qquad P(\unknot) = 1,
\]
as, for instance, in \cite[p.~231]{Mu08}.
\begin{thm}\label{t:congruence}
Let $L$ be a strongly involutive link. Then, 
\[
P^e(L)\big|_{z\,\mapsto\, z^2} = z^{\ell(L) - 1} P(L)\qquad\bmod 2.
\]
\end{thm}

\begin{rem}
In~\cite[p.~350]{Co08}, Couture asks whether his polynomial invariant $W$ coincides
with the Jones polynomial modulo \(2\). This coincidence is asserted, without
proof, in the introduction of~\cite{Co09}. In view of \cite[Remark~4.6]{LW21} and Theorem~\ref{t:main}---see also Appendix~\ref{app:tablePe}---Theorem~\ref{t:congruence} may therefore be viewed as giving a positive answer to a generalized version of Couture's question.
\end{rem}

Finally, we prove the following result, which provides evidence that, in certain cases, the invariant~\(P^e\) carries strictly more information than previously known equivariant invariants.

\begin{thm}\label{t:comparison}
There exist infinitely many pairs of strongly invertible knots with identical refined Lobb--Watson invariants 
that are nevertheless distinguished by \(P^e\).
\end{thm}

\subsection*{Structure of the paper}

Section~\ref{s:LW} reviews the Lobb--Watson refinement of Khovanov homology in the strongly involutive
setting, studies the skein relations satisfied by its graded Euler characteristic and 
proves Theorem~\ref{t:pretzel}. In Sections~\ref{s:Nbraket}--\ref{s:Npoly} we introduce equivariant MOY graphs, define 
the equivariant $N$--bracket,
and construct the equivariant $\mathfrak{sl}_N$ polynomials, proving their invariance.
The local relations satisfied by the equivariant $N$--bracket are established in Section~\ref{s:Nbraket},
with some of the more technical proofs deferred to Appendix~\ref{app:proofs}. Section~\ref{s:Pe} establishes existence 
and properties of the polynomial invariant $Q(L)$, introduces 
$P^e(L)$ and proves Theorems~\ref{t:main} and~\ref{t:congruence}. 
Finally, Section~\ref{s:comparison} compares $P^e(L)$ with 
the Lobb--Watson invariant and establishes Theorem~\ref{t:comparison}. 
In Appendix~\ref{app:tablePe} we list the values of $P^e$ for low--crossing
strongly invertible knots.

\section{The graded Euler characteristic of the Lobb-Watson invariant}\label{s:LW}

In this section we recall the Lobb–Watson refinement of Khovanov homology in the strongly involutive setting 
and establish equivariant skein relations satisfied by its graded Euler characteristic. 
These relations will later be generalized in the construction of our polynomial invariant. 
At the end of the section we prove Theorem~\ref{t:pretzel}. 
Throughout the section we will work in characteristic $2$, and hence all vector spaces will be vector spaces over the field with two elements~$\mathbb{F}_2$. 

In \cite{LW21}, Lobb and Watson introduce a refinement of Khovanov homology for links equipped with an involution. Their construction is general, but here we restrict to the strongly involutive case. The refinement consists of two spectral sequences converging to a triply graded object, ${\rm Kh}_{\tau}^{*,*,*}$. For strongly involutive links, almost every page of these spectral sequences is itself an invariant. We briefly recall their construction to establish notation, referring the reader to \cite{LW21} for details.

We start with the usual construction of Khovanov homology. 
Let $D$ be an oriented link diagram. 
Each crossing of $D$, regardless of orientation, admits two smoothings, 
labeled $0$ and~$1$, as shown in Figure~\ref{f:smoothings}.
\begin{figure}[ht]
\centering
\begin{subfigure}{0.28\textwidth}
\centering
\includegraphics[scale=0.5]{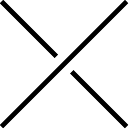}
\caption*{} 
\end{subfigure}
~
\begin{subfigure}{0.28\textwidth}
\centering
\includegraphics[scale=0.5]{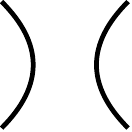}
\caption*{$0$} 
\end{subfigure}
~
\begin{subfigure}{0.28\textwidth}
\centering
\includegraphics[scale=0.5]{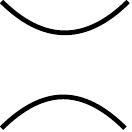}
\caption*{$1$} 
\end{subfigure}
\caption{A crossing and its two smoothings: the $0$-smoothing (centre) and the $1$-smoothing (right).}
\label{f:smoothings}
\end{figure}
A \emph{state} of $D$ is the set of circles in the plane obtained by choosing a smoothing for each crossing. 
The \emph{Khovanov chain complex} ${\rm CKh}(D)$ is the bi-graded $\mathbb{F}_2$-vector space spanned by the enhanced states of $D$, 
where an \emph{enhanced state} is a state whose circles have been labelled either by $x_{+}$ or by $x_{-}$.
Each enhanced state $s$ has a {\em homological grading}~$i(s)$ and a {\em quantum grading}~$j(s)$, which are defined as follows:
\[
i(s) = \vert\, s\, \vert - n_{-}(D)\qquad\qquad j(s) = h(s) + \vert\, s\, \vert + n_{+}(D) - 2 n_{-}(D)\,.
\]
Here, $\vert\, s\, \vert$ is the number of $1$-smoothings in the state underlying $s$; 
$h(s)$ is the number of circles labelled $x_{+}$ minus the number of circles labelled $x_{-}$; 
and $n_{+}(D)$ (resp.~$n_{-}(D)$) denotes the number of positive (resp.~negative) crossings of $D$. 
The Khovanov differential $\partial_{\rm Kh}$ increases the homological degree by~$1$ while preserving the quantum degree, 
thereby yielding a bi-graded homology.

If $D$ is a strongly involutive diagram, the involution $\tau$ acts on the set of enhanced states preserving both homological and quantum degrees. 
Furthermore, $\tau$ commutes with the differential. 
It follows that $\partial_{\tau} = {\rm id} + \tau$ and $\partial_{\rm Kh}$ are a pair of commuting differentials. 
The total complex $({\rm CKh}(D), \partial = \partial_{\rm Kh} + \partial_{\tau})$ is graded with respect to the quantum grading $j$, 
and filtered with respect to both the Khovanov homological grading $i$- and the $k${\em -grading} -- which is defined on an enhanced state $s$ as follows:
\[
k(s) = \vert\,s\, \vert^\mathrm{on} - n_-^\mathrm{on}(D) 
        + \tfrac{1}{2}\left(\, \vert\, s\, \vert^\mathrm{off} - n_{-}^\mathrm{off}(D)\, \right)\in \mathbb{Z}\,.
\]
Here the superscripts ``$\mathrm{on}$'' and ``$\mathrm{off}$'' indicate that we consider only the contributions of crossings on the axis and off the axis, respectively. 
The formula for the $k$-grading given here is a simplified version that applies to strongly involutive links; 
compare with~\cite[Definition~3.7 and Remark~3.8]{LW21}.

The two filtrations -- $\mathcal{F}$, with respect to $i$, and $\mathcal{G}$, with respect to $k$ -- 
give rise to spectral sequences converging to the associated graded objects in homology. 
Taking the associated graded object with respect to both filtrations simultaneously, together with the quantum grading, 
yields the triple grading on 
\[
{\rm Kh}_{\tau}(D) = {\rm H}({\rm CKh}(D), \partial)\,.
\]
The main result of~\cite{LW21} may be stated as follows:

\begin{thm}
The triply graded $\mathbb{F}_{2}$-vector space ${\rm Kh}_{\tau}^{*,*,*}$ 
and the bi-graded $\mathbb{F}_{2}$-vector spaces given by the pages $(E_{\mathcal{F}}^{p})^{*,*}$ ($p>1$) 
and $(E_{\mathcal{G}}^{q})^{*,*}$ ($q>2$) are invariants of strongly involutive links. 
\end{thm}

We will be primarily concerned with 
$(E_{\mathcal{G}}^{3})^{*,*} = \bigoplus_{k,j} (E_{\mathcal{G}}^{3})^{k,j}$, 
the third page of the spectral sequence associated with the $k$-filtration $\mathcal{G}$, 
and in particular with its graded Euler characteristic:
\[
J^{e}(L)(q) := \sum_{k,\, j\in \mathbb{Z}} (-1)^k 
\dim_{\mathbb{F}_2}\!\left((E_{\mathcal{G}}^{3})^{k,j}(L)\right)\, q^j 
\in \mathbb{Z}[q^{\pm 1}]\,.
\]
According to \cite[Remark~4.6]{LW21}, the bi-graded vector space $(E_{\mathcal{G}}^{3})^{*,*}$ 
is equivalent to the invariant introduced by Couture in \cite{Co09} using the language of signed divides. 
It follows that its graded Euler characteristic $J^e(L)$ is essentially the polynomial invariant defined by Couture in \cite{Co08}.

The page $E^2_{\mathcal{G}}(D)$ is homotopy equivalent to a bi-graded chain complex spanned by 
$\tau$-invariant enhanced states, with a differential that admits a simple combinatorial description 
(see \cite[Section~4.3]{LW21}). 
Consequently, $J^e(L)$ can be written explicitly as
\[
J^e(L)(q) = \sum_{s\in \mathcal{S}^e(D)} (-1)^{k(s)} q^{j(s)}\,,
\]
where $\mathcal{S}^e(D)$ denotes the set of all $\tau$-equivariant enhanced states of $D$. 
From the definitions of the $j$- and $k$-gradings, it follows that $J^e(L)$ is invariant under reversal 
of the orientation of $D$, and thus of $L$. 

The main result of this section is that the polynomial $J^e(L)$ satisfies the following properties:
\begin{gather}
q^{2}J^e(\negcrossingonaxis) - q^{-2}J^e(\poscrossingonaxis) = (q -q^{-1})J^e(\orientedresolutiononaxis) \label{eq:skeinJoneson}\\
q^4J^e(\negcrossingoffaxis) - q^{-4}J^e(\poscrossingoffaxis) = (q^2-q^{-2})J^e(\orientedresolutionoffaxis) \label{eq:skeinJonesoff}\\
J^e(\Claspa) = q^{2}(q+q^{-1})J^e(\poscrossingonaxis) - q^4 J^e(\orientedresolutiononaxis) \label{eq:skeinJonesclaspa}\\
J^e(\acirc{}) = q+q^{-1} \quad\text{and}\quad J^e(\acircs{}) =q^2 + q^{-2}\label{eq:skeinJonesnorm}\\
J^e(L\sqcup_e L') = J^e(L)J^e(L')\label{eq:skeinJonesmolt}\\
J^e(L^m)(q) = J^e(L)(q^{-1})\label{eq:skeinJonesmirr}
\end{gather}
where $L^m$ is the equivariant mirror image of $L$ and $\sqcup_{e}$ denotes the equivariant disjoint union.

\begin{prop}\label{p:skeinLWCpoly}
The invariant $J^e(L)\in \mathbb{Z}[q^{\pm 1}]$ satisfies Properties~\eqref{eq:skeinJoneson}--\eqref{eq:skeinJonesmirr}. 
\end{prop}

To prove Proposition~\ref{p:skeinLWCpoly}, it is useful to introduce an analogue of the Kauffman bracket for $J^e$, 
together with some of its basic properties. 
We call this polynomial the \emph{equivariant Kauffman bracket}, and define it as follows:
\[
\langle D\rangle^e := \sum_{s\in \mathcal{S}^e(D)} 
(-1)^{\vert\, s\, \vert^\mathrm{on} + \tfrac{1}{2}\vert\, s\, \vert^\mathrm{off}}
q^{h(s) + \vert\, s\, \vert} 
= (-1)^{n_{-}^\mathrm{on}(D) + \tfrac{1}{2} n_{-}^\mathrm{off}(D)} 
   q^{2n_{-}(D)-n_{+}(D)} J^{e}(D)(q)\,.
\]
Note that the equivariant Kauffman bracket is defined for \emph{unoriented} involutive diagrams. 
We also remark that this bracket is closer in spirit to the modified Kauffman bracket appearing in Khovanov homology 
(cf.~\cite[Section~1.2]{Tu17}) than to the original Kauffman bracket. 

Before proceeding, we introduce some notation. A strongly involutive diagram $D$ is said to be \emph{split} if there exists an equivariant simple closed curve $\gamma$ in the plane that does not intersect $D$ and separates it into two nonempty parts. If $D'$ and $D''$ denote the two subdiagrams of $D$ determined by $\gamma$, then we say that $D$ is the \emph{equivariant disjoint union} of $D'$ and $D''$, and we write $D = D' \sqcup_e D''$. We now require the following lemma:

\begin{lem}\label{l:skeinbracketpoly}
The equivariant Kauffman bracket satisfies the following relations:
\begin{gather}
\langle \crossingonaxisa \rangle^e = \langle \zeroresonaxis \rangle^e - q\, \langle \oneresonaxis \rangle^e, 
\qquad 
\langle \crossingonaxisb \rangle^e = \langle \oneresonaxis \rangle^e - q\, \langle \zeroresonaxis \rangle^e, 
\label{eq:skeinbracketon} \\[0.5em]
\langle \doublecrossingoffa \rangle^e = \langle \doublezerores \rangle^e - q^2 \langle \doubleoneres \rangle^e, 
\qquad 
\langle \doublecrossingoffb \rangle^e = \langle \doubleoneres \rangle^e - q^2 \langle \doublezerores \rangle^e, 
\label{eq:skeinbracketoff} \\[0.5em]
\langle \acirc{} \rangle^e = q + q^{-1}, 
\qquad 
\langle \acircs{} \rangle^e = q^2 + q^{-2}, 
\label{eq:skeinbracketnorm} \\[0.5em]
\langle D \sqcup_e D' \rangle^e = \langle D \rangle^e \, \langle D' \rangle^e, 
\label{eq:skeinbracketunion}
\end{gather} 
\end{lem} 
\begin{proof}
We start with Equation~\eqref{eq:skeinbracketon}. 
Let $D$ be an unoriented involutive diagram, and let $c$ be a crossing of $D$ on the axis. 
Denote by $D_0$ and $D_1$ the diagrams obtained by performing a $0$- and a $1$-smoothing on~$c$, respectively. 
There are obvious maps $\calS^e(D_0)\to\calS^e(D)$ and $\calS^e(D_1)\to\calS(D)$, yielding a bijection
\[
\varphi:\mathcal{S}^e(D_0)\sqcup \mathcal{S}^e(D_1) \longrightarrow \mathcal{S}^e(D)\,.
\]
It is straightforward to check that for any $s\in \mathcal{S}^e(D_*)$, and $*\in \{0,1\}$, we have~$\vert\varphi(s)\vert = \vert s\vert + *$, 
$\vert\varphi(s)\vert^\mathrm{on} = \vert s\vert^\mathrm{on} + *$, 
$\vert\varphi(s)\vert^\mathrm{off} = \vert s\vert^\mathrm{off}$  
and $h(\varphi(s))=h(s)$. 
It follows that
\begin{align*}
\langle D\rangle^e 
&= \sum_{s\in \mathcal{S}^e(D)} 
   (-1)^{\vert\, s\, \vert^\mathrm{on} + \tfrac{1}{2}\vert\, s\, \vert^\mathrm{off}}
   q^{h(s) + \vert\, s\, \vert} \\[0.5em]
&= \sum_{s_0\in \mathcal{S}^e(D_0)} 
   (-1)^{\vert\, \varphi(s_0)\, \vert^\mathrm{on} + \tfrac{1}{2}\vert\, \varphi(s_0)\, \vert^\mathrm{off}}
   q^{h(\varphi(s_0)) + \vert\, \varphi(s_0)\,\vert} \\[0.3em]
&\quad + \sum_{s_1\in \mathcal{S}^e(D_1)} 
   (-1)^{\vert\, \varphi(s_1)\, \vert^\mathrm{on} + \tfrac{1}{2}\vert\, \varphi(s_1)\, \vert^\mathrm{off}}
   q^{h(\varphi(s_1)) + \vert\, \varphi(s_1)\,\vert} \\[0.5em]
&= \langle D_0\rangle^e - q\, \langle D_1\rangle^e,
\end{align*}
as desired. 
The proof of Equation~\eqref{eq:skeinbracketoff} is entirely analogous, while 
Equation~\eqref{eq:skeinbracketnorm} is a direct consequence of the definition. 
Finally, Equation~\eqref{eq:skeinbracketunion} follows easily using the natural 
identification of $\calS^e(D\sqcup_e D')$ with $\calS^e(D)\times\calS^e(D')$. 
\end{proof}

\begin{proof}[Proof of Proposition~\ref{p:skeinLWCpoly}]
Equalities~\eqref{eq:skeinJoneson} and~\eqref{eq:skeinJonesoff} are immediate consequences 
of the definitions together with Equations~\eqref{eq:skeinbracketon} and~\eqref{eq:skeinbracketoff}, respectively. 

Let $D_{++}$, $D_{+}$, and $D_{0}$ be the three strongly involutive diagrams, differing only in a small ball near the axis, 
involved in Equation~\eqref{eq:skeinJonesclaspa} (see Figure~\ref{fig:skeinclaspcoerent}). 
Given an oriented strongly involutive diagram $D$, set
\[
\eta(D) := (-1)^{n_{-}^\mathrm{on}(D) + \tfrac{1}{2}n_{-}^\mathrm{off}(D)} 
            q^{\,n_{+}(D)\, -\, 2n_{-}(D)}\,,
\]
and observe that $\eta(D_{++}) = q\,\eta(D_{+}) = q^2 \eta(D_{0})$. 
The proof of Equation~\eqref{eq:skeinJonesclaspa} is then a straightforward computation:
\begin{align*} 
J^e(D_{++}) 
&= \eta(D_{++}) \langle D_{++} \rangle^e \\[0.3em]
&\overset{(\star)}{=} \eta(D_{++})
   \bigl(\langle \zeroresonaxis \rangle^e - q^2(q+q^{-1})\, \langle \oneresonaxis \rangle^e\bigr) \\[0.3em]
&= \eta(D_{++})
   \bigl(q(q+q^{-1}) \langle \crossingonaxisa \rangle^e 
        - q^2 \langle \zeroresonaxis \rangle^e \bigr) \\[0.3em]
&= q^2(q+q^{-1}) J^e(D_{+}) - q^4 J^e(D_{0})\,.
\end{align*}
In the equality marked $(\star)$ we used 
\eqref{eq:skeinbracketnorm}, \eqref{eq:skeinbracketoff}, and \eqref{eq:skeinbracketunion}, 
while the subsequent equality is a consequence of \eqref{eq:skeinbracketon}. 
Equation~\eqref{eq:skeinJonesmolt} follows immediately from \eqref{eq:skeinbracketunion}.

It remains to prove Equation~\eqref{eq:skeinJonesmirr}. 
To each enhanced state $s$ of $D$, we associate an enhanced state $\iota(s)$ of $D^m$. If we interpret each $i$-smoothed crossing in $s$ as an $(1-i)$-smoothed crossing in $\iota(s)$, the circles in the underlying states of $s$ and $\iota(s)$ can be naturally identified. 
The labels on the circles in $s$ are replaced with complementary labels in $\iota(s)$. More precisely, every circle labelled by $x_\pm$ in $s$ is labelled by $x_\mp$ in $\iota(s)$. It is clear that $\iota$ is a bijection. 
Moreover, for each $\tau$-invariant enhanced state $s$, we have
\[
j(s) = - h(\iota(s)) + n - \vert \iota(s)\vert + n_{-}(D^m) - 2n_{+}(D^m)  
     = -\, j(\iota(s))\,,
\]
and
\[
k(s) = n^\mathrm{on}(D^m) - \vert \iota(s)\vert^\mathrm{on} - n_{+}^\mathrm{on}(D^m) 
       + \tfrac{1}{2}\bigl(n^\mathrm{off}(D^m) - \vert \iota(s)\vert^\mathrm{off} - n_{+}^\mathrm{off}(D^m)\bigr) 
       \equiv k(\iota(s)) \pmod{2}.
\]
Now, Equation~\eqref{eq:skeinJonesmirr} follows directly from the state-sum definition of $J^e$.
\end{proof}

The invariant $J^e$ is quite efficient in distinguishing symmetries. For instance, $J^e$ can be used to differentiate all strongly invertible prime knots with less than $8$ crossings -- as it can be easily deduced from the tables in  \cite[Section~6.5]{LW21}. As an application of Proposition~\ref{p:skeinLWCpoly} we now prove Theorem~\ref{t:pretzel}.

\begin{proof}[Proof of Theorem~\ref{t:pretzel}]
By \cite[Corollary~1.4~b)]{BZ87}, the knot $K_n := K(5,3,3,2 n, 5)$ has symmetry group isomorphic to 
$\mathbb{Z}/2\mathbb{Z}$. Therefore, this group must be generated by the strong involution $\tau_n$ of $K_n$ with 
axis of fixed points shown in the diagram on the left-hand side of Figure~\ref{f:pretzel}.
\begin{figure}[ht]
\includegraphics[scale=0.5]{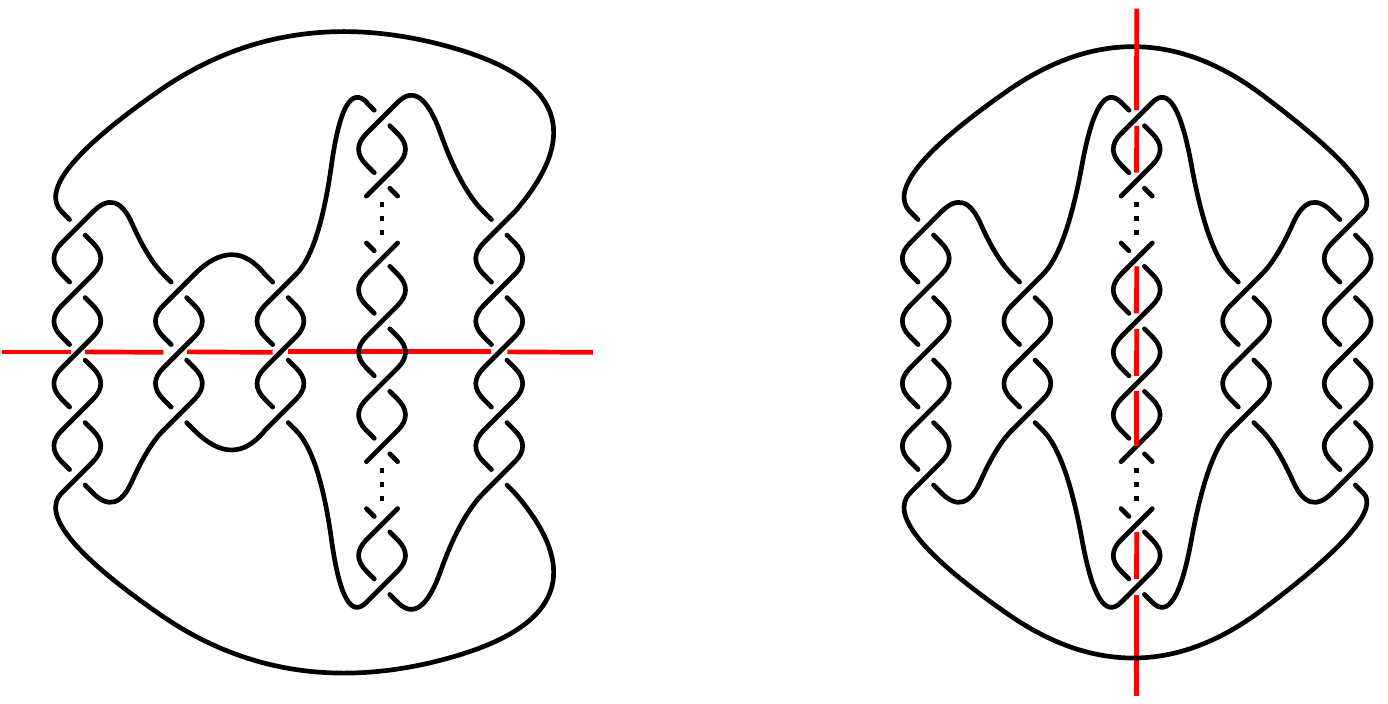}
\caption{The strongly invertible knots $(K_n,\tau_n)$ and $(K'_n,\tau'_n)$}
\label{f:pretzel}
\end{figure}
It follows that $\tau_n$ is the only strong inversion of $K_n$ up to equivalence. Now consider the strong 
involution $\tau'_n$ of $K'_n := K(5,3,2 n,3,5)$ with axis of fixed points shown in the diagram on the right-hand side of Figure~\ref{f:pretzel}. To prove the first part of the statement it suffices to show that 
$J^e(K_n,\tau_n)\neq J^e(K'_n,\tau'_n)$ for every sufficiently large $n$. 
Skein Relation~\eqref{eq:skeinJonesoff} applied to one of the $n$ symmetric pairs of (negative) crossings  
shows that, setting $R_n := J^e(K_n,\tau_n)$, we have 
\[
q^4 R_n - q^{-4} R_{n-2} = (q^2 - q^{-2})\, R_{n-1}.
\]
Thus, $R_n$ satisfies the homogeneous, second-order difference equation
\begin{equation}\label{e:difference}
x_n - q^{-4} (q^2 - q^{-2})\, x_{n-1} - q^{-8} x_{n-2} = 0.
\end{equation}
Each sequence $\{x_n\}_{n\geq 0}$ is clearly determined by the initial values $x_0$ and $x_1$, and moreover an explicit expression for $x_n$ in terms of $x_0$ and $x_1$ can be found with the following well-known procedure -- see, for instance,~\cite[Section 2.3]{El05}. 
The characteristic roots of Equation~\eqref{e:difference} are 
\[
r_{\pm} = \frac{q^2 - q^{-2} \pm (q^2 + q^{-2} )}{q^4} = 
\pm q^{-4 \pm 2}\,
\]
and the general solution of Equation~\eqref{e:difference} is of the form 
$R_n = A q^{-2n} + (-1)^n B q^{-6n}$, 
where $(A,B)$ is the unique solution of the system 
\[
\begin{cases}
	A + B &= R_0\\
	A q^{-2} - B q^{-6}  &= R_1\,.
\end{cases}
\]
By Skein Relation~\eqref{eq:skeinJoneson}, $R'_n := J^e(K'_n,\tau'_n)$ satisfies 
$q^2 R'_n - q^{-2} R'_{n-2} = (q - q^{-1})\, R'_{n-1}$,  
and arguing as above we conclude $R'_n = A' q^{-n} + (-1)^n B' q^{-3n}$,
where $(A',B')$ is the unique solution of  
\[
\begin{cases}
A' + B'  &= R'_0\\
A' q^{-1} + B' q^{-3}  &= R'_1\,.
\end{cases}
\]
Since the rational functions $A$, $B$, $A'$, $B'$ are independent of $n$, if at least one of them is nonzero it follows that $R_n\neq R'_n$ for every sufficiently large $n$. Hence, to finish the proof it suffices to check that 
$R'_0 = J^e(K'_0,\tau_0)\neq 0$. This is an easy exercise which is left to the reader -- cf.~Proposition~\ref{p:Pneofverticalsums}. 
\end{proof}

\section{The Equivariant \texorpdfstring{$N$}{N}-Bracket}\label{s:Nbraket}

In~\cite{MOY98}, Murakami, Ohtsuki, and Yamada introduced a definition of the HOMFLY–PT polynomial based on invariants of certain plane graphs. In this section, we adapt the MOY graph formalism to the strongly involutive setting by introducing equivariant MOY graphs and defining the corresponding equivariant bracket. 

A \emph{strongly involutive MOY graph} is an oriented, trivalent graph $\Gamma$ embedded in $\mathbb{R}^2$ such that:
\begin{enumerate}[label =(\roman*)]
    \item each vertex has at least one incoming and one outgoing edge;
    \item if an edge is the unique incoming edge at one of its vertices, then it is also the unique outgoing edge at its other vertex;
    \item the reflection $(x,y) \mapsto (x,-y)$ maps $\Gamma$ to itself while reversing the orientation of all edges.
\end{enumerate}
We allow $\Gamma$ to have a finite number of closed components without vertices, called \emph{loops}, which still satisfy condition~(iii).

The set $E(\Gamma)$ of edges of $\Gamma$ contains a distinguished subset $TE(\Gamma) \subset E(\Gamma)$ consisting of those edges $e \in E(\Gamma)$ that are the unique incoming edge at one of their vertices - and hence the unique outgoing edge at the other. These edges are called \emph{thick edges} and will be depicted in red, non-thick edges, depicted in black, will be called \emph{thin}; see Figure~\ref{f:thickedges}.
\begin{figure}[ht]
\centering
\includegraphics[scale=0.4]{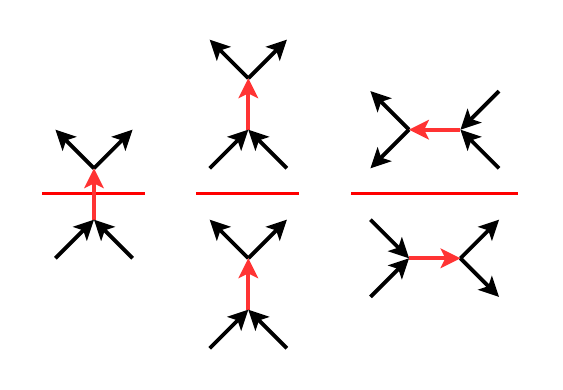}
\caption{Examples of thick edges in the equivariant setting}
\label{f:thickedges}
\end{figure}

We now define the notion of state of $\Gamma$. Fix an integer $N \geq 2$ and set
\[
X_N = \{-N+1, -N+3, \ldots, N-3, N-1\}.
\]
An \emph{equivariant $N$-state} $\sigma$ is an assignment of a subset $A \subset X_N$ to each edge $e \in E(\Gamma)$, subject to the following conditions:
\begin{itemize}
    \item if $e \in TE(\Gamma)$, then $|\sigma(e)| = 2$;
    \item if $e \in E(\Gamma) \setminus TE(\Gamma)$, then $|\sigma(e)| = 1$;
    \item if $e$ is the only thick edge at a vertex $v$, with corresponding outgoing thin edges $e_1$ and~$e_2$, then $\sigma(e) = \{\sigma(e_1), \sigma(e_2)\}$;
    \item if $e$ and $e'$ are exchanged by the reflection $(x,y) \mapsto (x,-y)$, then $\sigma(e) = \sigma(e')$.
\end{itemize}

The set of equivariant $N$-states of $\Gamma$ is denoted by $\mathcal{S}_N^e(\Gamma)$, or simply $\mathcal{S}_N^e$ when $\Gamma$ is understood. Figure~\ref{f:statexample} shows an example of an equivariant state in $\mathcal{S}_4^e(\Gamma)$.
\begin{figure}[ht]
\centering
\begin{overpic}[scale=0.4,tics=10]{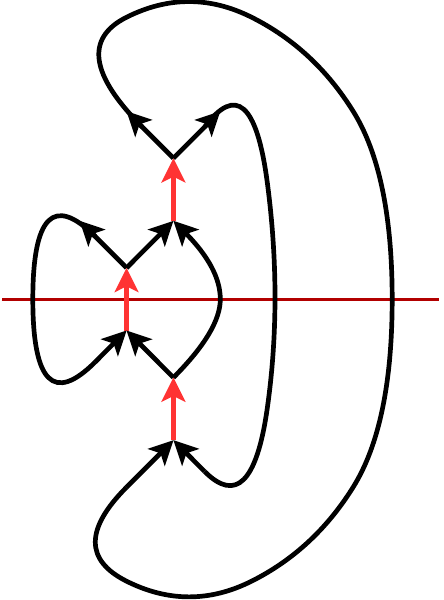}
\put(5,65){$\scriptstyle 3$}
\put(10,10){$\scriptstyle 1$}
\put(17,34){$\scriptstyle -1$}
\put(18,63){$\scriptstyle -1$}
\put(37,42){$\scriptstyle 1$}
\put(36,14){$\scriptstyle -1$}
\end{overpic}
\caption{An equivariant state of $\calS_4^e(\Ga)$}
\label{f:statexample}
\end{figure}

The \emph{equivariant $N$-bracket} $\langle\Ga\rangle^e_N$ is defined by the sum 
\[
\langle\Ga\rangle^e_N = \sum_{\si\in\calS_N^e}\left(\prod_{v\in V(\Ga)}\wt(v,\si)\right)q^{\rot(\si)}, 
\]
where $V(\Ga)$ is the set of vertices of $\Ga$ and $\wt(v,\si)$, $\rot(\si)$ are, respectively, 
the \emph{weight} associated to $v$ and $\si$ and the \emph{rotation} of $\si$ defined 
in~\cite{MOY98} as follows. The weight of $v$ and $\si$ is given by 
\[
\wt(v,\si) = q^{\frac{1}{2}\sgn(x_2 - x_1)} = 
\begin{cases}
q^{1/2}  & \text{if $x_1 < x_2$}\\
q^{-1/2}  &   \text{if $x_1 > x_2$}
\end{cases}, 
\]  
where $q$ is a formal variable and the outgoing/incoming edges are colored as in Figure~\ref{f:colorededges}.
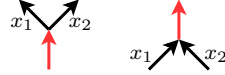
\begin{figure}[ht]
\centering
\begin{overpic}[abs,unit=1pt,scale=0.5]{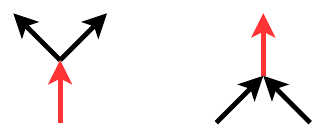}
\put(0,20){$\scriptstyle x_1$}
\put(22,20){$\scriptstyle x_2$}
\put(45,8){$\scriptstyle x_1$}
\put(73,8){$\scriptstyle x_2$}
\end{overpic}
\caption{Coloring of the outgoing/incoming edges at a vertex}
\label{f:colorededges}
\end{figure}
In order to define $\rot(\si)$ observe that at each thick edge we must have one of the configurations of Figure~\ref{f:configurations} (left). 
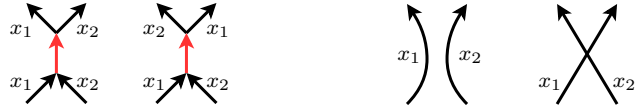
\begin{figure}[ht]
\centering
\begin{overpic}[abs,unit=1pt,scale=0.5]{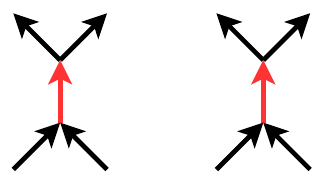}
\put(-3,10){$\scriptstyle x_1$}
\put(22,10){$\scriptstyle x_2$}
\put(-3,30){$\scriptstyle x_1$}
\put(22,30){$\scriptstyle x_2$}
\put(47,10){$\scriptstyle x_1$}
\put(71,10){$\scriptstyle x_2$}
\put(47,30){$\scriptstyle x_2$}
\put(71,30){$\scriptstyle x_1$}
\end{overpic}
\hspace{2cm}
\begin{overpic}[abs,unit=1pt,scale=0.5]{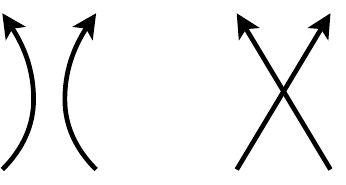}
\put(-3,20){$\scriptstyle x_1$}
\put(20,20){$\scriptstyle x_2$}
\put(50,10){$\scriptstyle x_1$}
\put(78,10){$\scriptstyle x_2$}
\end{overpic}
\caption{Configurations of the outgoing and incoming edges at a vertex}
\label{f:configurations}
\end{figure}
Replacing each of these configurations with the corresponding configuration on the right of 
Figure~\ref{f:configurations}, we obtain a finite set of simple closed curves 
$C$. Each curve $\gamma$ in $C$ is embedded,  naturally oriented and labelled by a color $x(\gamma)\in X_N$. Furthermore, the labelling of $C$ is invariant under the reflection of $\bbR^2$ across the $x$-axis -- see Figure~\ref{f:curves}.
\begin{figure}[ht]
\begin{overpic}[scale=0.4,tics=10]{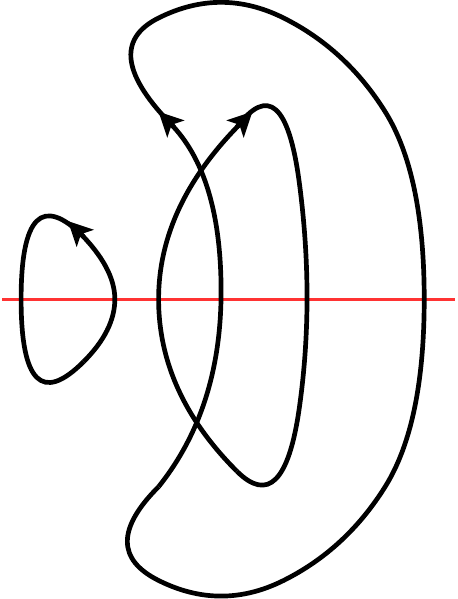}
\put(4,65){$\scriptstyle 3$}
\put(16,10){$\scriptstyle 1$}
\put(53,52){$\scriptstyle -1$}
\end{overpic}
\caption{Colored simple closed curves associated to the state in Figure~\ref{f:statexample}}
\label{f:curves}
\end{figure}
Then, the rotation number $\rot(\si)$ is defined by the formula 
\[
\rot(\si) = \sum_{\gamma\in C} x(\gamma) \rot(\gamma),
\]
where $\rot(\gamma)$ is $1$ if $\gamma$ is oriented clockwise and $-1$ otherwise.  

\medskip
\noindent\textbf{Properties of the equivariant $N$-bracket.}
Now we analyze the equivariant $N$-bracket and establish the relations it satisfies, 
preparing the ground for the construction of equivariant $\mathfrak{sl}_N$ polynomial 
invariants in the next section. 
The proofs of Lemmas~\ref{l:bubbles}--\ref{l:M1identity} are combinatorial and somewhat lengthy, and are therefore deferred to Appendix~\ref{app:proofs}.

In what follows, diagrams are identical outside the angle brackets $\langle\,\cdot\,\rangle_N^e$, and each equality remains valid if all diagram orientations are simultaneously reversed. Set 
\[
[N] := \frac{q^N - q^{-N}}{q - q^{-1}}
\quad\text{and}\quad
[N]_{q^2} := \frac{q^{2N} - q^{-2N}}{q^2 - q^{-2}}.
\]

\begin{lem}\label{l:unknots}
The following equalities hold:
\[
\langle\, \acirc{} \rangle_N^e = [N]
\quad\text{and}\quad 
\langle\, \acircs{} \rangle_N^e = [N]_{q^2}
\]
\end{lem}
\begin{proof}
We prove the second equality proceeding along the lines of \cite[Lemma~2.1]{MOY98}. The first equality can be proved similarly. 
By definition we have
\[ \langle\, \acircs{} \rangle_N^e = \sum_{\sigma\in\mathcal{S}_N^e} q^{\rot (\sigma)} = \sum_{x\in X_{N}} q^{\pm 2 x}= [N]_{q^{2}}\, . \]
In the penultimate term, the sign depends on the orientation of the two circles—which is the same for both, 
since the graph is equivariant. However, this sign is irrelevant, as $[N]_{q^2}$ remains unchanged under the 
substitution $q\mapsto q^{-1}$.
\end{proof}

\begin{lem}\label{l:bubbles}
The following equalities hold:
\[
\begin{split}
\langle\Soneedgeresol\rangle_N^e= [N-1]_{q^2}\langle\arcsoffaxis\rangle_N^e\,,
\qquad 
\langle\Roneedgeresol\rangle_N^e = [N-1]\langle\arcuponaxis\rangle_N^e\,,
\qquad
\langle\Roneedgeresolb\rangle_N^e = [N-1]\langle\arcuponaxisb\rangle_N^e
\end{split}
\]
\end{lem}

\begin{lem}\label{l:edgesbubbles}
The following equalities hold:
\[
\langle\edgebubbleoffaxis\rangle^e_N = [2]_{q^2} \langle\edgeresolutionoffaxis\rangle^e_N,
\qquad
\langle\twoedgesbubbleaxis\rangle^e_N = [2] \langle\edgeresolutiononaxis\rangle^e_N,
\qquad 
\langle\threeedgesaxisbubbles\rangle^e_N = [2]_{q^2}  \langle\edgeresolutiononaxis\rangle^e_N
\]
\end{lem}

\begin{lem}\label{l:square}
The following equalities holds:
\begin{gather*}
\langle\Rtwodoubleedgeresolution\rangle^e_N  = \langle\Rtwonocrossings\rangle^e_N + [N-2] \langle\arcsdownuponaxis\rangle^e_N,\qquad\qquad \langle\IRtworesii\rangle^e_N =\langle\IRtwobottomleft\rangle^e_N + [N-2]_{q^2}\langle\IRtworesoob\rangle^e_N\,,\\
\langle\Squaredoubleedgesoffaxis\rangle^e_N  = \langle\arcsdownuponaxis\rangle^e_N + [N-2] \langle\Rtwonocrossings\rangle^e_N 
\end{gather*}
\end{lem}

\begin{lem}\label{l:M1identity}
The following equalities holds:
\[
\langle\Monedoubleedgeoffedgeon\rangle^e_N  + \langle\arcedgeon\rangle^e_N = 
\langle\Moneedgeondoubleedgeoff\rangle^e_N + \langle\edgeonarc\rangle^e_N ,\qquad \langle\IRthreeonethicka\rangle^e_N  + \langle \IRthreethreethicka \rangle^e_N = 
\langle\IRthreethreethickb\rangle^e_N  + \langle \IRthreeonethickb \rangle^e_N
\]
\end{lem}

\section{The equivariant \texorpdfstring{$\mathfrak{sl}_N$}{slN} Polynomials}\label{s:Npoly}

By~\cite[Theorem~2.3]{LW21}, diagrams $D_1$ and $D_2$ of equivalent involutive oriented links $L_1$ 
and $L_2$, respectively, 
are related by a sequence of equivariant planar isotopies and moves given in~\cite[Figure~9]{LW21} and 
called IR1, IR2, IR3, R1, R2, M1, M2 and M3. From now on, we will simply refer to these as the~\emph{Lobb–Watson moves}. 

In this section, we introduce the bracket polynomial $\langle\,\cdot\,\rangle^e_N$ for knotted involutive graph diagrams and the equivariant $\mathfrak{sl}_N$ polynomial $P^e_N$.
We examine how these polynomials behave under the oriented Lobb–Watson moves and show that $P^e_N$ remains invariant under all such moves -- 
and thus that $P_N^e$ defines an invariant of strongly involutive links. We warn the reader that in Lobb–Watson~\cite{LW21} the fixed-point axis is taken to be vertical rather than horizontal. 

Let $D$ be an oriented diagram of a knotted graph obtained by equivariantly replacing some crossings in an oriented strongly involutive link diagram with thick edges.
In analogy with~\cite[\S 3]{MOY98}, we define \emph{the bracket polynomial} $\langle D\rangle^e_N$ recursively by setting: 
\begin{align*}
\langle\poscrossingonaxis\rangle^e_N  & := 
q^{-1} \langle\orientedresolutiononaxis\rangle^e_N - 
\langle\edgeresolutiononaxis\rangle^e_N\,,
& 
\langle\negcrossingonaxis\rangle^e_N  & := 
q \langle\orientedresolutiononaxis\rangle^e_N - 
\langle\edgeresolutiononaxis\rangle^e_N\,,
\\ 
\langle\poscrossingoffaxis\rangle^e_N & := 
q^{-2} \langle\orientedresolutionoffaxis\rangle^e_N - 
\langle\edgeresolutionoffaxis\rangle^e_N\,,
& 
\langle\negcrossingoffaxis\rangle^e_N & := 
q^2 \langle\orientedresolutionoffaxis\rangle^e_N - 
\langle\edgeresolutionoffaxis\rangle^e_N\,.
\end{align*}

It follows immediately from the definition, and from the invariance under orientation reversal of the equivariant $N$-bracket, that the equivariant bracket polynomial $\langle D\rangle^e_N$ satisfies  
\begin{equation}\label{e:rev-bracket}
\langle D^r\rangle^e_N = \langle D\rangle^e_N,
\end{equation}
where $D^r$ is the diagram obtained from $D$ by reversing the orientation. Similarly, we have 
\begin{equation}\label{e:mirror-bracket}
\langle D^m\rangle^e_N (q) = \langle D\rangle^e_N (q^{-1}),
\end{equation}
where $D^m$ denotes the diagram obtained by switching all the crossings of $D$.

\begin{prop}\label{p:IR1-behaviour}
The bracket polynomial $\langle\,\cdot\, \rangle^e_N$ satisfies the following identities: 
\[
\langle\IRonenegloop\rangle^e_N = q^{2N} \langle\arcsoffaxis\rangle^e_N\,, 
\qquad
\langle\IRoneposloop\rangle^e_N = q^{-2N}\langle\arcsoffaxis\rangle^e_N\,.
\]
\end{prop}

\begin{proof}
The two identities are obtained by ``doubling" the proof 
of~\cite[Lemma~3.3]{MOY98} as in the proof of Lemma~\ref{l:bubbles}. 
\end{proof}

\begin{prop}\label{p:IR23-invariance}
The bracket polynomial $\langle\,\cdot\, \rangle^e_N$ satisfies the following identities: 
\[ 
\langle\IRtwotopleft\rangle^e_N = 
\langle\IRtwobottomleft\rangle^e_N\,,
\quad
\langle\IRtwotopright\rangle^e_N  = 
\langle\IRtwobottomright\rangle^e_N\,,
\quad
\langle\IRthreeone\rangle^e_N = 
\langle\IRthreetwo\rangle^e_N \]
\end{prop}
\begin{proof}
Applying first the definition and then Lemmas~\ref{l:bubbles} and~\ref{l:square} to the left-hand side of the first equality we obtain   
\[
\langle\IRtwotopleft\rangle^e_N = \langle\IRtworesoo\rangle^e_N - q^{-2}\langle\IRtworesoi\rangle^e_N - q^2\langle\IRtworesio\rangle^e_N + \langle\IRtworesii\rangle^e_N = 
\]
\[
= \left([N]_{q^2} - (q^2+q^{-2})[N-1]_{q^2}+[N-2]_{q^2}\right)\langle\IRtworesoob\rangle^e_N
+ \langle\IRtwobottomleft\rangle^e_N = \langle\IRtwobottomleft\rangle^e_N\,.
\]
The second equality is a consequence of the first one and Equation~\eqref{e:mirror-bracket}. 
To prove the third equality we observe that, thanks to the first equality in Lemma~\ref{l:edgesbubbles} and the planar isotopy 
invariance of the bracket polynomial, we have:
\[
\begin{split}
\langle\IRthreeone\rangle^e_N = 
q^{-6}\langle\IRthreeresooo\rangle^e_N -q^{-4}\left[ \langle\IRthreeresioo\rangle^e_N+\langle\IRthreeresoio\rangle^e_N \right]\\  
+ q^{-2}\left[ \langle\IRthreeresiio\rangle^e_N+\langle\IRthreeresoii\rangle^e_N \right] + \langle\IRthreeresoio\rangle^e_N - \langle\IRthreeresiiia\rangle^e_N\,.
\end{split}
\]
A completely analogous computation yields:
\[ 
\begin{split}
\langle\IRthreetwo\rangle^e_N =  
q^{-6}\langle\IRthreeresooo\rangle^e_N -q^{-4}\left[ \langle\IRthreeresioo\rangle^e_N+\langle\IRthreeresoio\rangle^e_N \right]\\ 
+ q^{-2}\left[ \langle\IRthreeresiio\rangle^e_N+\langle\IRthreeresoii\rangle^e_N \right] + \langle\IRthreeresioo\rangle^e_N - \langle\IRthreeresiiib\rangle^e_N\,. 
\end{split}
\]
Now the third, and last, equality follows immediately from Lemma~\ref{l:M1identity}.
\end{proof}

Given a diagram $D$ of an involutive oriented link $L$, \emph{the equivariant 
$\mathfrak{sl}_N$ polynomial} $P^e_N(D)$ is defined by setting 
\[
P^e_N(D) := q^{N w(D)} \langle D\rangle^e_N\, ,
\]
where $w(D)$ denotes the writhe of $D$.

\begin{prop}\label{p:IR123-invariance}
The equivariant $\mathfrak{sl}_N$ polynomial $P^e_N(D)$ remains invariant under the oriented Lobb-Watson moves of types IR1, IR2 and IR3. 
\end{prop}

\begin{proof}
We make use of results from~\cite{Po10}, which identifies specific generating sets for the oriented Reidemeister moves. 
Adopting the notation from that paper, Propositions~\ref{p:IR1-behaviour} and~\ref{p:IR23-invariance} together imply that $P^e_N(L)$ 
is invariant under equivariant pairs of Reidemeister moves of types in the set
$S=\{\Omega1b,\Omega1d,\Omega2c,\Omega2d, \Omega3b\}$. The proposition then follows from~\cite[Theorem~1.2]{Po10}, which shows that the set 
$S$ generates all oriented Reidemeister moves.
\end{proof}

\begin{prop}\label{p:R1-behaviour}
The bracket polynomial $\langle\,\cdot\, \rangle^e_N$ satisfies the following identities: 
\[
\begin{split}
\langle\Roneposloop\rangle^e_N = q^{-N}\langle\arcuponaxis\rangle^e_N\,,
\qquad
\langle\Ronenegloop\rangle^e_N = q^N \langle\arcuponaxis\rangle^e_N\,, 
\\
\langle\Roneposloopb\rangle^e_N = q^{-N}\langle\arcuponaxisb\rangle^e_N\,,
\qquad
\langle\Ronenegloopb\rangle^e_N = q^N \langle\arcuponaxisb\rangle^e_N\,. 
\end{split}
\]
\end{prop}

\begin{proof} 
The equalities follow from the definition and Lemmas~\ref{l:unknots} and~\ref{l:bubbles}.
Since the proofs are similar, we give the argument only for the first equality: 
\[
\langle\Roneposloop\rangle^e_N = q \langle\Roneorientedres\rangle^e_N - \langle\Roneedgeresol\rangle^e_N
= (q [N] - [N-1]) \langle\arcuponaxis\rangle^e_N =  q^N \langle\arcuponaxis\rangle^e_N\,. 
\]
\end{proof}

\begin{prop}\label{p:R2-invariance}
The bracket polynomial $\langle\,\cdot\, \rangle^e_N$ satisfies the following identities:
\[
\langle\Rtwocrossings\rangle^e_N = \langle\Rtwonocrossings\rangle^e_N 
= \langle\Rtwocrossingsmirror\rangle^e_N
\]
\end{prop}

\begin{proof}
In view of Equation~\eqref{e:mirror-bracket}, it suffices to prove the first equality. 
\[
\langle\Rtwocrossings\rangle^e_N = q \langle\Rtwoorientedresolution\rangle^e_N  - \langle\Rtwoedgeresolution\rangle^e_N =
\]
\[
= q \left(q^{-1}\langle\Rtwodoubleresolution\rangle^e_N - \langle\Rtwoedgeorientedresolution\rangle^e_N\right) - 
q^{-1} \langle\Rtwoorientededgeresolution\rangle^e_N + \langle\Rtwodoubleedgeresolution\rangle^e_N =
\]
\[
= ([N] - (q+q^{-1})[N-1] + [N-2])\langle\arcsdownuponaxis\rangle^e_N + \langle\Rtwonocrossings\rangle^e_N  
= \langle\Rtwonocrossings\rangle^e_N  
\]
\end{proof}

\begin{lem}\label{l:box}
The moves 
\[ 
\LemmastatoneA\ \raisebox{1pt}{$\overset{Ma}{\longleftrightarrow}$}\ \raisebox{2pt}{$\LemmastatoneB$}
\qquad \text{and}\qquad 
\LemmastattwoA\ \raisebox{1pt}{$\overset{Mb}{\longleftrightarrow}$}\ \raisebox{2pt}{$\LemmastattwoB$} 
\]
are equivalent modulo Lobb-Watson {\rm R2}-moves and planar isotopies -- the grayed area indicates a fixed subdiagram. 
\end{lem}

\begin{proof}
Assuming $Mb$, the following shows that $Ma$ holds. The converse is proved similarly. 
\[ 
\LemmastatoneA
\ \raisebox{1pt}{$\overset{R2}{\longleftrightarrow}$}
\ \LemmaproofA 
\ \raisebox{1pt}{$\overset{Mb}{\longleftrightarrow}$} 
\ \raisebox{2.5pt}{$\LemmaproofB$}
\ \overset{R2}{\longleftrightarrow} 
\ \raisebox{1pt}{$\LemmastatoneB$} 
\]
\end{proof}

\begin{prop}\label{p:M1-invariance}
The bracket polynomial $\langle\,\cdot\, \rangle^e_N$ satisfies the following identities:
\[
\begin{split}
\langle\Monecrossonfirst\rangle^e_N = \langle\Monecrossonsecond\rangle^e_N\,,
\quad 
\langle\Monecrossonfirstb\rangle^e_N = \langle\Monecrossonsecondb\rangle^e_N\,,
\\
\langle\Moneposcrossonfirst\rangle^e_N = \langle\Moneposcrossonsecond\rangle^e_N\,,
\quad
\langle\Moneposcrossonfirstb\rangle^e_N = \langle\Moneposcrossonsecondb\rangle^e_N\,.
\end{split}
\]
\end{prop}

\begin{proof}
In view of Equations~\eqref{e:rev-bracket} and~\eqref{e:mirror-bracket} and Lemma~\ref{l:box}, 
it suffices to prove the first equality. By the definition of the bracket polynomial we get 
\[
\begin{split}
\langle\Monecrossonfirst\rangle^e_N = 
q^3 \langle\Moneallresolvedfirst\rangle^e_N - q [2] \langle\arcedgeon\rangle^e_N 
- q\langle\edgeonarc\rangle^e_N + \langle\Moneedgeondoubleedgeoff\rangle^e_N\,,
\\
\langle\Monecrossonsecond\rangle^e_N =
q^3 \langle\Moneallresolvedsecond\rangle^e_N - q \langle\arcedgeon\rangle^e_N 
- q [2] \langle\edgeonarc\rangle^e_N + \langle\Monedoubleedgeoffedgeon\rangle^e_N\,.
\end{split}
\]
The equality follows applying Lemma~\ref{l:M1identity}.
\end{proof}

\begin{prop}\label{p:M2invariance}
The bracket polynomial $\langle\,\cdot\, \rangle^e_N$ satisfies the following identities: 
\[
\langle\Mtwocrossingsneg\rangle^e_N = \langle\poscrossingonaxis\rangle^e_N
\quad\text{and}\quad 
\langle\Mtwocrossingspos\rangle^e_N = \langle\negcrossingonaxis\rangle^e_N.
\]
\end{prop}

\begin{proof}
Applying the recursive definition and Lemma~\ref{l:edgesbubbles} we get 
\[
\langle\Mtwocrossingsneg\rangle^e_N = q^{-2} \langle\negcrossingonaxis\rangle^e_N - 
\langle\Mtwoedgeresolutions\rangle^e_N = 
q^{-2}\left(q\langle\orientedresolutiononaxis\rangle^e_N - \langle\edgeresolutiononaxis\rangle^e_N\right) 
- \left(q\langle\twoedgesbubbleaxis\rangle^e_N - \langle\threeedgesaxisbubbles\rangle^e_N\right)
\]
\[
= q^{-1}\langle\orientedresolutiononaxis\rangle^e_N + 
\left(-q^{-2} - q [2] + [2]_{q^2}\right) \langle\edgeresolutiononaxis\rangle^e_N 
= q^{-1}\langle\orientedresolutiononaxis\rangle^e_N - \langle\edgeresolutiononaxis\rangle^e_N
= \langle\poscrossingonaxis\rangle^e_N
\]
This yields the first equality. The second equality can be proved applying Equation~\eqref{e:mirror-bracket}.
\end{proof}

\begin{lem}\label{l:edgeR3}
The bracket polynomial $\langle\,\cdot\, \rangle^e_N$ satisfies the following identities: 
\[
\langle\edgeIRthreefirst\rangle^e_N = \langle\edgeIRthreesecond\rangle^e_N\,,
\qquad
\langle\edgeIRthreefirsta\rangle^e_N = \langle\edgeIRthreeseconda\rangle^e_N\,.
\]
\end{lem}

\begin{proof}
By the definition of the equivariant bracket and Proposition~\ref{p:IR23-invariance} we have 
\[
q^2 \langle\IRthreeres\rangle^e_N - \langle\edgeIRthreefirst\rangle^e_N = \langle\IRthreeoneProof\rangle^e_N = 
\langle\IRthreetwoProof\rangle^e_N = q^2 \langle\IRthreeres\rangle^e_N - \langle\edgeIRthreesecond\rangle^e_N,
\]
which immediately implies the first equality. The second equality follows similarly, except that one must apply 
Proposition~\ref{p:IR123-invariance} in place of Proposition~\ref{p:IR23-invariance}, using also the fact that IR3 
moves preserve the writhe and thus leave $\langle\,D\, \rangle^e_N = q^{-N w(D)} P^e_N(D)$ invariant.
\end{proof}

\begin{lem}\label{l:edgeM3}
The bracket polynomial $\langle\,\cdot\, \rangle^e_N$ satisfies the following identities:
\[
\langle\Mthreeedgeonfirst\rangle^e_N = \langle\Mthreeedgeonsecond\rangle^e_N\,,
\qquad
\langle\Mthreeedgeonfirsta\rangle^e_N = \langle\Mthreeedgeonseconda\rangle^e_N\,.
\]
\end{lem}

\begin{proof}
We claim that the following holds:
\[
[2]\langle\Mthreeedgeonfirst\rangle^e_N 
= \langle\Mthreebubblededgefirst\rangle^e_N 
= \langle\MthreebubblededgeRtwo\rangle^e_N 
= \langle\Mthreeapplylemmafirst\rangle^e_N 
= \langle\Mthreeedgesymmetric\rangle^e_N\,.
\]
Indeed, the first equality follows from Lemma~\ref{l:edgesbubbles}, the second one by Proposition~\ref{p:IR23-invariance}, 
the third equality from Lemma~\ref{l:edgeR3} and the last one by Proposition~\ref{p:M1-invariance}. 
Applying a $\pi$-rotation about a vertical axis in the plane to the diagrams above produces a similar sequence of equalities, leading~to
\[
[2]  \langle\Mthreeedgeonsecond\rangle^e_N  = \langle\Mthreeedgesymmetric\rangle^e_N = [2]\langle\Mthreeedgeonfirst\rangle^e_N.
\]
Since the ring $\mathbb{Z}[q, q^{-1}]$ is a domain, we can cancel the factor~ $[2]$, thereby establishing the first equality. 
The proof of the second equality follows the same reasoning, with the only difference being that Lemma~\ref{l:edgesbubbles} 
must be used in conjunction with Equation~\eqref{e:rev-bracket}, and Proposition~\ref{p:IR23-invariance} must be 
replaced by Proposition~\ref{p:IR123-invariance}.
\end{proof}

\begin{prop}\label{p:M3-invariance-1}
The bracket polynomial $\langle\,\cdot\, \rangle^e_N$ satisfies the following identities:
\[
\begin{split}
\langle\Mthreefirst\rangle^e_N = \langle\Mthreesecond\rangle^e_N\,,
\qquad
\langle\Mthreefirsta\rangle^e_N = \langle\Mthreeseconda\rangle^e_N\,,
\\
\langle\Mthreefirstb\rangle^e_N = \langle\Mthreesecondb\rangle^e_N\,,
\qquad
\langle\Mthreefirstd\rangle^e_N = \langle\Mthreesecondd\rangle^e_N\,.
\end{split}
\]
\end{prop}

\begin{proof}
By Proposition~\ref{p:M1-invariance} and the first identity of Lemma~\ref{l:edgeM3} we have 
\[
\langle\Mthreefirst\rangle^e_N = q \langle\Mthreeresonfirst\rangle^e_N 
- \langle\Mthreeedgeonfirst\rangle^e_N = q \langle\Mthreereson\rangle^e_N - 
\langle\Mthreeedgeonsecond\rangle^e_N = 
\]
\[
q \langle\Mthreeresonsecond\rangle^e_N - \langle\Mthreeedgeonsecond\rangle^e_N = 
\langle\Mthreesecond\rangle^e_N\,. 
\]
This completes the proof of the first equality. The second follows in the same way, simply by replacing 
$q$ with $q^{-1}$ in the argument above. To prove the third equality, observe that by Proposition~\ref{p:IR123-invariance}, 
the bracket polynomial  $\langle\,\cdot\, \rangle^e_N = q^{-N w(D)} P^e_N$ is invariant under any oriented IR2 move. 
Therefore, applying the first equality together with IR2 invariance, we obtain 
\[
\langle\Mthreefirstb\rangle^e_N = 
\langle\Mthreebproof\rangle^e_N =
\langle\Mthreesecondb\rangle^e_N\,.
\]
The fourth equality follows from the third by combining Propositions~\ref{p:IR123-invariance} and~\ref{p:M2invariance}. 
Specifically, one begins by applying M2 invariance at the rightmost crossing on the axis in the left-hand side diagram.
\end{proof}

\begin{prop}\label{p:M3-invariance-2}
The bracket polynomial $\langle\,\cdot\, \rangle^e_N$ satisfies the following identities:
\[
\begin{split}
\langle\Mthreefirstrev\rangle^e_N = \langle\Mthreesecondrev\rangle^e_N\,,
\qquad
\langle\Mthreefirstarev\rangle^e_N = \langle\Mthreesecondarev\rangle^e_N\,,
\\
\langle\Mthreefirstbrev\rangle^e_N = \langle\Mthreesecondbrev\rangle^e_N\,,
\qquad
\langle\Mthreefirstdrev\rangle^e_N = \langle\Mthreeseconddrev\rangle^e_N\,.
\end{split}
\]
\end{prop}

\begin{proof}
The first two identities can be proved in the same way as those in Proposition~\ref{p:M3-invariance-1}, 
using Proposition~\ref{p:M1-invariance} together with Lemma~\ref{l:edgeM3}. The last two identities then 
follow directly from the corresponding identities in Proposition~\ref{p:M3-invariance-1}, by applying Lemma~\ref{l:box}.
\end{proof}

\begin{thm}\label{t:LW-invariance}
The equivariant $\mathfrak{sl}_N$ polynomial $P^e_N(D)$ is invariant under all the oriented Lobb-Watson moves. 
\end{thm}

\begin{proof}
Invariance under the oriented moves IR1, IR2, and IR3 is established in Proposition~\ref{p:IR123-invariance}. 
Combining Equations~\eqref{e:rev-bracket} and~\eqref{e:mirror-bracket} with Propositions~\ref{p:R1-behaviour}, 
\ref{p:R2-invariance}, \ref{p:M1-invariance}, \ref{p:M2invariance}, \ref{p:M3-invariance-1}, and~\ref{p:M3-invariance-2}, 
we conclude that the equivariant $\mathfrak{sl}_N$ polynomial $P^e_N(L)$ is invariant under all oriented R1, R2, R3, M1, M2, and M3 moves. 
This completes the proof. 
\end{proof}

\begin{rem}
In \cite{Sa86}, Sakuma defines two strongly invertible knots to be equivalent if there exists an orientation-preserving diffeomorphism of $S^3$ that commutes with the involution and sends one knot to the other.
Invariance under Lobb--Watson moves is slightly weaker. Indeed, by \cite[Proposition 2.4]{LW21}, two diagrams represent Sakuma-equivalent knots if and only if they are related by a sequence of Lobb--Watson moves, possibly followed by a $\pi$-rotation about a point on the axis.
Since $P^{e}_N$ is defined as a suitable rescaling of a state sum, and such a $\pi$-rotation affects neither the rescaling factor nor the state sum, it follows that $P^{e}_N$ is invariant under Sakuma equivalence of strongly invertible knots. 
\end{rem}

Let $\Lambda$ denote the ring $\bbZ[q,q^{-1}]$ of Laurent polynomials with integer 
coefficients, endowed with the ring involution $*:\Lambda\to\Lambda$ sending $q$ to $q^{-1}$. 
Let $(R,\varphi)$ be a pair consisting of an integral, unital domain 
$R$ with an involution $*:R\to R$ together with a homomorphism $\varphi:\Lambda\to R$ of rings 
with involution. We regard $R$ as a $\Lambda$-algebra via $\varphi$. 
When $\varphi$ is understood we write $\lambda\, r$ instead of $\varphi(\lambda)\, r$. 

Let $\sid$ be the set of strongly involutive diagrams and $N\geq 2$ an integer. 
Given $D,\,D'\in \sid$ we write $D\stackrel{\rm LW}{\sim} D'$ if the two diagrams are 
related by a finite sequence of Lobb-Watson moves and equivariant planar isotopies. 
Given a pair $(R,\varphi)$ as above, a map \(I_N \colon \sid \to R\) 
will be called an \emph{equivariant skein invariant of type $(R,\varphi,N)$} 
if it satisfies 
\begin{equation}\label{eq:lwinvariance}
I_N(D) = I_N(D')\quad\text{whenever}\quad D\stackrel{\rm LW}{\sim} D' 
\end{equation}
as well as the following properties: 
\begin{gather}
q^N I_N(\negcrossingonaxis) - q^{-N} I_N(\poscrossingonaxis) = (q-q^{-1}) I_N(\orientedresolutiononaxis) \label{eq:PNskeinonaxis}\\
q^{2N}I_N(\negcrossingoffaxis) - q^{-2N}I_N(\poscrossingoffaxis) = (q^2 - q^{-2})I_N(\orientedresolutionoffaxis) \label{eq:PNskeinoffaxis}\\
I_N(\Claspa) = (q+q^{-1})q^N I_N(\poscrossingonaxis) - q^{2N} I_N(\orientedresolutiononaxis) \label{eq:PNskeincoerentclasp}\\
I_N(\acirc{}) = \varphi([N])\quad\text{and}\quad I_N(\acircs{}) = \varphi([N]_{q^2})\label{eq:PNskeinnormalisation}\\
I_N(D\sqcup_e D') =  I_N(D) I_N(D')\label{eq:PNskeinunion}\\
I_N(D^*) =  *\, I_N(D) \label{eq:PNskeinmirror}
\end{gather}
where $D^*$ is the mirror image of $D$ and $\sqcup_{e}$ denotes the equivariant disjoint union.

\begin{prop}\label{p:skeinPN}
For each $N\geq 2$, the polynomial invariant $P^e_N$ is an equivariant skein invariant 
of type $(\Lambda,{\rm id}_\Lambda,N)$. 
\end{prop}

\begin{proof}
Property~\eqref{eq:PNskeinnormalisation} follows immediately from the definition.
Since Properties~\eqref{eq:PNskeinonaxis}, \eqref{eq:PNskeinoffaxis},  \eqref{eq:PNskeinunion}, 
and \eqref{eq:PNskeinmirror} are proven exactly as in the non-equivariant case, we concentrate 
on the remaining Property~\eqref{eq:PNskeincoerentclasp}. 
Let $D_{++}$, $D_{+}$ and $D_{0}$ be three strongly involutive diagrams differing only in small ball near the axis as illustrated in Figure~\ref{fig:skeinclaspcoerent}.
\begin{figure}[ht]
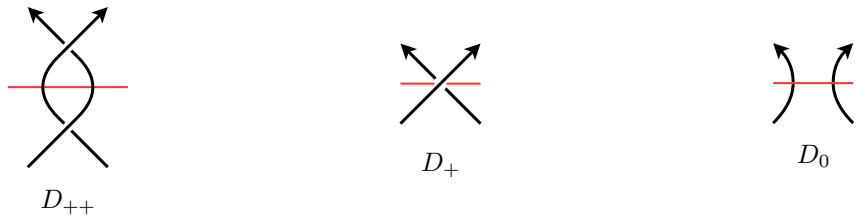

\centering
\begin{subfigure}[h]{0.3\textwidth}
\centering
\includegraphics[scale=0.5]{figurepdf/ClaspCoer.pdf}
\caption*{$D_{++}$}
\end{subfigure}
~
\begin{subfigure}[h]{0.3\textwidth}
\centering
\includegraphics[scale=0.5]{figurepdf/poscrossingonaxis.pdf}
\caption*{$D_{+}$}
\end{subfigure}
~
\begin{subfigure}[h]{0.3\textwidth}
\centering
\includegraphics[scale=0.5]{figurepdf/orientedresolutiononaxis.pdf}
\caption*{$D_{0}$}
\end{subfigure}
\caption{The local differences in the diagrams involved in Property~\eqref{eq:PNskeincoerentclasp}.}
\label{fig:skeinclaspcoerent}
\end{figure}

Then, we have
\begin{align*} P_{N}^e(D_{++}) &= q^{Nw(D_{++})}\langle \Claspa \rangle_{N}^e =q^{Nw(D_{++})}\left( q^{-2}\langle \orientedresolutiononaxis\rangle_{N}^e - \langle\twoedgesbubbleaxis\rangle^e_N \right) =\\ & =q^{Nw(D_{++})}\left( (q^{-2}- q^{-1}[2])\langle \orientedresolutiononaxis\rangle_{N}^e + [2] \left(q^{-1}\langle \orientedresolutiononaxis\rangle_{N}^e-  \langle\edgeresolutiononaxis\rangle^e_N \right)\right)
\\ & =q^{Nw(D_{++})}\left( -\langle \orientedresolutiononaxis\rangle_{N}^e + [2] \langle\poscrossingonaxis\rangle^e_N \right) = -q^{2N}P_{N}^e(D_{0}) + (q+q^{-1})q^{N}P_{N}^e(D_{+})
\end{align*}
which completes the proof.
\end{proof}

\begin{defn}\label{d:slnpol}
In view of Theorem~\ref{t:LW-invariance}, given a strongly involutive link $L$ we define its 
\emph{equivariant $\mathfrak{sl}_N$ polynomial $P^e_N(L)$} to be the equivariant 
$\mathfrak{sl}_N$ polynomial $P^e_N(D)$, where $D$ is any diagram of $L$. 
\end{defn}

\section{The Polynomial Invariant \texorpdfstring{$P^e$}{Pe}}\label{s:Pe}

In this section we prove uniqueness of the polynomials $P_N^e$, 
construct the auxiliary two-variable polynomial invariant $Q$ of strongly involutive links 
and prove Theorem~\ref{t:main}.
 
If a strongly involutive diagram $D\in\sid$ represents a strongly involutive link $L$,  
we denote by~$\ell(D)$, respectively $2 m(D)$, the number of connected components of $L$ with, 
respectively without, fixed points. With these conventions in place, we construct a map 
$Q:\sid\to\bbZ[a^{\pm 1}, z]$ such that $Q(D)=Q(D')$ if $D\stackrel{\rm LW}{\sim} D'$, 
and show that it satisfies the following properties:
\begin{equation}\label{eq:skeinonaxis}
a\, Q\left(\negcrossingonaxis\right) - a^{-1} Q\left(\poscrossingonaxis\right)
  = z^{m - m_0}\, Q\left(\orientedresolutiononaxis\right),
\end{equation}
where 
\[
m = m\left(\negcrossingonaxis\right) = m\left(\poscrossingonaxis\right),\quad 
m_0 = m\left(\orientedresolutiononaxis\right),
\]
\begin{equation} \label{eq:skeinoffaxis}
a^2 Q\left(\negcrossingoffaxis\right) - a^{-2} Q\left(\poscrossingoffaxis\right)
  = z^{m-m_0+1}\, Q\left(\orientedresolutionoffaxis\right),
\end{equation}
where
\[
m = m\left(\negcrossingoffaxis\right) = m\left(\poscrossingoffaxis\right),\quad 
m_0 = m\left(\orientedresolutionoffaxis\right),
\]
\begin{equation}\label{eq:skeinincoerentclasp}
Q\left(\Claspa\right)
  = a\, z^{m - m_0 +1}\, Q\left(\poscrossingonaxis\right) - a^2 Q\left(\orientedresolutiononaxis\right), 
\end{equation}
where 
\[
m = m\left(\Claspa\right) = m\left(\orientedresolutiononaxis\right),\quad 
m_0 = m\left(\poscrossingonaxis\right),
\]
\begin{gather} 
Q\left(\acirc{}\right) = 1\quad\text{and}\quad Q\left(\acircs{}\right) = a + a^{-1}, \label{eq:skeinnormalisation}\\
Q(D \sqcup_e D') = (a - a^{-1})\,  Q(D)\, Q(D'), \label{eq:skeinunion}\\
Q(D^*)(a,z) = (-1)^{\ell(L)+m(L)+1} Q(D)(a^{-1},-z). \label{eq:skeinmirror}
\end{gather}

\medskip
\noindent\textbf{Outline of the proof.}
The argument proceeds by reducing the general case to progressively simpler
configurations, combining equivariant skein relations with a double
induction on the complexity of the diagram. We denote by \(\fsid\subset\sid\) the subset of \emph{free} involutive diagrams, i.e.~those diagrams $D\in\sid$ with $\ell(D)=0$, and by \(\nfsid\subset\sid\) the subset of \emph{strongly invertible} involutive diagrams, i.e.~those diagrams with $m(D)=0$. 

We first introduce the notion of a \emph{good} diagram and show that
goodness propagates along equivariant crossing changes and skein
relations. This allows us to reduce the proof to a class of diagrams
that are simpler from the equivariant viewpoint.

Next, we prove that all free involutive diagrams are good and show that,
after suitable equivariant crossing changes, any diagram can be separated
into its strongly invertible and free parts. This reduces the problem to
the case of strongly invertible diagrams.

For strongly invertible diagrams, we eliminate crossings on the axis and
then argue by induction on the number $p^{\mathrm{on}}(D)$ of interleaving
connected pairs of fixed points. When $p^{\mathrm{on}}(D)=0$, the diagram
can be reduced to a crossingless unlink. In the general case, a minimal
interleaving configuration gives rise to a clasp, and the equivariant
skein relation reduces the crossing number. This completes the induction 
and yields the desired uniqueness result,
as well as the existence and uniqueness of the polynomial $Q(D)$.
\medskip

We denote by \(|D| = \ell(D) + 2m(D)\) the total number of connected components of $L$, by \(n(D)\) the total number of crossings of $D$, and by \(n^\mathrm{on}(D)\) the number of crossings of $D$ on the axis. A strongly involutive diagram is \emph{split} if it can be separated by an equivariant simple closed curve.

\begin{defn}\label{d:good}
A strongly involutive diagram $D\in\sid$ is \emph{good} if, 
for each $N\geq 2$, it satisfies the following two conditions:
\begin{itemize}
\item 
if $I_N$ is an equivariant skein invariant of type $(R,\varphi,N)$, then $I_N(D)=\varphi(P_N^e(D))$;
\item 
there exists a polynomial $Q(D)\in\bbZ[a^{\pm 1}, z]$ such that 
\begin{equation}\label{eq:good}
(q - q^{-1})^{\ell(D)} (q^2 - q^{-2})^{m(D)} P^e_N(D) = 
(q^N - q^{-N}) Q(D)\big|_{a=q^N,\ z=q^2-q^{-2}}
\end{equation}
\end{itemize}
\end{defn}

Let $\gsid \subseteq \sid$ denote the subset of good diagrams. We will prove that $\gsid = \sid$. This result 
has two important consequences. First, $P_N^e$ is the unique equivariant skein invariant of type $(\Lambda,{\rm id}_\Lambda, N)$ for each 
$N \ge 2$. Second, there exists a map $Q$ assigning to every strongly involutive link diagram $D$ a  
polynomial $Q(D) \in \mathbb{Z}[a^{\pm 1}, z]$ such that Equation~\eqref{eq:good} is satisfied.

\begin{lem}\label{l:specialize-vanish}
Let \(R\) be an integral domain. If
\[
P(a,q)\in R[a^{\pm1},q^{\pm1}]
\]
satisfies
\[
P(q^N,q)=0\in R[q^{\pm1}]
\]
for all integers \(N\ge 2\), then \(P(a,q)=0\).
\end{lem}

\begin{proof}
Note that, since $a,\,q\in R[a^{\pm 1} ,q^{\pm 1}]$ and $q\in R[q^{\pm 1}]$ are invertible, there is no loss of generality in working with $a^r q^s P(a,q)\in R[a,q]$ in lieu of $P(a,q)$. Thus, we may assume that~$P(a,q)$ is a polynomial, and not just a Laurent polynomial, in $a$ and $q$. Therefore, assuming by contradiction \( P(a,q) \neq 0\), we can write:
\[
P(a,q)=\sum_{i=0}^d a^i f_i(q),\qquad f_i(q)\in R[q],\qquad f_d\neq 0.
\]
Substituting \(a=q^N\), and using the hypothesis, gives
\[
0=P(q^N,q)=\sum_{i=0}^d q^{Ni} f_i(q)\in R[q]
\qquad\text{for all }N\ge 2.
\]
Let \(M_i=\deg f_i\). The maximal \(q\)-degree of the summand \(q^{Ni}f_i(q)\) equals \(Ni+M_i\).
For \(N\) sufficiently large, the maximum of \(Ni+M_i\) over \(i=0,\dots,d\) is achieved uniquely at \(i=d\)
(since \(Ni\) grows strictly with \(i\)). Hence, for all such \(N\), the highest-degree term of
\(\sum_{i=0}^d q^{Ni} f_i(q)\) comes only from \(q^{Nd}f_d(q)\) and cannot cancel with the other summands.
This forces \(f_d(q)=0\), a contradiction. Therefore \(P(a,q)=0\).
\end{proof}

\begin{rems}\label{rs:Q-LW} We collect here some easy consequences of the definitions which will be needed afterwards.
\begin{enumerate}
\item[(a)] If $D$ is a good, strongly involutive link diagram, then the polynomial $Q(D)$ 
satisfying Equation~\eqref{eq:good} is unique. In fact, if $Q_1(D)$ and $Q_2(D)$ are two 
such polynomials, the fact that $P(D) = Q_1(D) - Q_2(D)$ vanishes follows immediately applying 
Lemma~\ref{l:specialize-vanish} to the Laurent polynomial $P(a,q) = P(D)|_{z=q^2-
q^{-2}}\in\bbZ[a^{\pm 1},q^{\pm 1}]$ and using the fact that 
$q^2-q^{-2}\in\bbZ[a^{\pm 1},q^{\pm 1}]$ is not the root of 
a non-zero polynomial with coefficients in $\bbZ[a^{\pm 1}]$.
\item[(b)]
Let $D, D'$ be strongly involutive link diagrams with $D\stackrel{\rm LW}{\sim} D'$. 
Then, $D$ is good if and only if $D'$ is good. In fact, if $D$ is good and $I_N$ is 
an equivariant skein invariant of type $(R,\varphi,N)$, then 
$\varphi(P_N^e(D')) = \varphi(P_N^e(D)) = I_N(D) = I_N(D')$, and since $\ell(D)=\ell(D')$ and  
$m(D)=m(D')$  the polynomial $Q(D)$ satisfying Equation~\eqref{eq:good} for $D$, 
satisfies the same equation for $D'$. 
\item[(c)]
If follows easily from Property~\eqref{eq:PNskeinunion} and Equation~\eqref{eq:good}
that if $D, D'$ are strongly involutive, good link diagrams, then $D\sqcup_e D'\in\gsid$.
\item[(d)]
Property~\eqref{eq:PNskeinmirror} and Equation~\eqref{eq:good} imply  
$D\in\gsid\Longleftrightarrow D^*\in\gsid$.
\end{enumerate}
\end{rems}

\medskip
\noindent\textbf{Propagation of goodness.}
We begin by establishing that goodness is preserved under the equivariant
skein relations and crossing changes. These results allow us to deduce
goodness of a diagram from that of simpler diagrams obtained by
local modifications.

For the rest of this section we denote by $I_N$ an equivariant skein invariant of type $(R,\varphi,N)$, for some $N\geq 2$. Consider a triple of diagrams $D_+, D_-, D_0\in\sid$, where $D_+$ and $D_-$ are obtained on from the other by changing a crossing on the axis, and $D_0$ is obtained by smoothing that same crossing. Assuming that $D_\pm$ and $D_0$ are good, we want to show that $D_\mp$ is good. Equation~\eqref{eq:PNskeinonaxis}
can be written as follows: 
\[
q^{\pm N} I_N(D_\mp) = q^{\mp N} I_N(D_\pm) + (q-q^{-1}) I_N(D_0).
\]
Since by assumption $\varphi(P_N^e(D_\pm))=I_N(D_\pm)$ and $\varphi(P_N^e(D_0))=I_N(D_0)$ and 
by Proposition~\ref{p:skeinPN} we have that $P_N^e$ is an equivariant skein invariant of type 
$(\Lambda,{\rm id}_\Lambda,N)$, this implies $\varphi(P_N^e(D_\mp))=I_N(D_\mp)$.  
It is easy to check that $\ell(D_0)=\ell(D_\pm)+1$ and $m(D_\pm)\geq m(D_0)$. Multiplying both sides 
of the equation by $q^{\mp N}(q - q^{-1})^{\ell(D_\pm)} (q^2 - q^{-2})^{m(D_\pm)}$ and using the assumption that that $D_\pm$ and $D_0$ are good yields 
\begin{multline*}
(q - q^{-1})^{\ell(D_\pm)} (q^2 - q^{-2})^{m(D_\pm)} P_N^e(D_\mp) = \\
(q^N - q^{-N}) \left(q^{\mp 2 N} Q(D_\pm)\big|_{a=q^N,\ z=q^2-q^{-2}} 
+ q^{\mp N} (q^2-q^{-2})^{m(D_\pm)-m(D_0)} Q(D_0)\big|_{a=q^N,\ z=q^2-q^{-2}}\right).
\end{multline*}
If we define $Q(D_\mp)\in\bbZ[a^{\pm 1},z]$ as the polynomial 
\begin{equation*}
a^{\mp 2} Q(D_\pm) + a^{\mp 1} z^{m(D_\pm)-m(D_0)} Q(D_0),
\end{equation*}
then Equation~\eqref{eq:good} holds for $D_\mp$ and we conclude   
\begin{equation}\label{e:treenode1-good}
	D_\pm, D_0\in\gsid\quad\Longrightarrow\quad D_\mp\in\gsid\,. 
\end{equation}
Similarly, let $D_{++}, D_{--}, D_{00}\in\sid$ be a triple obtained by equivariantly 
changing and resolving crossings $c,\tau(c)\in D\in\{D_{++},D_{--}\}$. 
Assuming that $D_{\pm\pm}$ and $D_{00}$ are good, we want to show that $D_{\mp\mp}$ is good. 
Equation~\eqref{eq:PNskeinoffaxis} can be written as follows: 
\[
q^{\pm 2N} I_N(D_{\mp\mp}) = q^{\mp 2N} I_N(D_{\pm\pm}) + (q^2-q^{-2}) I_N(D_{00}).
\]
Since $D_{\pm\pm}$ and $D_{00}$ are good, and since, by definition, $I_N$ satisfies Equation~\eqref{eq:PNskeinoffaxis}, we have the equality $\varphi(P_N^e(D_{\mp\mp}))=I_N(D_{\mp\mp})$. 
Since the number of fixed points is the same in the three diagrams,  $\ell(D_{00})=\ell(D_{\pm\pm})$. Moreover, a simple check shows that~$m(D_{\pm\pm})+1\geq m(D_{00})$. Multiplying both sides 
by $q^{\mp 2N}(q - q^{-1})^{\ell(D_{\pm\pm})} (q^2 - q^{-2})^{m(D_{\pm\pm})}$ and using the assumption that 
that $D_{\pm\pm}$ and $D_{00}$ are good yields 
\begin{multline*}
(q - q^{-1})^{\ell(D_{\pm\pm})} (q^2 - q^{-2})^{m(D_{\pm\pm})} P_N^e(D_{\mp\mp}) = 
(q^N - q^{-N})\left(q^{\mp 4N} Q(D_{\pm\pm})\big|_{a=q^N,\ z=q^2-q^{-2}} +\right.\\
\left. q^{\mp 2N} (q^2-q^{-2})^{m(D_{\pm\pm})+1-m(D_{00})} Q(D_{00})\big|_{a=q^N,\ z=q^2-q^{-2}}\right).
\end{multline*}
If we define $Q(D_{\mp\mp})\in\bbZ[a^{\pm 1},z]$ as the polynomial  
\begin{equation*}
a^{\mp 4} Q(D_{\pm\pm}) + a^{\mp 2} z^{m(D_{\pm\pm})+1-m(D_{00})} Q(D_{00}),
\end{equation*}
then Equation~\eqref{eq:good} holds for $D_{\mp\mp}$ and  and we conclude   
\begin{equation}\label{e:treenode2-good}
	D_{\pm\pm}, D_{00}\in\gsid\quad\Longrightarrow\quad D_{\mp\mp}\in\gsid\,. 
\end{equation}

\begin{rem}\label{r:skeinonoffaxis}
The arguments just provided also show that if $D_+, D_-, D_0\in\sid$ are obtained 
by changing and resolving a crossing $c\in D_\pm$ on the axis, then 
Skein Relation~\eqref{eq:skeinonaxis} holds. Similarly, if $D_{++}, D_{--}, D_{00}\in\sid$ is 
a triple obtained by equivariantly changing and resolving an equivariant pair of crossings $c,\, \tau(c)$ in $D\in\{D_{++},D_{--}\}$, then Skein Relation~\eqref{eq:skeinoffaxis} holds. 
\end{rem}

The following lemma establishes the base case of the inductive arguments that follow.

\begin{lem}\label{l:crossingless} 
Every crossingless strongly involutive diagram is good.
\end{lem} 

\begin{proof}
A crossingless strongly involutive diagram $D$ is an equivariant disjoint
union of $\tau$-invariant circles on the axis and pairs of symmetric circles
disjoint from the axis. 

Then, if $\ell$ and $2m$ denote, respectively, the number of connected components of $D$ belonging to $\fsid$, 
respectively $\nfsid$, using Property~\eqref{eq:PNskeinnormalisation} one can check directly that the polynomial 
	\[
	Q(D) := (a - a^{-1})^{\ell - 1} (a^2 - a^{-2})^{m}
	\]
satisfies the requirements of Definition~\ref{d:good}, and moreover \(\varphi(P_N^e(D))=I_N(D)\) -- since both $\varphi\circ P_N^e$ and $I_N$ satisfy Equations \eqref{eq:PNskeinnormalisation} and \eqref{eq:PNskeinunion}. 
Therefore, $D$ is good. 
\end{proof} 

A finite sequence $D_1,D_2,\ldots,D_k\in\sid$ such that $D_{i+1}$ is obtained 
from $D_i$ by either a crossing change on the axis or an equivariant pair of crossing changes, 
for $i=1,\ldots, k-1$, will be called a \emph{crossing change sequence}. We shall repeatedly use the following lemma. 

\begin{lem}\label{l:good-ccs}
	Let $D_1,D_2,\ldots,D_k\in\sid$ be a crossing change sequence. If for \(n(D') < n(D_1)\) we have \(D'\in\gsid\), then
\[
D_k\in\gsid\quad\Longrightarrow\quad D_1\in\gsid
\] 
\end{lem}

\begin{proof} 
We argue by induction on $n(D_1)$. If $n(D_1)=0$, then $D_1$ is crossingless and therefore good by Lemma~\ref{l:crossingless}.
This proves the base case.

Note that the crossing number does not change during a crossing change sequence. The inductive step follows by induction on the length of the crossing change sequence, using the assumption $D_k\in\gsid$,  Implications~\eqref{e:treenode1-good} and~\eqref{e:treenode2-good}, and the relations $n(D_0)=n(D_\pm)-1$ 
and~$n(D_{00})=n(D_{\pm\pm})-2$.
\end{proof}

\begin{prop}\label{p:fsid-good}
Every diagram $D\in\fsid$ is good.
\end{prop}

\begin{proof}
We proceed by induction on the number of crossings $n(D)$.
If $n(D)=0$, then $D$ is crossingless and therefore good by Lemma~\ref{l:crossingless}. 

We now assume that the statement holds for each $D'\in \fsid$ such that $n(D')<n(D)$.
We claim that, given $D\in\fsid$, there exists a crossing change sequence $D=D_1,\ldots,D_k\in\fsid$ such that $D_k$ is good. By Lemma~\ref{l:good-ccs}, this implies that $D$ is good.  
We prove the claim by induction on the number of connected components $|D|\geq 2$ (note that, as $D$ is free, $|D|$ is even). 
	
If $|D|=2$, then $D = D'\cup\tau(D')$ for some connected component $D'\subset D$. 
Applying a sequence of equivariant crossing changes, we may arrange that at each crossing $c\in D$ the overstrand belongs to $D'$. Accordingly, at each crossing $\tau(c)\in D$, the understrand belongs to $\tau(D')$. Further equivariant crossing changes turn $D'$ into the diagram $D''$ of an unknot, with the property that at every crossing $c\in D''\cup\tau(D'')$ the overstrand lies in $D''$. Consequently, $D''$ also lies above the axis at each intersection with the axis, and hence it represents an unknot $U\subset\mathbb{R}^3$ disjoint from the $xz$ plane, bounding a disk $\Delta$ also disjoint from the $xz$ plane. Then $\tau(U)$ bounds the disk $\tau(\Delta)$, disjoint from $\Delta$, and there is an equivariant isotopy taking $U$ to an unknot which projects to a crossingless diagram $D'''$ disjoint from the $x$ axis. 
By Lemma~\ref{l:crossingless} $D'''\cup\tau(D''')$ is good. Moreover, by~\cite[Theorem~2.3]{LW21}, $D''\cup\tau(D'')$ and $D'''\cup\tau(D''')$ are related by a sequence of Lobb--Watson moves, therefore by Remark~\ref{rs:Q-LW}(b) we have that $D''\cup\tau(D'')$ is good.  
Since there is a crossing change sequence from $D=D'\cup\tau(D')$ to $D''\cup\tau(D'')$, using the induction hypothesis on the number of crossings, by Lemma~\ref{l:good-ccs} the claim holds when $|D|=2$. 
	
If $|D|>2$ and $D'\subset D$ is a connected component, we first apply equivariant crossing changes so that at every crossing $c\in D'$ the overstrand belongs to $D'$. 
It is then straightforward to check that the subdiagram $D'\cup\tau(D')\neq D$ can be separated from the rest of $D$ by a sequence of Lobb--Watson moves. 
Thus, there is a crossing change sequence $D=D_1,\ldots,D_k$ such that $D_k\stackrel{\rm LW}{\sim}\tilde D$, where $\tilde D$ is a split diagram with $|\tilde D|=|D|$. 
Hence $\tilde D = \tilde D_1\sqcup_e \tilde D_2$ with $|\tilde D_s|<|D|$, for $s=1,2$. By the induction hypotheses, each $\tilde D_s$ is already a good diagram, if $n(\tilde D_s)< n(D)$, or it admits a crossing change sequence to a good diagram -- and hence it is a good diagram by Lemma~\ref{l:good-ccs}. Therefore by Remark~\ref{rs:Q-LW}(c) the diagram $D$ is good, and the proof is complete. 
\end{proof}

\medskip
\noindent\textbf{Reduction to the strongly invertible case.}
We next reduce the general case to strongly invertible diagrams.
This is achieved by separating the free and strongly invertible parts
of a diagram via equivariant crossing changes.

First, we wish to reduce the proof of $\sid \subseteq \gsid$ to the proof that $\nfsid\subseteq \gsid$. 
In order to do so, for each diagram $D$, we are going to separate the subdiagram $D_{\rm inv} \subseteq D$, representing the union strongly invertible components, from the subdiagram $D_{\rm free} \subseteq D$, where $\tau$ acts freely.

We say that a crossing change sequence $D = D_1$, ..., $D_k$ \emph{separates the strongly invertible part of~$D$} if $D_k \stackrel{\rm LW}{\sim} D_{\rm inv} \sqcup_e D'$ with $D'\in \fsid$.

\begin{lem}
For each $D\in \sid$ there is a crossing change sequence that separates the strongly invertible part.   
\end{lem}

\begin{proof}
Write $D$ as the (not necessarily disjoint) union $D_{\rm inv} \cup D_{\rm free}$, where  $D_{\rm inv} \subseteq D$ represents the union strongly invertible components and $D_{\rm free} \subseteq D$ is the subdiagram of $D$ where $\tau$ acts freely. Note that the link represented by $D_{\rm free}$ has an even number of components; in fact, $D_{\rm free}$ consists of pairs of sub-diagrams of the form $D' \cup \tau(D')$, where each $D'$ represents a knot. Order these pairs of subdiagrams, and fix a distinguished component for each pair. We perform equivariant crossing changes so that, for the first pair, the distinguished component consistently passes over at every crossing, while its symmetric counterpart passes under. We then repeat this procedure for each of the remaining pairs, in the prescribed order, ensuring that no crossing already considered when treating a preceding pair is altered. This process yields a diagram consisting of $D_{\text{inv}}$ plus a diagram representing the equivariant disjoint union of the components of~$D_{\text{free}}$. Since the latter diagram overpasses at each crossing with $D_{\text{inv}}$, and it represents a disjoint union, it can then be separated from $D_{\text{inv}}$ via Lobb-Watson moves. 
\end{proof}
 
The next lemma reduces the proof of Theorem~\ref{t:gsid=sid} to the case of strongly invertible diagrams.

\begin{lem}\label{l:nfsid-good}
Let $n_0$ be a non-negative integer.
If every strongly invertible diagram with at most $n_0$ crossings is good, then every strongly involutive diagram with at most $n_0$ crossings is good. 
In particular, if \(\nfsid \subseteq \gsid\) then $\sid = \gsid$.
\end{lem}

\begin{proof}
Let $D\in\nfsid$ be a strongly involutive diagram with $n(D)$ crossings. 
Throughout the proof, we assume that every strongly involutive diagram with at most $n_0$ crossings is good.
We proceed by induction on the number $n(D)$. 
If $n(D)=0$, then $D$ is crossingless and therefore good by Lemma~\ref{l:crossingless}.

Now suppose $1<n(D)\leq n_0$. By the inductive hypothesis, every diagram in $\sid$ with less than $n(D)$ 
crossings is good. Let \( {\rm sp}_{si} (D) \) denote the minimum length of a crossing change sequence that 
separates the strongly involutive part of $D$. To prove the inductive step, we will proceed by induction on~\( {\rm sp}_{si} (D) \).
If \(  {\rm sp}_{si} (D) = 0 \) then \(D=D_1\sqcup_e D_2\) with \(D_1\in\nfsid\) and \(D_2\in\fsid\).
By Proposition~\ref{p:fsid-good}, \( D_2 \) is a good diagram. 
Since~\(n(D_1) \leq n(D) \leq n_0\) and~\( D_1 \in \nfsid \), the diagram \( D_1\) is good by hypothesis. Therefore \( D \) is good by Remark~\ref{rs:Q-LW}.(c). 

Choose a crossing $c$ or an equivariant pair of crossings $(c,\tau(c))$ occurring in a crossing change sequence realizing \(  {\rm sp}_{si}(D) \). 
Smoothing $c$ (or $c \cup \tau(c)$) reduces the number of crossings, and the resulting diagram is good by the inductive hypothesis on $n(D)$. Similarly, changing $c$ (or $c \cup \tau(c)$) reduces~\(  {\rm sp}_{si} \), and the resulting diagram is good by inductive hypothesis on~\(  {\rm sp}_{si} \). 
Therefore $D$ is good by Implications~\eqref{e:treenode1-good} and~\eqref{e:treenode2-good}. This completes the proof. 
\end{proof}

Before proving that $\nfsid \subseteq \sid^{\rm good}$ we need a few technical lemmas.

\begin{lem}\label{l:clasp-skein}
Let $D, D'$ and $D''$ be the diagrams appearing from left to right in Skein 
Relation~\eqref{eq:PNskeincoerentclasp}. If $D'$ and $D''$ are good, then 
$D$ is good and Skein Relation~\eqref{eq:skeinincoerentclasp} holds. 
\end{lem} 

\begin{proof}
Let $I_N$ be an equivariant skein invariant of type $(R,\varphi,N)$.  
Clearly, $\varphi(P_N^e(D'))=I_N(D')$ and $\varphi(P_N^e(D''))=I_N(D'')$ imply $\varphi(P_N^e(D))=I_N(D)$.
Moreover, we have $\ell(D)=\ell(D'')=\ell(D')+1$ and $m(D)=m(D'')$. 
Multiplying both sides of Equation~\eqref{eq:PNskeincoerentclasp} by 
\[
(q-q^{-1})^{\ell(D)} (q^2 - q^{-2})^{m(D)} = 
(q - q^{-1})^{\ell(D')} (q^2 - q^{-2})^{m(D')} (q - q^{-1}) (q^2 - q^{-2})^{m(D) - m(D')}
\]
and using the assumption that $D'$ and $D''$ are good, yields
\begin{multline*}
(q-q^{-1})^{\ell(D)} (q^2 - q^{-2})^{m(D)} P_N^e(D) = \\
(q^N - q^{-N})\left(q^N (q^2 - q^{-2})^{m(D)-m(D')+1} Q(D')\big|_{a=q^N, z=q^2-q^{-2}} - q^{2N} Q(D'')\big|_{a=q^N, z=q^2-q^{-2}}\right). 
\end{multline*}
If we define $Q(D)\in\bbZ[a^\pm,z]$ as the polynomial 
$a z^{m(D)-m(D')+1} Q(D') - a^2 Q(D'')$, then Equation~\eqref{eq:good} holds for $D$, 
so $D$ is good, and Skein Relation~\eqref{eq:skeinincoerentclasp} holds as well. 
\end{proof}

Recall that two pairs $p_1<p_2$ and $q_1<q_2$ of points on the real line
are~\emph{interleaving} if the associated open intervals $(p_1,p_2)$ and $(q_1,q_2)$ 
overlap, but neither contains the other. Given $D\in\nfsid$, we say that a pair of 
fixed points in $D$ are~\emph{connected} if they belong to a subdiagram of $D$ representing a knot.
Moreover, we define $p^\mathrm{on}(D)$ to be the number of interleaving and connected pairs of fixed points in $D$.
Recall that \(n^\mathrm{on}(D)\) denotes the number of crossings on the axis of a diagram $D$. 

\begin{lem}\label{l:erasecrossesonaxis2}
Let $D\in\nfsid$. Then, by a sequence of equivariant crossing changes and Lobb--Watson equivalences that do not increase the number of crossings, $D$ can be transformed into a diagram $D'\in\nfsid$ with 
\[
p^\mathrm{on}(D')=p^\mathrm{on}(D),\qquad n(D') = n(D) - n^\mathrm{on}(D),\qquad n^\mathrm{on}(D') = 0.
\]
\end{lem}

\begin{proof}
We argue by induction on $n^\mathrm{on}(D)$. If $n^\mathrm{on}(D)=0$ there is nothing to prove. 
Assume $n^\mathrm{on}(D)>0$. Then there exists a point $p$ of $D$ on the axis and a crossing $c$ of $D$ 
joined to $p$ by the arcs $a$ and $\tau(a)$ emanating from $p$. By equivariant crossing changes, 
we may arrange that $a$ is the understrand and $\tau(a)$ the overstrand at every crossing other than $c$.  
Alternatively, we may arrange the opposite convention, with $a$ always the overstrand 
and $\tau(a)$ the understrand. The appropriate choice (depending on the type of crossing) 
yields the configuration illustrated in Figure~\ref{f:equivisot2}, left. 
\begin{figure}[ht]
    \centering
    \includegraphics[width=0.6\linewidth]{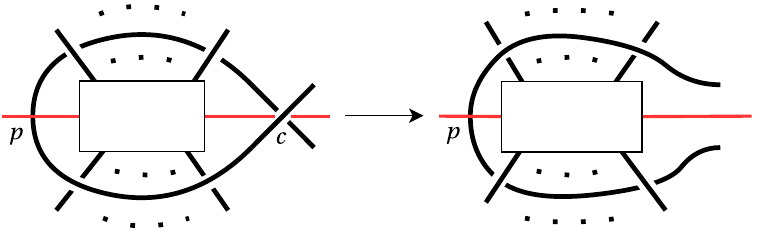}
    \caption{Equivariant isotopy cancelling a crossing.}
    \label{f:equivisot2}
\end{figure}
An equivariant isotopy of the associated link then produces a new strongly invertible link with diagram 
\(\widetilde D\) satisfying 
\[
p^\mathrm{on}(\widetilde D)=p^\mathrm{on}(D),\qquad n(\widetilde D)=n(D)-1,\qquad n^\mathrm{on}(\widetilde D)=n^\mathrm{on}(D)-1
\]
(see Figure~\ref{f:equivisot2}, right). By~\cite[Theorem~2.3]{LW21}, \(D\) and \(\widetilde D\) are related 
by a sequence of Lobb--Watson moves. Applying the inductive hypothesis to \(\widetilde D\), we obtain a diagram 
\( D'\in\nfsid\) with
\[
p^\mathrm{on}(D')=p^\mathrm{on}(\widetilde D)=p^\mathrm{on}(D),\quad 
n(D')=n(\widetilde D)-n^\mathrm{on}(\widetilde D)=n(D)-n^\mathrm{on}(D),\quad 
n^\mathrm{on}(D')=0,
\]
as required.
\end{proof}

\begin{lem}\label{l:p^on=0}
Let \( D\in \nfsid\) be such that \( p^\mathrm{on}(D) = 0\). 
If all strongly involutive diagrams with less than \(n(D)\) crossings are good, then \( D\in \gsid \).     
\end{lem}

\begin{proof}
We claim that, by a sequence of equivariant crossing changes and Lobb--Watson equivalences which do not increase the number of crossings, $D$ can be transformed into a diagram of a strongly invertible unlink. 

Write \( D = \De_1 \cup...\cup \De_\ell\), where \(\De_1 ,\, \dots,\, \De_\ell\) are subdiagrams of \( D \) representing strongly invertible knots. 
In view of Lemma~\ref{l:erasecrossesonaxis2}, we may assume that \(D\) has no crossings on the axis. Each \(\De_i\) then decomposes as the union of two arcs \(\gamma_i\) and \(\tau(\gamma_i)\) joining the fixed points \(p_i\) and \(q_i\). Since there are no crossings on the axis, up to relabeling the arcs, we may assume that \(\gamma_i\) lies entirely above the axis and \(\tau(\gamma_i)\) entirely below.

Perform equivariant crossing changes along \(\gamma_1\) so that \(\gamma_1\) is the overpassing strand at every crossing. 
Concretely, starting at \(p_1\), traverse \(\gamma_1\) and, upon first encountering each crossing, change it so that \(\gamma_1\) passes 
over. By symmetry, the corresponding crossings along \(\tau(\gamma_1)\) are changed so that \(\tau(\gamma_1)\) is always the underpassing 
strand. Thus \(\gamma_1\cup\tau(\gamma_1)\) becomes unknotted. We repeat this procedure for the remaining \(\De_i\)'s without altering crossings involving previously treated subdiagrams. This yields a crossing change sequence \( D = D_1 ,..., D_k \) where the last diagram represents a strongly involutive link with unknotted components.

We claim that \(D_k\) represents in fact an unlink. Since \(p^{\mathrm{on}}=0\), any two arcs \(\gamma_i\) and \( \gamma_j\) intersect in an even number of points -- the same holds, symmetrically, for \(\tau(\gamma_i)\) and \(\tau(\gamma_j)\). Therefore, starting from \(\gamma_1\) and \( \tau(\gamma_1)\), and proceeding in the given order, we can separate and unknot all \(\gamma_i\)'s and, symmetrically, all \(\tau(\gamma_i)\)'s by equivariant pairs of Reidemeister moves. The resulting diagram is a crossingless diagram of a strongly invertible unlink. This proves the claim.

The resulting diagram is crossingless and therefore good by Lemma~\ref{l:crossingless}.
Furthermore, by Lemma~\ref{l:nfsid-good}, if all strongly involutive diagrams with less than \(n(D)\) crossings are good, then all diagrams in \(\sid\) with less than \( n(D) \) crossings are good. Then, we can apply Lemma~\ref{l:good-ccs} and obtain that \( D \) is good. 
\end{proof}

We are now ready to show that every strongly involutive link diagram is good. 

\medskip
\noindent\textbf{Strongly invertible diagrams and the double induction.}
We now complete the proof in the strongly invertible case.
The argument proceeds by a double induction: first on the crossing number,
and then on the number $p^{\mathrm{on}}(D)$ of interleaving connected pairs
of fixed points. The key step is the reduction to a clasp configuration,
which allows the equivariant skein relation to decrease the crossing number.

\begin{thm}\label{t:gsid=sid}
Let $D$ be a strongly involutive link diagram and $I_N$ any equivariant skein invariant 
of type $(R,\varphi,N)$ with $N\geq 2$. Then, $\varphi(P_N^e(D))=I_N(D)$, and there is a unique 
polynomial $Q(D)\in\bbZ[a^{\pm 1},z]$ satisfying Equation~\eqref{eq:good}.
\end{thm}

\begin{proof}
In view of Remark~\ref{rs:Q-LW}(a) and Lemma~\ref{l:nfsid-good}, it suffices to show that, for each integer 
$n_0\geq 0$, every strongly invertible diagram with at most $n_0$ crossings is good. 

We proceed by induction on $n_0$. If $n_0=0$, then any strongly invertible diagram with at most $n_0$ crossings
is crossingless and hence good by Lemma~\ref{l:crossingless}.

Next, let $D\in\nfsid$ with $n(D)=n_0$, and suppose that all strongly involutive diagrams with at most \( n_0-1\) 
crossings are good. By Lemma~\ref{l:nfsid-good}, this implies that all diagrams in \( \sid\) 
with at most \( n_0-1\) crossings are good. We prove that $D$ is good by induction on 
$p_0 := p^\mathrm{on}(D)$. 

By Lemma~\ref{l:erasecrossesonaxis2}, \(D\) is related by equivariant crossing changes and
Lobb–Watson moves to a diagram with no crossings on the axis,
without increasing the total number of crossings and preserving
$p^{\mathrm{on}}$. By Remark~\ref{rs:Q-LW}(b), goodness is preserved under
Lobb–Watson equivalence. Moreover, by Lemma~\ref{l:good-ccs}, the inductive hypothesis on $n_0$ implies that if the last diagram of a crossing change sequence is good, then also the starting diagram is good. So we may assume that D has no crossings on the axis.
If $p^{\mathrm{on}}(D)=0$, the conclusion follows from Lemma~\ref{l:p^on=0}.

Now we assume that each $D'\in \nfsid$ with \( p^\mathrm{on}(D')< p^\mathrm{on}(D) \) and at most $n_0$ crossings is good. 
Let $p_1 < q_1$ and $p_2 < q_2$ be two interleaving pairs of connected fixed points in $D$,
with $p_1 < p_2$. Since the set of connected pairs of fixed points is finite,
and the corresponding open intervals are partially ordered by inclusion,
we may choose such a pair so that the interval $(p_2,q_1)$ is minimal with
respect to inclusion among those arising from interleaving pairs.
In particular, the pair $(p_2,q_1)$ does not interleave any other connected
pair of fixed points in $D$.

Let $D'$ be the smallest subdiagram of $D$ containing all fixed points
in the open interval $(p_2,q_1)$.
If $D' \neq \emptyset$, then $p^{\mathrm{on}}(D') < p^{\mathrm{on}}(D)$.
Indeed, by construction the pair $(p_1,q_1)$ and $(p_2,q_2)$
is no longer present as an interleaving pair inside $D'$,
while no new interleaving pairs are created.
Therefore, by the inductive hypothesis 
on $p^\mathrm{on}$, $D'$ is good.    

Let $D'' = D \setminus D'$. Up to applying equivariant crossing changes, we may assume that $D'$ overpasses 
at each crossing with $D''$ above the axis, and underpasses at each crossing with $D''$ below the axis. 
Since there are no crossing on the axis, $D \stackrel{\rm LW}{\sim} D' \sqcup_e D''$. 

Since $D'$ is either good or empty, by Remarks~\ref{rs:Q-LW}(b) and (c) it suffices to show that $D''$ is good.  
There exists a sequence $D = D_1, \dots, D_k \in D_{\mathrm{inv}}$
of equivariant crossing changes and Lobb--Watson moves.
These do not increase the total number of crossings and can be chosen
so that the arcs of $D_k$ emanating upwards from $q_1$ and $p_2$
meet at their first intersection above the axis.
This can be achieved by performing equivariant crossing changes
so that these arcs overpass all intermediate strands
until they meet.
A symmetric argument applies below the axis.
Therefore, $D_k$ contains a ``positive'' clasp as on the left-hand side of Equation~\eqref{eq:PNskeincoerentclasp}, or a mirror (negative) version of it. In view of Remark~\ref{rs:Q-LW}.(d),  
we may assume without loss of generality that the clasp is positive. By Lemma~\ref{l:good-ccs}, it suffices to prove $D_k\in\gsid$. 
Let $D'$ and $D''$ be the diagrams appearing on the right-hand side of Skein Relation~\eqref{eq:PNskeincoerentclasp}. Clearly $n(D'),\, n(D'') \leq n(D_k) -1 \leq n_0-1$. 
Therefore, by the inductive hypothesis on $n_0$, $D',D''\in\gsid$ and by Lemma~\ref{l:clasp-skein} it follows that $D_k\in\gsid$, completing the inductive step.	
This shows that $\nfsid\subseteq\gsid$.
By Lemma~\ref{l:nfsid-good} it follows that $\sid = \gsid$,
and hence Equation~\eqref{eq:good} holds for every strongly involutive diagram.
\end{proof}

\begin{rmk}\label{r:algorithm}
Although we do not pursue computational aspects in this paper, the proof of
Theorem~\ref{t:gsid=sid} provides, in a natural way, an explicit recursive procedure for computing
the polynomial $Q(L)$ from a strongly involutive diagram of $L$. Indeed, the reduction
arguments show that any such diagram can be transformed, via equivariant skein relations
and crossing changes, into a finite combination of elementary diagrams whose $Q$--values
are fixed by normalization. In this sense, the uniqueness proof implicitly yields an
algorithm for the computation of $Q$.
\end{rmk}

\begin{cor}\label{c:mainQ}
The polynomial $Q(D)$ of Theorem~\ref{t:gsid=sid} satisfies Properties~\eqref{eq:skeinonaxis}--\eqref{eq:skeinmirror} and defines an invariant of strongly involutive links. 
\end{cor}

\begin{proof}
By Remark~\ref{r:skeinonoffaxis}, Skein Relations~\eqref{eq:skeinonaxis} 
and~\eqref{eq:skeinoffaxis} hold, while Skein 
Relation~\eqref{eq:skeinincoerentclasp} holds by Lemma~\ref{l:clasp-skein}. 
Properties~\eqref{eq:skeinnormalisation}, \eqref{eq:skeinunion} and~\eqref{eq:skeinmirror} follow from 
Equation~\eqref{eq:good} and, respectively, Skein Relations~\eqref{eq:PNskeinnormalisation}, 
\eqref{eq:PNskeinunion} and~\eqref{eq:PNskeinmirror}. This proves the first part of 
the statement. 

Let $D, D'$ be strongly involutive link diagrams such that $D\stackrel{\rm LW}{\sim} D'$. 
Since, by Theorem~\ref{t:gsid=sid}, Equation~\eqref{eq:good} holds for $D$ and $D'$, 
it follows that the Laurent polynomials $Q(D)$ and $Q(D')$ of Theorem~\ref{t:gsid=sid} satisfy 
\[
Q(D)|_{a=q^N,\ z=q^2-q^{-2}} = Q(D')|_{a=q^N,\ z=q^2-q^{-2}}
\]
for every $N\geq 2$. Then, by Remark~\ref{rs:Q-LW}(a), $Q(D) = Q(D')$. 
\end{proof}

\begin{defn}\label{d:P^e}
Given a strongly involutive diagram $D$, we define the Laurent polynomial $P^e(D)\in\bbZ[a^{\pm 1},q^{\pm 1}]$
by setting $P^e(D)(a,z) = z^{-m(D)} Q(D)(a,z)$. 
\end{defn}

\begin{proof}[Proof of Theorem~\ref{t:main}]
Let $D, D'$ be strongly involutive link diagrams such that $D\stackrel{\rm LW}{\sim} D'$. 
By Corollary~\ref{c:mainQ} we have $Q(D)=Q(D')$, and since $m(D)=m(D')$ it follows 
that $P^e(D) = P^e(D')$. 

By Corollary~\ref{c:mainQ}, Properties~\eqref{eq:skeinonaxis}--\eqref{eq:skeinmirror} hold
and therefore, by the definition of $P^e$, so do  
Properties~\eqref{eq:Pskeinonaxis}--\eqref{eq:Pskeinmirror}. 
Now let $\widetilde{P}^e$ be any polynomial invariant satisfying 
Properties~\eqref{eq:Pskeinonaxis}--\eqref{eq:Pskeinmirror}. Then, it is 
straightforward to check that, for each $N\geq 2$, the map $I_N:\sid\to\bbQ(q)$ given by 
\[
I_N(D) = (q-q^{-1})^{-\ell(D)} (q^N - q^{-N}) \widetilde{P}^e(D)\big|_{a=q^N,\ z=q^2-q^{-2}}
\]
is an equivariant skein invariant of type $(\bbQ(q),\iota,N)$, where $\iota:\Lambda\to\bbQ(q)$
is the canonical inclusion, and therefore it coincides with $\iota\circ P^e_N$ by 
Theorem~\ref{t:gsid=sid}. Note that, for a sufficiently large integer 
$M$, we have 
\[
P(a,z) = z^M P^e(D),\ \widetilde{P}(a,z) = z^M \widetilde{P}^e(D)\in\bbZ[a,a^{-1},z].
\]
The equality $I_N = \iota\circ P^e_N$ implies  
\[
P(a,z)\big|_{a=q^N,\ z=q^2-q^{-2}} = \widetilde{P}(a,z)\big|_{a=q^N,\ z=q^2-q^{-2}} 
\]
for each $D\in\sid$. As in Remark~\ref{rs:Q-LW}(a), applying Lemma~\ref{l:specialize-vanish} to 
the Laurent polynomials $P(a,z)\big|_{z=q^2-q^{-2}}$ and $\widetilde{P}(a,z)\big|_{z=q^2-q^{-2}}$
shows that $P(a,z)=\widetilde{P}(a,z)$ and therefore $P^e(D)=\widetilde{P}^e(D)$. 
This proves the uniqueness of $P^e$ and concludes the proof of the first part of the theorem. 

By Proposition~\ref{p:skeinLWCpoly}, $J^e$ is an equivariant skein invariant of type $(\Lambda,{\rm id}_\Lambda,2)$. 
Therefore, by Theorem~\ref{t:gsid=sid} we have $J^e=P_2^e$. The second part of the theorem follows 
at once from Equation~\eqref{eq:good} in the case $N=2$ and the defintion of $P^e$. 
\end{proof}

The following lemma will be used in the proof of Theorem~\ref{t:congruence}.

\begin{lem}\label{l:evalmod2}
Let $f, g\in\bbZ/2\bbZ[a^{\pm 1},z]$ such that 
\[
f(q^N, q+q^{-1}) = g(q^N,(q+q^{-1})^2)
\]
for each $N\geq 2$. Then, $f(a,z) = g(a,z^2)$. 
\end{lem}

\begin{proof}
Let $h(a,z):=g(a,z^2)$. Then, 
\[
f(q^N,q+q^{-1}) = h(q^N,q+q^{-1}) 
\]
for each $N\geq 2$. Then, the polynomial 
$L(a,q):=f(a,q+q^{-1}) - h(a,q+q^{-1})\in\bbZ/2\bbZ[a^{\pm 1},q^{\pm 1}]$ satisfies 
$L(q^N,q)=0$ for each $N\geq 2$, and by Lemma~\ref{l:specialize-vanish} we have $L(a,q)=0$. 
Since $q+q^{-1}$ is not a zero of a non-vanishing polinomial in $\bbZ/2\bbZ[a^{\pm 1}]$, 
we conclude 
\[
f(a,z) = h(a,z) = g(a,z^2).
\]
\end{proof}

\begin{proof}[Proof of Theorem~\ref{t:congruence}]
Let $D$ be an oriented link diagram and, for $N\geq 2$, let $P_N(D)\in\bbZ[q^{\pm 1/2}]$ the polynomial 
invariant defined in~\cite{MOY98}. Consider the following mod $2$ evaluation of $P_N$: 
\[
\widetilde P_N(D) := P_N(D)\big|_{q\, \mapsto\, q^2} \bmod 2 \in \bbZ/2\bbZ[q^{\pm 1}]
\]
We claim that $\widetilde P_N(D)$ is an equivariant skein invariant of type $(\bbZ/2\bbZ[q^{\pm 1}],\pi,N)$, 
where the map $\pi:\bbZ[q^{\pm 1}]\to\bbZ/2\bbZ[q^{\pm 1}]$ is the mod $2$ reduction. In view of the skein 
relation~\cite[Theorem~3.2]{MOY98} and the normalization of $P_N$ chosen in~\cite{MOY98} (see the definition of $\langle G\rangle_n$ and Lemma~2.1 therein), we have 
\begin{equation}\label{e:skein-mod2}
    q^N \widetilde P_N(D_-) + q^{-N} \widetilde P_N(D_+) = (q + q^{-1}) \widetilde P_N(D_0),\qquad  
    \widetilde P_N(U) = [N] \bmod 2, 
\end{equation}
where $U$ is a diagram of the unknot. Since $P_N$ is invariant under Reidemeister moves so is  
$\widetilde P_N(D)$, and therefore a fortiori invariant under the Lobb--Watson moves. 
Moreover, Equations~\eqref{e:skein-mod2} imply that $\widetilde P_N(D)$ satisfies 
Relations~\eqref{eq:PNskeinonaxis}--\eqref{eq:PNskeinmirror}. Indeed, Relation~\eqref{eq:PNskeinonaxis} is 
immediate, while \eqref{eq:PNskeinoffaxis} can be deduced as an easy exercise using the symmetry of the 
link. 

Switching and resolving the top crossing of the clasp on the left-hand side of~\eqref{eq:PNskeincoerentclasp} and 
using~\eqref{e:skein-mod2} we get  
\begin{equation}\label{e:claspequiv}
q^N \widetilde P_N(\claspRtwo) + q^{-N} \widetilde P_N(\Claspa) = (q + q^{-1}) \widetilde P_N(\claspresol).
\end{equation}
Since 
\[
\widetilde P_N(\claspRtwo) = \widetilde P_N(\orientedresolutiononaxis)\quad\text{and}\quad 
\widetilde P_N(\claspresol) = \widetilde P_N(\poscrossingonaxis),
\]
Equation~\eqref{e:claspequiv} is equivalent to~\eqref{eq:PNskeincoerentclasp}. Moreover, when $\varphi=\pi$ Relations~\eqref{eq:PNskeinnormalisation} become 
\[
I_N(\acirc{}) = [N]\bmod 2,\qquad I_N(\acircs{}) = [N]_{q^2} \bmod 2 =[N]^2 \bmod 2,
\]
which are clearly satisfied by $\widetilde P_N$. Finally, Relations~\eqref{eq:PNskeinunion} and~\eqref{eq:PNskeinmirror} follow easily. This proves the claim that $\widetilde P_N(D)$ is an equivariant skein invariant of type $(\bbZ/2\bbZ[q^{\pm 1}],\pi,N)$. 
Therefore, by Theorem~\ref{t:gsid=sid}, we have 
\[
\widetilde P_N(D) = P^e_N(D) \bmod 2 
\]
for every $N\geq 2$. Now observe that the following identity holds:
\begin{equation}\label{e:tildePN=homflymod2}
\widetilde P_N(D) = [N] P(D)\big|_{a=q^N,\, z=q+q^{-1}}\bmod 2,    
\end{equation}
where $P(D)$ is the standard HOMFLY--PT polynomial of the link represented by $D$, characterized by: 
\[
a^{-1}P(\poscrossing) - a P(\negcrossing) = z P(\orientedresol),\qquad P(\unknot) = 1.
\]
Identity~\eqref{e:tildePN=homflymod2} follows from Theorem~\ref{t:gsid=sid} together with the fact that 
\[
[N] P(D)\big|_{a=q^N,\, z=q+q^{-1}}\bmod 2
\] 
is also an equivariant skein invariant of type $(\bbZ/2\bbZ[q^{\pm 1}],\pi,N)$, 
which can be checked as we just did for $\widetilde P_N(D)$. 
By Theorem~\ref{t:gsid=sid}, Equation~\eqref{eq:good} holds as well, therefore  
\[
(q+q^{-1})^{|D|-1} [N] P(D)\big|_{a=q^N,\, z=q+q^{-1}} 
= [N] Q(D)\big|_{a=q^N,\, z=(q+q^{-1})^2} \bmod 2.
\]
Hence, by Lemma~\ref{l:evalmod2} and since $P^e(D) = z^{-m(D)} Q(D)$, we have 
\[
z^{-\ell(D)+1} P(D) 
= P^e(D)\big|_{z\mapsto z^2}\quad \bmod 2. 
\]
This concludes the proof. 
\end{proof}

\section{Comparison with the Lobb-Watson invariant}\label{s:comparison}
In this section we compare the invariant $Q$ with the Lobb--Watson invariant.
Our goal is to relate the behavior of $Q$ under a natural equivariant
connected sum construction to classical link invariants, and to use this
comparison to detect examples that are indistinguishable by previously
known equivariant invariants.

Define the strongly involutive link $L\#^v L^r$ as the connected sum of an oriented link $L$ with its reverse $L^r$ along a specified component $K$, endowed with the natural involution.
When it is useful to emphasise the choice of $K$, we will write $L\#^v_K L^r$ instead of $L\#^v L^r$.
Given a diagram $D$ of $L$ such that the subdiagram $D_K$, representing $K$, admits an unobstructed path $l$ to the $x$-axis, the link $L\#^v L^r$ can be represented by taking a symmetric copy of $D$ across that axis, reversing its orientation, and joining the two copies along $l\cup\tau(l)$, see Figure~\ref{f:verticalsum}.
Any diagram obtained in this way will be denoted by $D\#^v D^r$.

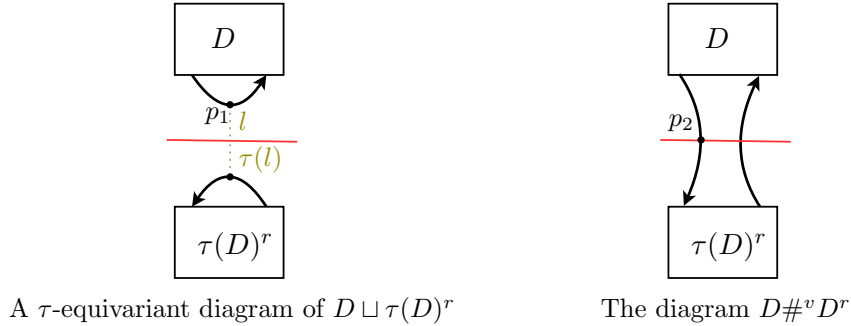
\begin{figure}[ht]
    \centering
\begin{subfigure}[b]{0.4\textwidth}
\centering
\begin{overpic}[abs,unit=1pt,scale=0.45]{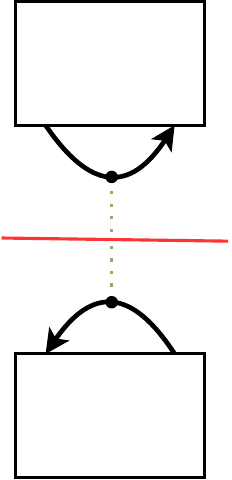}
\put(17,87){$D$}
\put(12,9.5){$\tau(D)^r$}
\put(27.5,56){\textcolor{olive}{\small $l$}}
\put(27.5,42.5){\textcolor{olive}{\small $\tau(l)$}}
\put(15,60){\small $p_1$}
\end{overpic}
\caption*{A $\tau$-equivariant diagram of $D\sqcup \tau(D)^r$}
\end{subfigure}
~~~
\begin{subfigure}[b]{0.4\textwidth}
\centering
\begin{overpic}[abs,unit=1pt,scale=0.45]{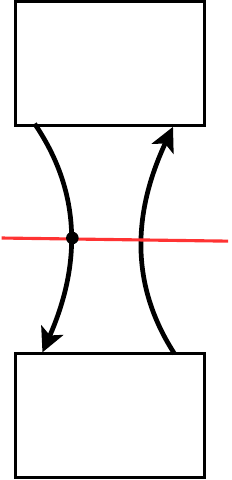}
\put(17,87){$D$}
\put(12,9.5){$\tau(D)^r$}
\put(3.5,57.5){\small $p_2$}
\end{overpic}
\caption*{The diagram $D\#^v D^r$}
\end{subfigure}
    \caption{The construction of $D\#^v D^r$.}
    \label{f:verticalsum}
\end{figure}

\begin{prop}\label{p:Pneofverticalsums}
For any link $L$ with a fixed component $K$, we have
\[
P_N^{e}(L\#^v_{K} L^r)(q) = [N]\, P_{N}(L)(q^2)\,,
\]
where $P_N$ denotes the $\mathfrak{sl}_N$ polynomial of $L$, normalised by $P_N(\bigcirc)(q) = 1$. 
In particular, for any link $L$, we have
\[
J^{e}({L\#^v_{K} L^r})(q) = (q+q^{-1}) J_L(q^2)\,,
\]
where $J_{L}$ denotes the Jones polynomial of $L$, normalised by $J_{\bigcirc}(q) = 1$.
\end{prop}

\begin{proof} 
Let $D$ be a diagram of $L$. 
Note that any equivariant crossing change or smoothing performed on a diagram of the form 
$D\#^vD^r$ again yields a diagram of the same form. 
We proceed by induction on
\[
\gamma(D) := c(D\#^vD^r) + u(D)\,,
\]
where $c(D)$ is the number of crossings of $D$ and $u(D)$ the minimal number of crossing changes 
turning $D$ into a diagram of the unlink. 

If $\gamma(D)=0$, then $D\#^vD^r$ is a crossingless diagram of the unlink.
The statement follows directly from Proposition~\ref{p:skeinPN},
using the normalization and multiplicativity properties of $P_N$ -- cf.~\cite[Theorem~8.4.1~and~\S~8.5]{Ka96}.

Assume the statement holds for each link $L'$ with a diagram $D'$ such that $\gamma(D')<n$. 
Choose a minimal collection $S$ of crossings of $D$ such that performing crossing changes on all of them 
results in a diagram of the unlink. Let $D_1$ and $D_2$ be the diagrams obtained by performing on $D$ a crossing change 
and an oriented smoothing, respectively, at one of the crossings in $S$. 
Note that $D_K$ remains a subdiagram of each $D_i$, and $\gamma(D) > \gamma(D_i)$ for $i=1,2$. 
Then, by Proposition~\ref{p:skeinPN} and the inductive hypothesis,
\begin{align*}
P_N^e(D\#^vD^r) 
  &= \pm q^{\pm 2N}\, P_N^e(D_1\#^vD_1^r) 
   \mp (q^2 + q^{-2})\, P_N^e(D_2\#^vD_2^r) \\[0.3em]
  &= [N]\bigl(\pm (q^2)^{\pm N} P_N(D_1)(q^2) 
               \mp [2]_{q^2} P_N(D_2)(q^2)\bigr) \\[0.3em]
  &= [N]\, P_N(L)(q^2).
\end{align*}
In the last equality we used the usual skein relations for the $\mathfrak{sl}_N$-polynomials 
(see, for example, \cite[Theorem~8.5.3]{Ka96}). 
\end{proof}

Now we restrict to the special case of a strongly involutive link $L$ with a distinguished strongly invertible component $K$.
In this setting, choosing one of the two fixed points of $K$ yields a reduced version $\widetilde{\rm Kh}_{\tau}^{*,*,*}(L,K)$ of the Lobb--Watson refined invariant. For the sake of completeness we start by recalling the construction here.

Let $D$ be a strongly invertible diagram of $L$. Let $p$ be a fixed base-point, chosen away from any crossing, on the subdiagram~$D_K \subseteq D$ representing $K$. The \emph{reduced Khovanov complex} in quantum grading $j$ is the subcomplex
\[
\widetilde{\rm CKh}^{\ast,j}(D,p) \subseteq {\rm CKh}^{\ast,j-1}(D)
\]
spanned by all enhanced states in which the circle containing $p$ is labelled $x^{-}$.

The reduced homology $\widetilde{\rm Kh}$ does not depend (up to isomorphism) on the choice of base-point, 
though it may depend on the component $K$ containing $p$. 
Since we are working over $\mathbb{F}_2$, there is a splitting
\[
{\rm Kh}^{i,j}(L) \cong \widetilde{\rm Kh}^{i,j-1}(L,K)\oplus\widetilde{\rm Kh}^{i,j+1}(L,K)\,,
\]
for every link $L$ and component $K$. 
One advantage of working with reduced Khovanov homology is that it behaves well under connected sums, 
as summarised by the following proposition:

\begin{prop}[{\cite[Prop.~3.3 and Cor.~3.5]{Kh02}}]\label{p:connectedsumKh}
Let $(L,K)$ and $(L',K')$ be two oriented links with distinguished components. Then
\[
\widetilde{\rm Kh}(L\#_{K,K'} L', K\#K') \cong 
\widetilde{\rm Kh}(L,K)\otimes_{\mathbb{F}_2}\widetilde{\rm Kh}(L',K')\,,
\]
as bi-graded $\mathbb{F}_2$-vector spaces.
\end{prop}

\medskip

Returning to the equivariant setting: if we chose our base-point $p\in K\subseteq L$ 
to be one of the two fixed points of $\tau$, the reduced sub-complex is also stable under ${\rm id} + \tau$. 
By repeating the construction recalled in Section~\ref{s:LW}, replacing ${\rm Kh}$ with $\widetilde{\rm Kh}$, one obtains a reduced version of the Lobb--Watson refined invariant $\widetilde{\rm Kh}_{\tau}^{*,*,*}(L,K)$.
This reduced version carries the same information as the unreduced one; indeed, there is a splitting
\[
{\rm Kh}_{\tau}^{i,j,k}(L) \cong 
\widetilde{\rm Kh}_{\tau}^{i,j-1,k}(L,K)\oplus\widetilde{\rm Kh}_{\tau}^{i,j+1,k}(L,K)\,,
\]
see~\cite[Proposition~1.5]{LW21}. 
Moreover, the same reasoning in \cite[Sections~4.2 and 4.3]{LW21} gives \emph{mutatis mutandis} an explicit description of $\widetilde{E}^2_{\mathcal{G}}$, the second page of the \emph{reduced} Lobb--Watson $\mathcal{G}$-spectral sequence. We provide here a quick sketch of the argument. The first page of the reduced $\mathcal{G}$-spectral sequence $\widetilde{E}^1_{\mathcal{G}}(D,p)$ is the homology of the reduced Khovanov complex under  the differential $({\rm id} + \tau)$. 
In analogy with the notation in Section~\ref{s:LW}, we denote by $\widetilde{\calS}(D,p)$ and~$\widetilde{\calS}^e(D,p)$ the collections of enhanched states and $\tau$-equivariant enhanched states, respectively, where the circle containing $p$ is labelled by $x_{-}$.
In particular, as $\mathbb{F}_2$-vector spaces we have:
\[ \widetilde{\rm CKh}(D,p) = {\rm Span}\, (\widetilde{\calS}(D,p))\,.\]
Fix a subset $\widetilde{S} \subset \widetilde{\calS}(D,p)$ in such a way that:
\[ \widetilde{\calS}(D,p) = \widetilde{\calS}^e(D,p)\, \sqcup\,  \widetilde{S} \sqcup\, \tau(\widetilde{S})\, . \]
One can proceed as in \cite[Section~4.2]{LW21}, and use Gaussian elimination to cancel out the pairs $(s,\tau(s))$ for each $s \in \widetilde{S}$. Alternatively, one can notice that the kernel of ${\rm id} + \tau$, is spanned by the union of $\widetilde{\calS}^e(D)$ and the set of all elements $x+\tau(x)$, for $x\in \widetilde{\calS}$. 
The latter elements are clearly boundaries since $({\rm id} + \tau)(x) = x+\tau(x)$, while there is no linear combination of elements in $\widetilde{\calS}^e(D)$ that is a boundary with respect to  ${\rm id} + \tau$.
Either way, we have an identification
\[ \widetilde{E}^2_{\mathcal{G}}(D,p) \cong {\rm Span}\, (\widetilde{\calS}^e(D,p))\,,\]
that preserves both the $k$- and $j$-gradings. 
This can be made into an ismorphism of chain complex by endowing ${\rm Span}\, (\widetilde{\calS}^e(D,p))$ with the differential $\tilde{d}$ given by the maps described at the end of \cite[Section~4.3]{LW21}. In the light of this identification, we will abuse the notation and do not distinguish between $({\rm Span}\, (\widetilde{\calS}^e(D,p)),\tilde{d}\,)$ and $\widetilde{E}^2_{\mathcal{G}}(D,p)$.

Recall that a link $L$ is \emph{Khovanov-thin} if its reduced Khovanov homology 
(with respect to any, and hence every, component) is supported in bi-degrees $(i,j)$ such that 
$2i-j = s$, for some fixed integer $s$. 
It turns out that this integer $s$ coincides with Rasmussen’s invariant $s(L)$. 
In particular, the Khovanov homology of thin links is completely determined by the Jones polynomial 
and the Rasmussen invariant.

\begin{prop}\label{p:E3Ghverticalsum}
Let $(L,K)$ be a Khovanov-thin link with a distinguished component, and write 
\[
J_{L}(q) = \sum_r a_r q^r
\]
for the Jones polynomial of $L$, normalised so that the unknot has Jones polynomial $1$. 
Then, the Poincaré polynomial of 
\[
\widetilde{E}^3_{\mathcal{G}} = \widetilde{E}^{3}_{\mathcal{G}}(L\#^v_K L^r,K\#K)
\]
is
\[
P(\widetilde{E}^3_\mathcal{G}) := \sum_{j,k} \dim (\widetilde{E}^3_{\mathcal{G}})^{j,k}\, q^{j}t^{k} 
= \sum_{r} \vert a_r \vert\, q^{2r} t^{\tfrac{1}{2}(r+s(L))}.
\]
In particular, the $\mathcal{G}$-spectral sequence collapses at the third page, 
which is therefore completely determined by the Jones polynomial and the Rasmussen invariant of $L$.
\end{prop}

\begin{proof}
Let $D$ be a diagram of $L$. We fix the base-points  $p_1\in D_K \subseteq D$ and $p_2\in D\#^v D^r$ as illustrated in Figure~\ref{f:verticalsum}. 

Since there is precisely a $\tau$-invariant pair of crossings in $D\#^vD^r$ for each crossing in $D$, there is a bijective correspondence between $\tau$-invariant resolutions of $D\#^v D^r$ and resolutions of $D$. 

Note that the vertical arcs in $D\#^v D^r$ always belong to the same circle in each equivariant resolution of $D\#^v D^r$ and, moreover, this is the unique $\tau$-invariant circle in each resolution of~$D\#^v D^r$. All circles not containing $p_1$ in a resolution of $D$ yield a $\tau$-invariant pair of of circles in the corrspoding resolution of $D\#^vD^r$, while the circle containing $p_1$ yields the circle containing~$p_2$.
This allows us to define a one-to-one correspondence
\[ \widetilde{\calS}^e(D\#^v D^r,p_2) \overset{1:1}{\longleftrightarrow} \widetilde{\calS}(D,p_1)\,.\] 
as follows; we label each circle that does not contain~$p_1$ as the corresponding $\tau$-invariant pair of circles, and the circle containing $p_1$ as the circle containing~$p_2$.

It is easy to check that this bijection sends $\tau$-equivariant states of bi-degree $(j,k)$ 
to states of bi-degree $(k,\tfrac{1}{2}j)$ (note $j$ is always even). 
This gives an isomorphism
\[
(\widetilde{E}^{2}_{\mathcal{G}})^{j,k}(D\#^v D,p_2) 
\cong \widetilde{\rm CKh}^{k,\tfrac{1}{2}j}(D,p_1)\,,
\]
of $\mathbb{F}_2$-vector spaces, for each $j,\,k\in\mathbb{Z}$ with $j$ even.
Comparing the Khovanov differential with that on $\widetilde{E}^{2}_{\mathcal{G}}$ 
(see \cite[Section~4.3]{LW21}), one notices that these isomorphisms give an isomorphism of chain complexes. Therefore, we have that 
\[(\widetilde{E}^{3}_{\mathcal{G}})^{j,k} \cong \widetilde{\rm Kh}^{k,\tfrac{1}{2}j}(L,K)\] 
for each $j,\,k\in\mathbb{Z}$ with $j$ even. 
Since $L$ is Khovanov-thin,
\[
\sum_{i,j} \dim \widetilde{\rm Kh}^{i,j}(L,K)\, h^i q^j 
= \sum_{j} \vert a_j \vert\, h^{\tfrac{1}{2}(j+s(L))} q^j.
\]
Differentials on pages $t>2$ increase the $k$-degree while preserving the quantum degree, 
so they must vanish. Therefore the spectral sequence collapses at page $3$ and the statement follows. 
\end{proof}

\begin{prop}\label{p:LWKhverticalsum}
Let $(L,K)$ be a Khovanov-thin link with a fixed component, and write 
\[
J_{L}(q) = \sum_r a_r q^r
\]
for the Jones polynomial of $L$, normalised so that the unknot has Jones polynomial $1$. 
Then both the $\mathcal{F}$- and $\mathcal{G}$-spectral sequences collapse at page~$3$, 
and the Poincaré polynomial of the reduced homology 
$\widetilde{\rm Kh}^{*,*,*}_\tau(L\#^v_K L^r,K\#K)$ is given by 
\[
\widetilde{P}(L,K) := 
\sum_{i,j,k} \dim \bigl(\widetilde{\rm Kh}_\tau^{i,j,k}(L\#^v_K L^r,K\#K)\bigr)\, h^i q^j t^k 
= \sum_{r} \vert a_r \vert\, h^{r+s(L)} q^{2r} t^{\tfrac{1}{2}(r+s(L))}.
\]

In particular, the refined invariant $\widetilde{\rm Kh}_\tau^{*,*,*}(L\#^v_K L^r,K\#K)$, 
as a triply graded vector space, together with all invariant pages of the 
$\mathcal{F}$- and $\mathcal{G}$-spectral sequences, 
are completely determined by the Jones polynomial and the Rasmussen invariant of $L$.
\end{prop}

\begin{proof}
The second page of the $\mathcal{F}$-spectral sequence converging to 
$\widetilde{\rm Kh}_\tau(L\#^v L^r,K\#K)$ is the group 
$\widetilde{\rm Kh}(L\#^v_K L^r,K\#K)$. 
The differential on page $t\geq 3$ decreases the homological degree by~$t-2$ 
while preserving the quantum degree (see~\cite[Sections~3.7 and~4.1]{LW21}). 

By Proposition~\ref{p:connectedsumKh}, the connected sum of thin links is thin. 
Since $L$ is thin, so is $L\#^v_K L^r$, and therefore, in each fixed quantum degree, 
the reduced Khovanov homology is supported in a single homological degree. 
It follows that the third page of the $\mathcal{F}$-spectral sequence is 
concentrated in a single homological degree in each quantum grading, and hence 
all higher differentials vanish. Thus the $\mathcal{F}$-spectral sequence 
collapses at the third page. 

In particular, $\widetilde{\rm Kh}_\tau^{i,j,*}(L\#^v_K L^r,K\#K)$ is nontrivial 
only if $2i-j = s(L\# L^r) = 2s(L)$.

To compute the Poincaré polynomial of 
$\widetilde{\rm Kh}^{*,*,*}_\tau(L\#^v_K L^r,K\#K)$, note that by 
Proposition~\ref{p:E3Ghverticalsum} the $\mathcal{G}$-spectral sequence 
also collapses at page~$3$. Therefore, for each quantum degree $2j$ there is 
a unique summand in $k$-degree $(j+s(L))/2$ of dimension $|a_j|$. 

Combining this with the constraint above on the homological grading, namely 
$2i-2j = 2s(L)$, we see that each such summand contributes a term 
$h^{j+s(L)} q^{2j} t^{(j+s(L))/2}$, and the stated formula follows.
\end{proof}

\begin{proof}[Proof of Theorem~\ref{t:comparison}]
Let $K_{n}$ be the connected sum of $2n$ copies of the figure-eight knot $4_1$, 
and let $K'_n$ be the connected sum of $n$ copies of the knot $8_9$. 
These knots are alternating and hence Khovanov-thin (see \cite[Theorem~1.2]{Le05}). 
Since both the Jones polynomial and the $\mathfrak{sl}_3$ polynomial $P_3$ 
are multiplicative under connected sums, being specialisations of the HOMFLY--PT polynomial, 
it follows that the Jones and $P_3$ polynomials of $K_1$ are the squares of those of the figure-eight knot. 
By consulting \cite{KI} and the table at the end of \cite{MOY98}, 
one sees that $K_1$ and $K'_1$ have the same Jones polynomial but are distinguished by $P_3$. 

Moreover, both knots are slice, so their Rasmussen invariant vanishes. 
Since the connected sum of slice knots is slice, the same holds for $K_n$ and $K'_n$ for all $n$. 
It follows that, for every $n$, the knots $K_n$ and $K'_n$ share the same Jones polynomial, 
are slice, and are distinguished by $P_3$. 

Therefore, by Proposition~\ref{p:LWKhverticalsum}, 
the strongly invertible knots $K_n\#^v K_n$ and $K'_n\#^v K'_n$ have the same refined Lobb--Watson invariants 
and isomorphic $\mathcal{F}$- and $\mathcal{G}$-spectral sequences. 
On the other hand, by Proposition~\ref{p:Pneofverticalsums}, 
they are distinguished by $P_3^e$, and hence, by Theorem~\ref{t:gsid=sid} and the definition of $P^e$, 
also by~$P^e$.
\end{proof} 
\appendix
\section{Proofs of Lemmas~\ref{l:bubbles}--\ref{l:M1identity}}\label{app:proofs}

In this appendix we provide the proofs of Lemmas~4.2--4.5 from Section~4, which establish
the local relations satisfied by the equivariant $N$--bracket.

\begin{proof}[Proof of Lemma~\ref{l:bubbles}]
To prove the first equality it suffices to repeat word-for-word the proof 
of~\cite[Lemmas~2.3 and~2.4]{MOY98}, the only difference being that our local graphs consist 
of two copies of the graphs in~\cite{MOY98}, which yield the evaluations of $[N-1]$ at $q^2$.  
 
The last two equalities have essentially the same proof, so we only provide the argument for the first one. 
Partition the states of the MOY graph $\Gamma$ on the left hand side according to 
the values they take on the thin edges in the local picture. Let $\mathcal{S}^e_{x,y}(\Gamma)$ be the collection of equivariant 
states behaving locally as shown in Figure~\ref{fig:eq_state_bubble}. Each $\sigma\in \mathcal{S}^e_{x,y}(\Gamma)$ determines a 
unique state $\sigma'$ of the graph $\Gamma'$ on the right hand side of the equation, which assigns $x$ to the local arc -- 
denote by $\mathcal{S}_{x}^{e}(\Gamma')\subseteq \mathcal{S}^e_N(\Gamma')$ such states. 

\begin{figure}[ht]
\centering
\includegraphics[width=0.25\linewidth]{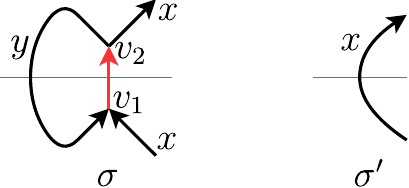}
\caption{The local assignment of an equivariant state $\sigma \in \mathcal{S}^e_{x,y}(\Gamma)$ (left) and of the corresponding $\sigma'\in\mathcal{S}^e_x(\Gamma')$ (right).}
\label{fig:eq_state_bubble}
\end{figure}

To conclude the proof it suffices to compare the two sides of the equation. Using the easily verified equality 
$\sum_{y\in X_{N}\setminus\{x\}} q^{\sgn(x-y) + y} = [N-1]$ we get
\begin{align*}
\langle\, \Roneedgeresol \rangle_N^e & = \sum_{x\in X_N}\, \sum_{y\in X_{N}\setminus\{x\}} \sum_{\sigma \in \mathcal{S}^e_{x,y}(\Gamma)} \left(\prod_{v\in V(\Ga)}\wt(v,\si)\right)q^{\rot(\si)} = \\
& = \sum_{x\in X_N}\, \sum_{y\in X_{N}\setminus\{x\}} q^{\sgn(x-y) + y} \sum_{\sigma' \in \mathcal{S}^e_{x}(\Gamma)} \left(\prod_{v\in V(\Ga')}\wt(v,\si')\right)q^{\rot(\si')} = \\
& = [N-1] \sum_{x\in X_N}\, \sum_{\sigma' \in \mathcal{S}^e_{x}(\Gamma')} \left(\prod_{v\in V(\Ga')}\wt(v,\si')\right)q^{\rot(\si')} = \\
& = [N-1]\langle\, \arcuponaxis \rangle_N^e
\end{align*}
\end{proof}

\begin{proof}[Proof of Lemma~\ref{l:edgesbubbles}]
As in Lemma~\ref{l:bubbles}, the proof of the first equality is a word-for-word repetition of the proof 
of~\cite[Lemma~2.2]{MOY98}, except for the evaluation at $q^2$  which is due to the doubling of the local graphs.
The second equality can be also proven similarly. 
We concentrate on the proof of the rightmost, and last, equality.
Denote by $\Gamma$ and $\Gamma'$, respectively, the equivariant MOY graphs on the right, respectively on the 
left hand side of the equation. 
There is a two-to-one correspondence between the states of $\Gamma$ and the states of $\Gamma'$, which is illustrated in Figure~\ref{fig:eq_state_doubledigon}.

\begin{figure}
\centering
\includegraphics[width=0.35\linewidth]{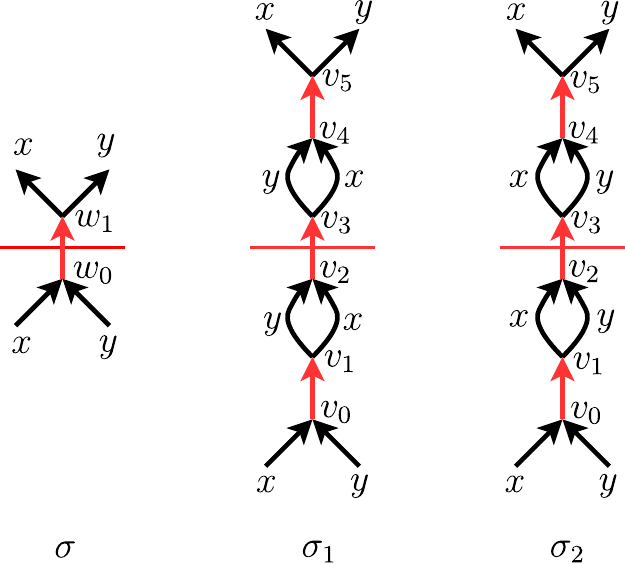}
\caption{An equivariant state $\sigma \in \mathcal{S}^e(\Gamma)$ (left) and the corresponding two equivariant states $\sigma_1$ (center) and $\sigma_2$ (right) in $\mathcal{S}^e(\Gamma')$.}
\label{fig:eq_state_doubledigon}
\end{figure}

Note that the states $\sigma_1, \sigma_2\in \mathcal{S}^e(\Gamma')$ and the corresponding state $\sigma\in \mathcal{S}^e(\Gamma)$ have the same rotation numbers. Moreover, a simple computation shows that
\[ 
\prod_{i=0}^{5} \wt(v_i,\sigma_1) =  q^{{\rm sgn}(x - y)} 
= q^{-2{\rm sgn}(y - x)}\cdot \wt(w_0,\sigma)\cdot \wt(w_1,\sigma)
\]
and
\[ 
\prod_{i=0}^{5} \wt(v_i,\sigma_2) =  q^{3{\rm sgn}(y - x)} 
= q^{2{\rm sgn}(y - x)}\cdot \wt(w_0,\sigma)\cdot \wt(w_1,\sigma)\,. 
\]
Therefore, we have 
\[
\left(\prod_{v\in V(\Ga')}\wt(v,\si_1)\right)q^{\rot(\si_1)} 
+ \left(\prod_{v\in V(\Ga')}\wt(v,\si_2)\right)q^{\rot(\si_2)} 
= [2]_{q^2}\left(\prod_{v\in V(\Ga)}\wt(v,\si)\right)q^{\rot(\si)}\]
and the desired equality follows.
\end{proof}

\begin{proof}[Proof of Lemma~\ref{l:square}]
We give the detail of the proof only for the first equality, the proofs of the second and the third equations, being similar, will be left to the reader. 

Denote by $\Gamma$ the strongly involutive MOY graph  on the left hand side of the equality.
We start by partitioning the equivariant states in $\mathcal{S}^e_N(\Gamma)$ into subsets $\mathcal{S}_{a,b,c}$, with $a,b,c\in X_N$ and~$a,c\neq b$, according to the labels on the boundary of the local picture -- as shown in Figure~\ref{f:squarestates}.(\subref{f:squarestate}). 
\begin{figure}[ht]
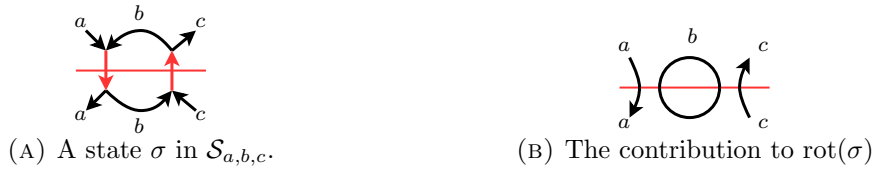

\centering
\begin{subfigure}[t]{0.45\textwidth}
\centering
\begin{overpic}[scale=0.5,tics=10]{figurepdf/Rtwodoubleedgeresolution}
\put(0,70){$\scriptstyle a$}
\put(91,70){$\scriptstyle c$}
\put(0,0){$\scriptstyle a$}
\put(91,0){$\scriptstyle c$}
\put(45,75){$\scriptstyle b$}
\put(45,-10){$\scriptstyle b$}
\end{overpic}    
\caption{A state $\sigma$ in $\mathcal{S}_{a,b,c}$.}
\label{f:squarestate}
\end{subfigure}
~
\begin{subfigure}[t]{0.45\textwidth}
\centering
\begin{overpic}[scale=0.5,tics=10]{figurepdf/Rtwodoubleresolution}
\put(0,45){$\scriptstyle a$}
\put(91,45){$\scriptstyle c$}
\put(0,-5){$\scriptstyle a$}
\put(91,-5){$\scriptstyle c$}
\put(45,50){$\scriptstyle b$}
\end{overpic} 
\caption{The contribution to $\rot(\sigma)$}
{\label{f:squarecircle}}
\end{subfigure}
\caption{The local labelling of a state of $\Gamma$ (\subref{f:squarestate}), 
and the associated circles (\subref{f:squarecircle}).}
\label{f:squarestates}
\end{figure}

Similarly, denote by $\Gamma'$  and $\Gamma''$ the strongly involutive MOY graph on first and second summand, respectively, on the right hand side of the equality.
We can partition the equivariant states in $\mathcal{S}^e_N(\Gamma')$ and $\mathcal{S}^e_N(\Gamma'')$ into subsets $\mathcal{S}'_{a}$ and $\mathcal{S}''_{a,c}$, respectively, where $a,c\in X_N$ -- as in Figure~\ref{f:squareresstates}. Note that, if $a\neq c$, the set $\mathcal{S}''_{a,c}$ can be empty -- e.g.~if the two local arcs belong to the same circle.

\begin{figure}[ht]
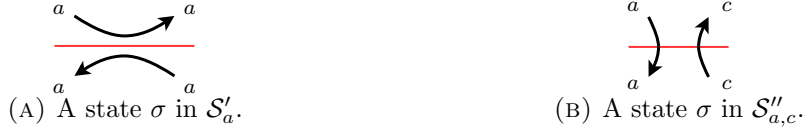

\centering
\begin{subfigure}[t]{0.45\textwidth}
\centering
\begin{overpic}[scale=0.5,tics=10]{figurepdf/Rtwonocrossings}
\put(0,45){$\scriptstyle a$}
\put(91,45){$\scriptstyle a$}
\put(0,-10){$\scriptstyle a$}
\put(91,-10){$\scriptstyle a$}
\end{overpic}    
\caption{A state $\sigma$ in $\mathcal{S}'_{a}$.}
\label{f:squarestate2}
\end{subfigure}
~
\begin{subfigure}[t]{0.45\textwidth}
\centering
\begin{overpic}[scale=0.5,tics=10]{figurepdf/arcsdownuponaxis}
\put(0,65){$\scriptstyle a$}
\put(91,65){$\scriptstyle c$}
\put(0,-10){$\scriptstyle a$}
\put(91,-10){$\scriptstyle c$}
\end{overpic}    
\caption{A state $\sigma$ in $\mathcal{S}''_{a,c}$.}
\label{f:squarestate3}
\end{subfigure}
\caption{The local labelling in the states of $\Gamma'$ and $\Gamma''$.}
\label{f:squareresstates}
\end{figure}

There is an obvious bijection, say $\varphi$, between states in $\mathcal{S}''_{a,a}$ and states in $\mathcal{S}'_a$. An easy check shows that $\rot(\sigma) =\rot(\varphi(\sigma)) - a$ for each sigma $\sigma \in \mathcal{S}''_{a,a}$.

We are now ready to prove the statement:
\begin{align*}
\langle\Rtwodoubleedgeresolution\rangle^e_N  & =   
\sum_{\scriptstyle \begin{array}{c} a, b, c\in X_{N}\\ b\neq a,c\end{array}}
\sum_{\si\in\calS_{a,b,c}}\left(\prod_{v\in V(\Ga)}\wt(v,\si)\right)q^{\rot(\si)} \\
 & =  \sum_{a, c\in X_{N}}  
\sum_{\scriptstyle \begin{array}{c} b\in X_{N}\\ b\neq a,c\end{array}}
\sum_{\si''\in\calS''_{a,c}} q^{b + \sgn(a-b)+\sgn(c-b)}\left(\prod_{v\in V(\Ga'')}\wt(v,\si'')\right)q^{\rot(\si'')} \\
 & \overset{(\star)}{=}   
 \sum_{a, c\in X_{N}} ([N-2] + \delta_{a,c} q^{a})\sum_{\si''\in\calS''_{a,c}}\left(\prod_{v\in V(\Ga'')}\wt(v,\si'')\right)q^{\rot(\si'')}
\end{align*}
where $\delta_{i,j}$ is Kronecker's delta, and the equality marked by ($\star$) follows from the formula:
\[ 
\sum_{x\in X_{N}\setminus \{ y,z\}} q^{x + \sgn(y-x) + \sgn(z-x) 
 } = [N-2] +  \delta_{y,z}q^y, \qquad \forall y,z\in X_N.
\]
We can now resume our computation:
\begin{align*}
\phantom{\langle\Rtwodoubleedgeresolution\rangle^e_N}  & = [N-2] \langle\arcsdownuponaxis \rangle^e_N + \sum_{a\in X_{N}}  q^{a}\sum_{\si'\in\calS'_{a}}\left(\prod_{v\in V(\Ga')}\wt(v,\si')\right)q^{\rot(\varphi^{-1}(\si'))} \\
& = [N-2] \langle\arcsdownuponaxis \rangle^e_N + \langle\Rtwonocrossings\rangle^e_N  \\
\end{align*}
as desired.
\end{proof}

\begin{proof}[Proof of Lemma~\ref{l:M1identity}]
We prove the first equality; the second, being analogous, is left to the reader.
The equivariant states on both sides of the equality can be grouped into three classes—labeled I through III—as shown 
in Table \ref{tab:statesIR3}, where we use the notation $\ep_{x,y} = 1 - \de_{x,y}$, with $\de_{x,y}$ equal to 
Kronecker's delta. To establish the claim, it suffices to show that, for each choice of $a,b,c\in X_N$, 
the total contribution to the state sum from the states in the left column equals that from the corresponding states 
in the right column. Since the rotational contribution is identical across all local configurations, it is enough to consider only the vertex weights.  
\begin{table}[ht]
\centering
\resizebox{13.5cm}{!}{
\input{tabella} 
}
\caption{States}
\label{tab:statesIR3}
\end{table}   
The difference between the LHS and the RHS contributions is given by 
\begin{equation}\label{e:contribs-diff}
\begin{split}
\ep_{b,c} q^{\sgn(c-b)} - \ep_{a,b} q^{\sgn(b-a)} + \ep_{a,c}(\ep_{a,b}-\ep_{b,c}) q^{\sgn(c-a)}\\
+\ep_{a,b}\ep_{b,c} q^{\sgn(b-a)} q^{\sgn(c-b)} (q^{\sgn(b-a)} - q^{\sgn(c-b)})\,.
\end{split}
\end{equation}
When $a,b$ and $c$ are pairwise distinct, Expression~\eqref{e:contribs-diff} becomes
\[
q^{\sgn(c-b)} - q^{\sgn(b-a)} 
+ q^{\sgn(b-a)} q^{\sgn(c-b)} (q^{\sgn(b-a)} - q^{\sgn(c-b)})\,.
\]
Setting $\eta_1 = \sgn(b-a)$ and $\eta_2 = \sgn(c-b)$, the latter expression can be written as 
\[
(q^{\eta_2} - q^{\eta_1}) (1 - q^{\eta_1} q^{\eta_2})\,,
\]
which vanishes for every $\eta_1,\eta_2\in\{\pm 1\}$. 
When $a=b$, Expression~\eqref{e:contribs-diff} becomes  
\[
\ep_{b,c} (q^{\sgn(c-b)} - q^{\sgn(b-c)} + q^{\sgn(b-c)} - q^{\sgn(c-b)}) = 0\,.
\]
When $a=c$, Expression~\eqref{e:contribs-diff} reduces to 
\[
\de_{b,c} (q^{\sgn(b-c)} - q^{\sgn(c-b)}) = 0\,, 
\]
and when $b=c$ to 
\[
- \ep_{a,b} q^{\sgn(b-a)} + \ep_{a,b} q^{\sgn(b-a)} = 0\,.
\]
\end{proof}

\section{\texorpdfstring{$P^e$}{Pe} polynomials of low-crossing strongly invertible knots}\label{app:tablePe}

In this appendix we list the values of $P^e$ for low--crossing strongly invertible prime knots. 
As remarked by Lobb and Watson in Section~6.5 of \cite{LW21}, knots with less than $8$ crossings are $2$--bridge knots, and hence each admits a pair of strong inversions unless it is a torus knot -- in which case it admits only one strong inversion.
For a knot $K$, the notation $K^L$ and $K^R$ refers to the strongly invertible
knots determined by the left-- and right--hand diagrams, respectively,
appearing in the table 
of Section~6.5 of \cite{LW21}.
The remaining knots, that is $3_1$, $5_1$, and $7_1$, are (positive) torus knots and we list the polynomial corresponding to their unique strong inversion. 
Evaluating $P^e$ at $a=q^2$ and $z=q^2-q^{-2}$ recovers the polynomial $J^e$, in accordance with
Theorem~\ref{t:main}.
With the exception of the pair $4_1^L$ and $4_1^R$, these specializations agree
with the values computed in~\cite{LW21}.
In that case, the discrepancy suggests that the left-- and right--hand diagrams
may have been interchanged in~\cite{LW21}, a conclusion also consistent with
the support constraint of Lemma~6.1 therein.

In \cite[\S 3.2]{Co08}, Couture lists the values of the polynomial $W$ for low--crossing strongly invertible prime knots. Other than for the strongly invertible knots of type $7_6$ and $7_7$, which are the mirror images of the corresponding strongly invertible knots in the table below, Couture uses the same conventions as we do; this can be checked, for instance, by computing the Jones polynomial modulo~$2$. In light of this observation, the table in \cite[\S 3.2]{Co08} and the table below are consistent with the formula
\begin{equation}
\label{e:PeandW}
P^e(K)|_{a = t,\, z = t - t^{-1}} = W_K(t)\, .
\end{equation}
This is in accordance with \cite[Remark~6.4]{LW21}, which states that $W$ and $J^e$ are equivalent.
\begin{longtable}{@{} l l @{\qquad} l l @{} }
\toprule
\endfirsthead
\toprule
\endhead
\bottomrule
\endfoot
\bottomrule
\endlastfoot
$3_{1}$ & $- a^2z + a^4$ &   &  \\[0.6em]

$4_{1}^{L}$ & $z + a^{-2} + 1 - a^{2}$ & $4_{1}^{R}$ & $-z - a^{-2} + 1 + a^{2}$ \\[0.6em]

$5_{1}$ & $a^4z^2 - a^6z + a^4$ &   &  \\[0.6em]

$5_{2}^{L}$ & $(-a^{2} - a^{4})z - a^{2} + a^{4} + a^{6}$ & $5_{2}^{R}$ & $(-a^{2} + a^{4})z + a^{2} + a^{4} - a^{6}$ \\[0.6em]

$6_{1}^{L}$ & $(-1 - a^{2})z - a^{-2} + a^{2} + a^{4}$ & $6_{1}^{R}$ & $(1 - a^{2})z + a^{-2} - a^{2} + a^{4}$ \\[0.6em]

$6_{2}^{L}$ & $a^{2}z^{2} + (1 - a^{2} - a^{4})z + a^{4}$ & $6_{2}^{R}$ & $-a^{2}z^{2} + (-1 - a^{2} + a^{4})z + a^{4}$ \\[0.6em]

$6_{3}^{L}$ & $-z^{2} + (-a^{-2} + 1 + a^{2})z + a^{-2} + 1 - a^{2}$ & $6_{3}^{R}$ & $-z^{2} + (-a^{-2} - 1 + a^{2})z - a^{-2} + 1 + a^{2}$ \\[0.6em]

$7_{1}$ & $ - a^{6} z^{3} + a^{8} z^{2} - 2 a^{6} z + a^{8}$ &   &  \\[0.6em]

$7_{2}^{L}$ & $(-a^{2} - a^{4} - a^{6})z - a^{2} + a^{6} + a^{8}$ & $7_{2}^{R}$ & $(-a^{2} + a^{4} - a^{6})z + a^{2} - a^{6} + a^{8}$ \\[0.6em]

$7_{3}^{L}$ & $(a^{4} + a^{6})z^{2} + (a^{4} - a^{6} - a^{8})z + a^{4}$ & $7_{3}^{R}$ & $(a^{4} - a^{6})z^{2} + (-a^{4} - a^{6} + a^{8})z + a^{4}$ \\[0.6em]

$7_{4}^{L}$ & $(-a^{2} + a^{6})z + 2a^{4} - a^{8}$ & $7_{4}^{R}$ & $(-a^{2} - a^{6})z + a^{8}$ \\[0.6em]

$7_{5}^{L}$ & $(a^{4} + a^{6})z^{2} + (a^{4} - 2a^{6} - a^{8})z + a^{8}$ & $7_{5}^{R}$ & $(a^{4} - a^{6})z^{2} + (-a^{4} - 2a^{6} + a^{8})z + a^{8}$ \\[0.6em]

$7_{6}^{L}$ & $a^{-2}z^{2} + (2a^{-4} - 1)z + a^{-6} - a^{-2} + 1$ & $7_{6}^{R}$ & $-a^{-2}z^{2} + (-2a^{-4} + 1)z - a^{-6} + a^{-2} + 1$ \\[0.6em]

$7_{7}^{L}$ & $-z^{2} + (-2a^{-2} + a^{2})z - a^{-4} + 2$ & $7_{7}^{R}$ & $-z^{2} + a^{2}z + a^{-4}$ \\[0.6em]

\end{longtable}

\printbibliography
\Addresses
\end{document}

%% file: tabella.tex
\begin{tabular}{|C{0.5cm}|C{2.2cm} C{2.7cm} C{2.2cm}|C{2.2cm} C{2.7cm} C{2.2cm}|}
			\hline
			     & \multicolumn{3}{c|}{\textbf{LHS}} & \multicolumn{3}{c|}{\textbf{RHS}}\\
			\hline
			 & 	local state &  $\prod_{i} {\rm wt} (v_i)$ & ${\rm rot} (\sigma)$  
			 & local state &  $\prod_{i} {\rm wt} (v_i)$ & ${\rm rot} (\sigma)$ \\
			
			\hline
			
			{I} & 
			\begin{tikzpicture}[scale=0.3, thick, inner sep=0.4mm]
			\draw[line width =1, -latex] (-29,20) node[below] {$a$} -- (-29,27) node[above] {$a$};
			
			\draw[red, thin] (-29.5, 23.5) -- (-23.5,23.5);
			\node (v0) at (-26,22) {};
			\node[above right] at (-26,22) {$v_0$};
			\draw[fill] (v0) circle (.1);
			
			\node (v1) at (-26,25) {};
			\node[below right] at (-26,25) {$v_1$};
			\draw[fill] (v1) circle (.1);
			
			\draw[line width =1, -latex] (-28,20)  node[below] {$b$} .. controls +(0,1) and +(-.75,-.5) .. (v0);
			\draw[line width =1, -latex] (-24,20)  node[below] {$c$} .. controls +(0,1) and +(.75,-.5) .. (v0);
			
			\draw[line width =1, latex-] (-28,27)  node[above] {$b$} .. controls +(0,-1) and +(-.75,.5) .. (v1);
			\draw[line width =1, latex-] (-24,27)  node[above] {$c$} .. controls +(0,-1) and +(.75,.5) .. (v1);
			
			\draw[line width =1.25, red, -latex] (v0) -- (v1); 
			\end{tikzpicture}	&
		    $\ep_{b,c} q^{\sgn(c-b)}$ & 
		    \begin{tikzpicture}[scale=0.3]
		    \draw[line width =1.5, blue, -latex]  (-29,20) node[below] {$a$} -- (-29,27) node[above] {$a$};
		    
		    \draw[red, thin] (-29.5, 23.5) -- (-23.5,23.5); 
		   
		    \draw[line width =1.5, orange, -latex]  (-28,20)  node[below] {$b$} .. controls +(0,1) and +(-.75,-.5) .. (-26.25,22) -- (-26.25,25).. controls +(-.75,.5) and +(0,-1) .. (-28,27) node[above] {$b$};
		    \draw[line width =1, -latex] (-24,20)  node[below] {$c$} .. controls +(0,1) and +(.75,-.5) .. (-25.75,22) -- (-25.75,25) .. controls +(.75,.5) and  +(0,-1) ..  (-24,27)node[above] {$c$};
		    \end{tikzpicture}  &
	        \begin{tikzpicture}[scale=0.3, thick, inner sep=0.4mm,rotate=180,yscale=-1] 
	       \draw[line width =1, -latex] (-29,20) node[below] {$c$} -- (-29,27) node[above] {$c$};
	       
	       \draw[red, thin] (-29.5, 23.5) -- (-23.5,23.5);
	       \node (v0) at (-26,22) {};
	       \node[above right] at (-26,22) {$v_0$};
	       \draw[fill] (v0) circle (.1);
	       
	       \node (v1) at (-26,25) {};
	       \node[below right] at (-26,25) {$v_1$};
	       \draw[fill] (v1) circle (.1);
	       
	       \draw[line width =1, -latex] (-28,20)  node[below] {$b$} .. controls +(0,1) and +(-.75,-.5) .. (v0);
	       \draw[line width =1, -latex] (-24,20)  node[below] {$a$} .. controls +(0,1) and +(.75,-.5) .. (v0);

	       \draw[line width =1, latex-] (-28,27)  node[above] {$b$} .. controls +(0,-1) and +(-.75,.5) .. (v1);
	       \draw[line width =1, latex-] (-24,27)  node[above] {$a$} .. controls +(0,-1) and +(.75,.5) .. (v1);

	       \draw[line width =1.25, red, -latex] (v0) -- (v1);
	       \end{tikzpicture}	& 
	        $\ep_{a,b} q^{\sgn(b-a)}$ & 
	        \begin{tikzpicture}[scale=0.3, thick, inner sep=0.4mm,rotate=180,yscale=-1] 
	        \draw[line width =1.5, -latex]  (-29,20) node[below] {$c$} -- (-29,27) node[above] {$c$};
	        
	        \draw[red, thin] (-29.5, 23.5) -- (-23.5,23.5); 
	        
	        \draw[line width =1.5, orange, -latex]  (-28,20)  node[below] {$b$} .. controls +(0,1) and +(-.75,-.5) .. (-26.25,22) -- (-26.25,25).. controls +(-.75,.5) and +(0,-1) .. (-28,27) node[above] {$b$};
	        \draw[line width =1.5,blue, -latex] (-24,20)  node[below] {$a$} .. controls +(0,1) and +(.75,-.5) .. (-25.75,22) -- (-25.75,25) .. controls +(.75,.5) and  +(0,-1) ..  (-24,27)node[above] {$a$};
        	\end{tikzpicture} \\
			\hline
			
			{II} & 
			\begin{tikzpicture}[scale=0.35, thick, inner sep=0.4mm]  
			 \draw[red, thin] (-29.5, 8.5) -- (-23.5,8.5); 
			
			\node (v5) at (-28,13.5) {};
			\node[below left] at (-28,13.5) {$v_5$};
			\draw[fill] (v5) circle (.1);

			\node (v4) at (-28,11) {};
			\node[below left] at (-28,11) {$v_4$};
			\draw[fill] (v4) circle (.1);

			\node (v1) at (-28,6.5) {};
			\node[above left] at (-28,6.5) {$v_1$};
			\draw[fill] (v1) circle (.1);

			\node (v0) at (-28,4) {};
			\node[above left] at (-28,4) {$v_0$};
			\draw[fill] (v0) circle (.1);

			\node (v3) at (-25,9.5) {};
			\node[above right] at (-25,9.5) {$v_3$};
			\draw[fill] (v3) circle (.1);

			\node (v2) at (-25,7.5) {};
			\node[below right] at (-25,7.5) {$v_2$};
			\draw[fill] (v2) circle (.1);
			
			\draw[line width =1.25, red, -latex] (v0) -- (v1);
			
			\draw[line width =1.25, red, -latex] (v4) -- (v5);
			
			\draw[line width =1.25, red, -latex] (v2) -- (v3);

			\draw[line width =1, -latex] (v1) -- (v2) node[midway,below]{$a$};
			\draw[line width =1, -latex] (v1) -- (v4) node[midway,left]{$b$};

			\draw[line width =1, -latex] (v3) -- (v4) node[midway,above]{$a$};
			\draw[line width =1, -latex] (v3) -- (-25,15) node[above] {$c$};
			
			\draw[line width =1, -latex] (-25,2.5) node[below] {$c$} -- (v2);
			
			\draw[line width =1, -latex] (-27,2.5) node[below] {$b$} -- (v0);
			\draw[line width =1, -latex] (-29,2.5) node[below] {$a$} -- (v0);

			\draw[line width =1, -latex] (v5) -- (-27,15) node[above] {$b$};
			\draw[line width =1, -latex] (v5) -- (-29,15) node[above] {$a$};
			\end{tikzpicture} & 
			 $\ep_{a,b} \ep_{a,c} q^{\sgn(c-a)}$  & 
			\begin{tikzpicture}[scale=0.3, thick, inner sep=0.4mm]  
	        	\draw[red, thin] (-29.5, 8.5) -- (-23.5,8.5); 
				\draw[line width =1.5, -latex] (-25,2.5) node[below] {$c$} -- (-25,15) node[above] {$c$};
				\draw[line width =1.5, orange, -latex] (-27.5,2.5) node[below] {$b$} -- (-27.5,4) --  (-27.5,13.5) -- (-27.5,15) node[above] {$b$};
				\draw[line width =1.5, blue, -latex] (-28.5,2.5) node[below] {$a$} -- (-28.5,4) -- (-28.5,6.5) --  (-25.5,7.5) --  (-25.5,9.5)  -- (-28.5,11) -- (-28.5,13.5) -- (-28.5,15) node[above] {$a$};
			\end{tikzpicture} & 
			\begin{tikzpicture}[scale=0.3, thick, inner sep=0.4mm,rotate=180,yscale=-1]  
			\draw[red, thin] (-29.5, 8.5) -- (-23.5,8.5); 
			
			\node (v5) at (-28,13.5) {};
			\node[below right] at (-28,13.5) {$v_5$};
			\draw[fill] (v5) circle (.1);

			\node (v4) at (-28,11) {};
			\node[below right] at (-28,11) {$v_4$};
			\draw[fill] (v4) circle (.1);

			\node (v1) at (-28,6.5) {};
			\node[above right] at (-28,6.5) {$v_1$};
			\draw[fill] (v1) circle (.1);

			\node (v0) at (-28,4) {};
			\node[above right] at (-28,4) {$v_0$};
			\draw[fill] (v0) circle (.1);

			\node (v3) at (-25,9.5) {};
			\node[above left] at (-25,9.5) {$v_3$};
			\draw[fill] (v3) circle (.1);

			\node (v2) at (-25,7.5) {};
			\node[below left] at (-25,7.5) {$v_2$};
			\draw[fill] (v2) circle (.1);
			
			\draw[line width =1.25, red, -latex] (v0) -- (v1);
			
			\draw[line width =1.25, red, -latex] (v4) -- (v5);
			
			\draw[line width =1.25, red, -latex] (v2) -- (v3);

			\draw[line width =1, -latex] (v1) -- (v2) node[midway,below]{$c$};
			\draw[line width =1, -latex] (v1) -- (v4) node[midway,right]{$b$};

			\draw[line width =1, -latex] (v3) -- (v4) node[midway,above]{$c$};
			\draw[line width =1, -latex] (v3) -- (-25,15) node[above] {$a$};
			
			\draw[line width =1, -latex] (-25,2.5) node[below] {$a$} -- (v2);
			
			\draw[line width =1, -latex] (-27,2.5) node[below] {$b$} -- (v0);
			\draw[line width =1, -latex] (-29,2.5) node[below] {$c$} -- (v0);

			\draw[line width =1, -latex] (v5) -- (-27,15) node[above] {$b$};
			\draw[line width =1, -latex] (v5) -- (-29,15) node[above] {$c$};
			\end{tikzpicture} & 
			$\ep_{a,c} \ep_{b,c} q^{\sgn(c-a)}$ &
			\begin{tikzpicture}[scale=0.3, thick, inner sep=0.4mm,,rotate=180,yscale=-1]  
			\draw[red, thin] (-29.5, 8.5) -- (-23.5,8.5);

			\draw[line width =1.5,blue, -latex] (-25,2.5) node[below] {$a$} -- (-25,15) node[above] {$a$};
			
			\draw[line width =1.5, orange, -latex] (-27.5,2.5) node[below] {$b$} -- (-27.5,4) --  (-27.5,13.5) -- (-27.5,15) node[above] {$b$};
			
			\draw[line width =1.5, -latex] (-28.5,2.5) node[below] {$c$} -- (-28.5,4) -- (-28.5,6.5) --  (-25.5,7.5) --  (-25.5,9.5)  -- (-28.5,11) -- (-28.5,13.5) -- (-28.5,15) node[above] {$c$};
			\end{tikzpicture} \\
			\hline
			{III} & 
			\begin{tikzpicture}[scale=0.3,thick, inner sep=0.4mm]  
			\draw[red, thin] (-29.5, 8.5) -- (-23.5,8.5); 
			
			\node (v5) at (-28,13.5) {};
			\node[below left] at (-28,13.5) {$v_5$};
			\draw[fill] (v5) circle (.1);

			\node (v4) at (-28,11) {};
			\node[below left] at (-28,11) {$v_4$};
			\draw[fill] (v4) circle (.1);

			\node (v1) at (-28,6.5) {};
			\node[above left] at (-28,6.5) {$v_1$};
			\draw[fill] (v1) circle (.1);

			\node (v0) at (-28,4) {};
			\node[above left] at (-28,4) {$v_0$};
			\draw[fill] (v0) circle (.1);

			\node (v3) at (-25,9.5) {};
			\node[above right] at (-25,9.5) {$v_3$};
			\draw[fill] (v3) circle (.1);

			\node (v2) at (-25,7.5) {};
			\node[below right] at (-25,7.5) {$v_2$};
			\draw[fill] (v2) circle (.1);
			
			\draw[line width =1.25, red, -latex] (v0) -- (v1);
			
			\draw[line width =1.25, red, -latex] (v4) -- (v5);
			
			\draw[line width =1.25, red, -latex] (v2) -- (v3);

			\draw[line width =1, -latex] (v1) -- (v2) node[midway,below]{$b$};
			\draw[line width =1, -latex] (v1) -- (v4) node[midway,left]{$a$};

			\draw[line width =1, -latex] (v3) -- (v4) node[midway,above]{$b$};
			\draw[line width =1, -latex] (v3) -- (-25,15) node[above] {$c$};
			
			\draw[line width =1, -latex] (-25,2.5) node[below] {$c$} -- (v2);
			
			\draw[line width =1, -latex] (-27,2.5) node[below] {$b$} -- (v0);
			\draw[line width =1, -latex] (-29,2.5) node[below] {$a$} -- (v0);

			\draw[line width =1, -latex] (v5) -- (-27,15) node[above] {$b$};
			\draw[line width =1, -latex] (v5) -- (-29,15) node[above] {$a$};
			\end{tikzpicture} & 
			$\ep_{a,b} \ep_{b,c} q^{f(a,b,c)}$ &
			\begin{tikzpicture}[scale=0.3,thick, inner sep=0.4mm]  
			\draw[red, thin] (-29.5, 8.5) -- (-23.5,8.5); 

			\draw[line width =1.5, -latex] (-25,2.5) node[below] {$c$} -- (-25,15) node[above] {$c$};
			
			\draw[line width =1.5, orange, -latex]  (-27.25,2.5) node[below] {$b$} -- (-28.25,4) -- (-28.25,6.5) --  (-25.5,7.5) --  (-25.5,9.5)  -- (-28.25,11) -- (-28.25,13.5) -- (-27.25,15) node[above] {$b$};
			
			\draw[line width =1.5, blue, -latex]  (-28.75,2.5) node[below] {$a$} -- (-27.75,4)  -- (-27.75,13.5) -- (-28.75,15) node[above] {$a$};
			\end{tikzpicture} & 
			\begin{tikzpicture}[scale=0.3,,thick, inner sep=0.4,rotate=180,yscale=-1]  
			\draw[red, thin] (-29.5, 8.5) -- (-23.5,8.5); 
			
			\node (v5) at (-28,13.5) {};
			\node[below right] at (-28,13.5) {$v_5$};
			\draw[fill] (v5) circle (.1);

			\node (v4) at (-28,11) {};
			\node[below right] at (-28,11) {$v_4$};
			\draw[fill] (v4) circle (.1);

			\node (v1) at (-28,6.5) {};
			\node[above right] at (-28,6.5) {$v_1$};
			\draw[fill] (v1) circle (.1);

			\node (v0) at (-28,4) {};
			\node[above right] at (-28,4) {$v_0$};
			\draw[fill] (v0) circle (.1);

			\node (v3) at (-25,9.5) {};
			\node[above left] at (-25,9.5) {$v_3$};
			\draw[fill] (v3) circle (.1);

			\node (v2) at (-25,7.5) {};
			\node[below left] at (-25,7.5) {$v_2$};
			\draw[fill] (v2) circle (.1);
			
			\draw[line width =1.25, red, -latex] (v0) -- (v1);
			
			\draw[line width =1.25, red, -latex] (v4) -- (v5);
			
			\draw[line width =1.25, red, -latex] (v2) -- (v3);

			\draw[line width =1, -latex] (v1) -- (v2) node[midway,below]{$b$};
			\draw[line width =1, -latex] (v1) -- (v4) node[midway,right]{$c$};

			\draw[line width =1, -latex] (v3) -- (v4) node[midway,above]{$b$};
			\draw[line width =1, -latex] (v3) -- (-25,15) node[above] {$a$};
			
			\draw[line width =1, -latex] (-25,2.5) node[below] {$a$} -- (v2);
			
			\draw[line width =1, -latex] (-27,2.5) node[below] {$b$} -- (v0);
			\draw[line width =1, -latex] (-29,2.5) node[below] {$c$} -- (v0);

			\draw[line width =1, -latex] (v5) -- (-27,15) node[above] {$b$};
			\draw[line width =1, -latex] (v5) -- (-29,15) node[above] {$c$};
			\end{tikzpicture} & 
			$\ep_{a,b} \ep_{b,c} q^{g(a,b,c)}$ &
			\begin{tikzpicture}[scale=0.3,,thick, inner sep=0.4,rotate=180,yscale=-1]  
			\draw[red, thin] (-29.5, 8.5) -- (-23.5,8.5);

			\draw[line width =1.5,blue, -latex] (-25,2.5) node[below] {$a$} -- (-25,15) node[above] {$a$};
			
			\draw[line width =1.5, orange, -latex]  (-27.25,2.5) node[below] {$b$} -- (-28.25,4) -- (-28.25,6.5) --  (-25.5,7.5) --  (-25.5,9.5)  -- (-28.25,11) -- (-28.25,13.5) -- (-27.25,15) node[above] {$b$};
			
			\draw[line width =1.5,  -latex]  (-28.75,2.5) node[below] {$c$} -- (-27.75,4)  -- (-27.75,13.5) -- (-28.75,15) node[above] {$c$};
			\end{tikzpicture}
			\\
			\hline
			& \multicolumn{3}{c|}{$f(a,b,c) = 2\sgn(b-a)+\sgn(c-b)$} 
			& \multicolumn{3}{c|}{$g(a,b,c) = \sgn(b-a) + 2\sgn(c-b)$}\\
			\hline
\end{tabular}